\documentclass{amsart}

\usepackage{mathrsfs}
\usepackage[centertags]{amsmath}
\usepackage{amsfonts}
\usepackage{amsthm}
\usepackage{amssymb}
\usepackage{color}
\usepackage[colorinlistoftodos,prependcaption]{todonotes}
\usepackage[colorlinks,linkcolor={blue},citecolor={blue},urlcolor={red}]{hyperref}
\usepackage{enumerate}

\let\emptyset \undefined
\let\ge       \undefined
\let\le       \undefined
\newsymbol\le          1336  \let\leq\le
\newsymbol\ge          133E  \let\geq\ge
\newsymbol\emptyset    203F
\newsymbol\notle       230A
\newsymbol\notge       230B

\allowdisplaybreaks

\newtheorem{theorem}{Theorem}[section]
\theoremstyle{remark}
\newtheorem{remark}[theorem]{Remark}
\newtheorem{example}[theorem]{Example}
\theoremstyle{plain}
\newtheorem{corollary}[theorem]{Corollary}
\newtheorem{lemma}[theorem]{Lemma}
\newtheorem{proposition}[theorem]{Proposition}
\newtheorem{definition}[theorem]{Definition}
\newtheorem{hypothesis}[theorem]{Hypothesis}

\numberwithin{equation}{section}

\def\N{{\mathbb N}}

\def\R{{\mathbb R}}
\def\C{{\mathbb C}}

\newcommand{\Om}{\Omega}
\newcommand{\om}{\omega}
\newcommand{\F}{\mathscr{F}}
\newcommand{\G}{\mathscr{G}}
\renewcommand{\P}{\mathbb{P}}
\newcommand{\E}{\mathbb{E}}

\newcommand{\la}{\lambda}
\newcommand{\eps}{\varepsilon}
\newcommand{\calL}{\mathscr{L}}
\newcommand{\n}{\|}
\newcommand{\one}{{\bf 1}}

\renewcommand{\ss}{^\star}
\newcommand{\lb}{\langle}
\newcommand{\rb}{\rangle}
\newcommand{\limn}{\lim_{n\to\infty}}
\newcommand{\wt}{\widetilde}

\newcommand{\ud}{\,{\rm d}}
\newcommand{\ov}{\overline}
\renewcommand{\H}{\mathscr{H}}
\newcommand{\Dom}{\mathsf{D}}

\newcommand{\bF}{\mathscr{P}}

\allowdisplaybreaks

\begin{document}

\title[Stochastic convolutions]{Maximal inequalities for stochastic convolutions and pathwise uniform convergence of time discretisation schemes}
\author{Jan van Neerven and Mark Veraar}
\address{Delft Institute of Applied Mathematics\\
Delft University of Technology\\
P.O. Box 5031, 2600 GA Delft\\
The Netherlands}
\email{J.M.A.M.vanNeerven/M.C.Veraar@TUDelft.nl}



\thanks{The second named author is supported by VIDI subsidy 639.032.427
of the Netherlands Organisation for Scientific Research (NWO)}

\date\today

\begin{abstract}
We prove a new Burkholder-Rosenthal type inequality for discrete-time processes taking values in a $2$-smooth Banach space. As a first application we prove that if $(S(t,s))_{0\leq s\le t\leq T}$ is a $C_0$-evolution family of contractions on a $2$-smooth Banach space $X$ and $(W_t)_{t\in [0,T]}$ is a cylindrical Brownian motion on a probability space $(\Omega,\mathbb{P})$ adapted to some given filtration, then for every $0<p<\infty$ there exists a constant $C_{p,X}$ such that for all progressively measurable processes $g: [0,T]\times \Om\to X$ the process $(\int_0^t S(t,s)g_s\,{\rm d} W_s)_{t\in [0,T]}$ has a continuous modification and
$$\mathbb{E}\sup_{t\in [0,T]}\Big\n \int_0^t S(t,s)g_s\,{\rm d} W_s \Big\n^p\leq C_{p,X}^p \mathbb{E} \Bigl(\int_0^T \n g_t\n^2_{\gamma(H,X)}\,{\rm d} t\Bigr)^{p/2}.$$
Moreover, for $2\leq p<\infty$ one may take $C_{p,X} = 10 D \sqrt{p},$ where $D$ is the constant in the definition of $2$-smoothness for $X$. The order $O(\sqrt{p})$ coincides with that of Burkholder's inequality and is therefore optimal as $p\to\infty$.

Our result improves and unifies several existing maximal estimates and is even new in case $X$ is a Hilbert space. Similar results are obtained if the driving martingale $g_t\,{\rm d} W_t$ is replaced by more general $X$-valued martingales $\!\,{\rm d} M_t$. Moreover, our methods allow for random evolution systems, a setting which appears to be completely new as far as maximal inequalities are concerned.

As a second application, for a large class of time discretisation schemes (including splitting, implicit Euler, Crank-Nicholson, and other rational schemes) we obtain stability and pathwise uniform convergence of time discretisation schemes for solutions of linear SPDEs
$$ \,{\rm d} u_t = A(t)u_t\,{\rm d} t + g_t\,{\rm d} W_t, \quad u_0 = 0,$$ where the family $(A(t))_{t\in [0,T]}$ is assumed to generate a $C_0$-evolution family $(S(t,s))_{0\leq s\leq t\leq T}$ of contractions on a $2$-smooth Banach spaces $X$. Under spatial smoothness assumptions on the inhomogeneity $g$, contractivity is not needed and explicit decay rates are obtained. In the parabolic setting this sharpens several know estimates in the literature; beyond the parabolic setting this seems to provide the first systematic approach to pathwise uniform convergence to time discretisation schemes.
\end{abstract}

\dedicatory{Dedicated to Lutz Weis on the occasion of his 70th birthday.}

\maketitle

\newpage

\setcounter{tocdepth}{1}
\tableofcontents

\section{Introduction}
In this paper we study maximal inequalities for the mild solutions of
time-dependent stochastic evolution equations of the form
\begin{equation}\label{eq:SEE-intro}
\begin{cases}
\ud u_t &= A(t)u_t + g_t\ud W_t, \quad t\in [0,T], \\ u_0 & = 0.
\end{cases}
\end{equation}
Here, $(A(t))_{t\in [0,T]}$ is a family of closed operators acting in a Banach space $X$ generating a $C_0$-evolution family $(S(t,s))_{0\le s\le t\le T}$,
$(W_t)_{t\in [0,T]}$ is a Brownian motion defined on a probability space $(\Om,\F,\P)$, adapted to some give filtration $(\F_t)_{t\in [0,T]}$, and $(g_t)_{t\in[0,T]}$ is a progressively measurable stochastic process with values in $X$. Under these assumptions the mild solution is given, at least formally, by the $X$-valued stochastic convolution-type integral
\begin{align}\label{eq:u} u_t := \int_0^t S(t,s)g_s\ud W_s, \quad t\in [0,T].
\end{align}
An important special case of \eqref{eq:SEE-intro} is the time-independent case where $A(t)\equiv A$ generates a $C_0$-semigroup $(S(t))_{t\ge 0}$ on $X$ and $S(t,s) = S(t-s)$.
More generally we will consider stochastic convolutions driven by cylindrical Brownian motions and assume that $g$ is operator-valued; this extension is mostly routine and for the ease of presentation will not be considered in this introduction.

In order to give a rigorous meaning to the stochastic integral in \eqref{eq:u} one needs to impose suitable measurability and integrability
assumptions on $g$ and geometrical properties on $X$, such as $2$-smoothness \cite{BD, Brz95, Brz97, Dett89, Dett91, Nei, Ondrejat04, Ond05}
or the UMD property \cite{NVW07a,NVW07b}. The UMD theory is in some sense the definitive theory, in that it features a two-sided Burkholder inequality and completely natural extensions of the martingale representation theorem \cite{NVW07a,NVW07b} and the Clark--Ocone theorem \cite{MaasNee08}; from the point of view of applications to SPDE it provides stochastic maximal $L^p$-regularity for parabolic problems \cite{NVW12b, NVW12a, NVW13, PorVer} which in turn can be used to study quasi- and semi-linear PDEs \cite{AgreVernon}. The $2$-smooth theory only allows for limited versions of these results, but it is easier in its basic constructions and adequate for many other purposes, and will provide the setting for this paper.

Instrumental in proving pathwise continuity of mild solutions to \eqref{eq:SEE-intro} is the availability of suitable estimates for the maximal function
$u^\star: \Om\to [0,\infty)$, $$u^\star := \sup_{t\in[0,T]} \n u_t\n,$$
where $(u_t)_{t\in [0,T]}$ is the process defined by \eqref{eq:u};
norms are taken in $X$ pointwise on $\Omega$. The first such estimate was obtained by Kotelenez \cite{Kot82, Kot84} who showed that if $(A(t))_{t\in [0,T]}$ generates a contractive evolution family $(S(t,s))_{0\le s \le t\le T}$ on a Hilbert space $X$, then the process
$(u_t)_{t\in [0,T]}$ defined by \eqref{eq:u} has a continuous modification which satisfies the maximal inequality
\begin{align}\label{eq:Kot}
\E \sup_{t\in[0,T]} \n u_t\n^2 \le C^2 \E \int_0^T \n g_t\n^2\ud t,
\end{align}
where $C$ is some absolute constant. The extension of \eqref{eq:Kot} to $2$-smooth Banach spaces and general exponents $0<p<\infty$ has been investigated by many authors
\cite{BrzPes, HauSei1, HauSei2,Ichi, NeeZhu, Tub} who all limited themselves to the special case of contraction semigroups. This development is surveyed in \cite{NV20a}, where also some extensions to evolution families are discussed. The more general case of stochastic convolutions driven by L\'evy processes has been studied in the $2$-smooth setting in
\cite{ZBH, ZBW}.

For Brownian motion as the driving process, the best result available to date is due to Zhu and the first author in \cite{NeeZhu}, where it was shown that if $(S(t))_{t\ge 0}$ is a $C_0$-semigroup of contractions on a $2$-smooth Banach space $X$ and $(g_t)_{t\in[0,T]}$ is a progressively measurable process with values in $X$, then the process
$(u_t)_{t\in [0,T]}$ defined by the stochastic convolution
\begin{align*} u_t := \int_0^t S(t-s)g_s\ud W_s, \quad t\in [0,T],
\end{align*}
has a continuous modification which satisfies, for every $0<p<\infty$,
\begin{align}\label{eq:NeeZhu} \E \sup_{t\in[0,T]} \n u_t\n^p \le C_{p,X}^p \E \Bigl(\int_0^T \n g_t\n^2\ud t\Bigr)^{p/2},
\end{align}
where $C_{p,X}$ is a constant depending only on $p$ and $X$.
In certain applications it is important to have explicit information on the constant in the asymptotic regime $p\to\infty$. In the special case $S(t) \equiv  I$ the estimate \eqref{eq:NeeZhu} reduces to the Burkholder inequality for $2$-smooth Banach spaces, for which the asymptotic dependence of  $C_{p,X}$ is known to be of order $O(\sqrt{p})$ as $p\to\infty$ \cite{Sei}. For Hilbert spaces $X$ and $C_0$-semigroup of contractions, \eqref{eq:NeeZhu} is known to hold to order $O(\sqrt{p})$ as $p\to\infty$ \cite{HauSei1,HauSei2}. In that setting the Sz.-Nagy dilation theorem can be used to reduce matters to the
Burkholder inequality. The order $O(\sqrt{p})$ can be used to derive exponential estimates which in turn can be used to study large deviations (see \cite{Chow} and \cite{Pes94}). Inspecting the proof of \eqref{eq:NeeZhu} in \cite{NeeZhu} in the $2$-smooth case, it is seen that the asymptotic $p$-dependence of the constant in that paper is non-optimal.

The aim of the present paper is to
simultaneously improve the results cited above in two directions:
\begin{itemize}
 \item
to extend \eqref{eq:NeeZhu} to arbitrary $C_0$-evolution families of contractions on $2$-smooth Banach spaces $X$ (not even assuming the existence of a generating family $(A(t))_{t\in [0,T]}$);
 \item
to show that the constant $C_{p,X}$ in the resulting maximal inequality is of order $O(\sqrt{p})$ as $p\to\infty$.
\end{itemize}
The precise statement of our main result, which corresponds to the special case of Theorem \ref{thm:contractionS-new} for Brownian motion, is as follows.

\begin{theorem}\label{thm:main}
Let $(S(t,s))_{0\leq s\leq t\leq T}$ be a $C_0$-evolution family of contractions on a $2$-smooth Banach space $X$.
Let $(W_t)_{t\in [0,T]}$ be an adapted Brownian motion on a probability space $(\Om,\P)$,
and let $(g_t)_{t\in[0,T]}$ be a progressively measurable process with values in $X$. Then
the $X$-valued process $(u_t)_{t\in [0,T]}$ defined by
$$ u_t:= \int_0^t S(t,s)g_s\ud W_s, \qquad t\in [0,T],$$
has a continuous modification which satisfies
\[\E\sup_{t\in [0,T]}\n u_t \n^p\leq C_{p,X}^p \E \Bigl(\int_0^T \n g_t\n^2\ud t\Bigr)^{p/2},\]
where the constant $C_{p,X}$ only depends on $p$ and the constant $D$ in the definition of $2$-smoothness for $X$.
For $2\le p<\infty$ the inequality holds with $$C_{p,X} = 10D\sqrt{p}.$$
\end{theorem}

Theorem \ref{thm:contractionS-new} considers the more general situation of a cylindrical Brownian motion with covariance given by the inner product of a Hilbert space $H$ and progressively measurable processes $g$ with values in the space
$\gamma(H,X)$ of $\gamma$-radonifying operators from $H$ to $X$ (the definition of this space is recalled in Section \ref{sec:prelim}).

For evolution families, Theorem \ref{thm:main} is new even for Hilbert spaces $X$. In the $2$-smooth case it completely settles the asymptotic optimality problem; this is new even in the semigroup case.
The proof of the theorem is very different from \cite{HauSei1,HauSei2} and \cite{NeeZhu} and
combines ideas of Kotelenez \cite{Kot82} and Seidler \cite{Sei}. Seidler's proof of the
$O(\sqrt{p})$ bound for the constant in Burkholder inequality in $2$-smooth Banach spaces is based on a clever modification
of the Burkholder--Rosenthal inequality due to Pinelis \cite{Pin}. We further extend Pinelis's inequality by accommodating additional predictable contraction operators in it which enable us to merge the inequality with a splitting technique already used by Kotelenez.

Theorems \ref{thm:main} and \ref{thm:contractionS-new} are also applicable in the setting where the evolution family $S$ itself is not contractive, but admits a dilation to a contractive evolution family on a $2$-smooth Banach space.
In the semigroup case, the boundedness of the $H^\infty$-calculus of the generator $A$ of angle $<\frac12\pi$ implies that the semigroup has a dilation to an isometric $C_0$-group (see \cite{FW,HNVW3,NV20a,Sei,VerWei}). In this case, however, there is no need to use Theorem \ref{thm:main} since one can apply the simpler method of \cite{HauSei1,HauSei2}.

Our method can be used quite naturally to prove the stability (uniformly in time) of certain numerical schemes associated with \eqref{eq:SEE-intro}. This is pursued in Section \ref{sec:numerical}, where we prove that if $(S(t))_{t\geq 0}$ is a $C_0$-semigroup of contractions on a $(2,D)$-smooth Banach space $X$ with generator $A$,
and $u$ is a continuous modification of the process $(\int_0^t S(t-s) g_s\ud W_s)_{t\in [0,T]}$, then for any contractive approximation scheme $R$ which approximates $(S(t))_{t\geq 0}$ to some order $\alpha\in (0,1]$ on the domain $\Dom(A)$ one has
\begin{align}\label{eq:num-intro-1}\E\sup_{j=0,\ldots,n}\|u_{t_j^{(n)}} - u_j^{(n)}\|^p\to 0 \ \ \hbox{as} \ \ n\to\infty,
\end{align}
where
\begin{equation}\label{eq:num-intro}
\begin{cases}
u_0^{(n)} & :=0, \\ u_j^{(n)} &:= R(T/n) \Bigl(u_{j-1}^{(n)} + \displaystyle \int_{t_{j-1}^{(n)}}^{t_{j}^{(n)}} g_s \ud W_s\Bigr), \quad j=1,\dots,n.
\end{cases}
\end{equation}
The crucial observation underlying \eqref{eq:num-intro-1}  is that the sequence $(u_j^{(n)})_{j=0}^n$ defined by \eqref{eq:num-intro} is precisely of the right format to apply our extension of Pinelis's inequality.
For $C_0$-semigroups which are not necessarily contractive and functions $g\in L^p(\Omega;L^2(0,T;\Dom(A))$, we show that convergence holds with the following explicit rate:
\begin{align}\label{eq:convrateintro}
\Bigl(\E\sup_{j=0,\ldots,n}\|u_{t_j^{(n)}} - u_j^{(n)}\|^p\Bigr)^{1/p} \le C \frac{\sqrt{\log (n+1)}}{n^\alpha} \|g\|_{L^p(\Omega;L^2(0,T;\Dom(A)))},
\end{align}
where $C$ is a constant independent of $n$ and $g$. This estimate is somewhat simpler, in that it directly uses Seidler's version of the Burkholder inequality of Proposition \ref{prop:Seid} in combination with a simple trick, in  Proposition \ref{prop:stochasticellinftyn}, involving switching back and forth from $\ell_n^\infty(X)$ to $\ell_n^{q(n)}(X)$  for a clever choice of $q(n)$. This can be done at the expense of a constant $n^{1/q(n)}$, exploiting the fact proven in Proposition \ref{prop:LpX-2mooth} that $\ell^{q(n)}(X)$ is $2$-smooth for $2\le q<\infty$ with constant of order $\sqrt{q(n)}$. This appears to be a new technique whose potential deserves further investigation.

Examples of numerical schemes to which our abstract results can be applied include the splitting method (with $R(t) = S(t)$ with $\alpha=1$), the implicit Euler method (with $R(t) = (I-tA)^{-1}$ and $\alpha=1/2$), and the Crank--Nicholson method (with
$R(tA) = (2+tA)(2-tA)^{-1}$ and $\alpha=2/3$). Moreover, if $g$ takes values in suitable intermediate spaces between $X$ and $\Dom(A^m)$ with $m\geq1$, appropriate rates of convergence can be obtained for each of these methods.

We expect that the new results in the simple linear setting will provide new insights for approximation of nonlinear SPDEs also by adapted time schemes and plan to address this in future work.

To illustrate the main result we consider the stochastic heat equation with the implicit Euler scheme (cf. Example \ref{ex:stochheat}). For simplicity, here we state the result in terms of Sobolev spaces. In Example \ref{ex:stochheat}, the use of Bessel potential spaces allows us to take the smoothness exponent $m$ fractional and also negative.
Further examples can be found in Section \ref{subsec:applischemes}.

\begin{example}[Heat equation, implicit Euler scheme]
Consider the inhomogeneous stochastic heat equation on $\R^d$:
\begin{equation}\label{eq:SEE-heat-intro}
\begin{cases}
\ud u_t &= \Delta u_t + \sum_{k\geq 1} g_{t}^{k}\ud W_t^k, \quad t\in [0,T].
\\ u_0 & = 0.
\end{cases}
\end{equation}
Here, $W=(W^k)_{k\geq 1}$ is a sequence of independent standard Brownian motions.
We further assume that each $g^k:\Omega\times [0,T]\times\R^d\to \R$ is progressively measurable and that
$p\in (0,\infty)$, $q\in [2, \infty)$, and $m\in \N=\{0,1, \ldots\}$ are such that
\[ \|g\|_{\mathbb{W}^{m,q,p}}^p:=\sum_{i=1}^d \sum_{j=0}^{m}\E \Big\{\int_0^T \Big(\int_{\R^d}\Big(\sum_{k\geq 1}|\partial_i^j g_k(t,x)|^2\Big)^{q/2} \ud x\Big)^{2/q}\ud t\Big\}^{p/2}\]
is finite.
For $n=1,2,\dots$ set $t_j^{(n)} := jT/n$ and consider the partition
$\pi^{(n)} := \{t_j^{(n)}: j=0,\ldots, n\}$.
Let $(S(t))_{t\geq 0}$ denote the heat semigroup on $L^q(\R^d)$ and set
$$u_t:= \int_0^t S(t-s) g_s\ud W_s, \quad t\in [0,T].$$
This stochastic integral is well defined as an $L^q(\R^d)$-valued It\^o integral by
Proposition \ref{prop:Seid} and \eqref{eq:gammaLq}.

Define the discrete approximation by $u_0^{(n)}  :=0$, and
\[
u_j^{(n)} := (1-\tfrac{T}{n}\Delta)^{-1} \Big(u_{j-1}^{(n)} + \int_{t_{j-1}^{(n)}}^{t_{j}^{(n)}} g_s \ud W_s\Big), \quad j=1,\dots,n,
\]
Let $W^{j,q}(\R^d)$ be the Sobolev space of smoothness $j$ and integrability $q$. Then the following results hold:
\begin{align*}
\E\sup_{j=0,\ldots,n}\|u_{t_j^{(n)}} - u_j^{(n)}\|^p_{W^{m-2,q}(\R^d)} \le \Bigl( C_{p,q,d,m} \frac{\sqrt{\log (n+1)}}{n}\Bigr)^p \|g\|_{\mathbb{W}^{m,q,p}}^p, \ & m\geq 2,
\\
\E\sup_{j=0,\ldots,n}\|u_{t_j^{(n)}} - u_j^{(n)}\|^p_{W^{m-1,q}(\R^d)}  \le \Bigl(C_{p,q,d,m} \frac{\sqrt{\log (n+1)}}{n^{1/2}}\Bigr)^p \|g\|_{\mathbb{W}^{m,q,p}}^p,\  & m\geq 1,
\\
\limn \E\sup_{j=0,\ldots,n}\|u_{t_j^{(n)}} - u_j^{(n)}\|^p_{W^{m,q}(\R^d)}  \to 0 \quad \qquad \qquad\qquad \qquad \qquad\qquad\qquad & m\geq 0.
\end{align*}
This follows from Theorems \ref{thm:generalapprox1} and \ref{thm:generalapprox2}.
\end{example}

In the final Section \ref{sec:random} we extend some of results to stochastic convolutions involving random evolution families, which arise naturally if the operator family $(A(t))_{t\in [0,T]}$ in \eqref{eq:SEE-intro} depends on a random parameter in an adapted way. That this is possible at all in the abstract setting of evolution equations in infinite dimensions is quite remarkable. It requires replacing the It\^o integral with the forward integral of \cite{RussoVallois} in order to avoid adaptedness problems. Stochastic convolution in the forward sense is known to still give the weak solution to \eqref{eq:SEE-intro} (see \cite[Proposition 5.3]{LeonNual}, \cite[Theorem 4.9]{ProVer14} and Theorem \ref{thm:contractionS-adaptedSDE} below). In the parabolic setting, space-time regularity results have been derived by Pronk and the second-named author in \cite{ProVer14} using so-called pathwise mild solutions (see Proposition \ref{prop:forwardequiv}) and a simple integration by parts trick. Pathwise mild solutions have been recently used to study quasilinear PDEs in \cite{fernando2015stochastic,KuhnNe20,mohan2017stochastic}
and random attractors in \cite{kuehn2020random}. The new maximal estimates proved in our current paper are expected to have implications for these results as well.

For adapted families $(A(t))_{t\in [0,T]}$,  maximal inequalities can be alternatively derived via It\^o's formula (see \cite{NV20a} and references therein). In contrast to the results obtained here, however, this does not lead to constants of order $O(\sqrt{p})$ as $p\to \infty$. In the setting of monotone (possible nonlinear) operators and $p=2$, the It\^o formula argument is applicable in a wider setting (see \cite{LiuRoc}). Some extensions to $p>2$  have been obtained recently in \cite{NeeSis}.

\section{Preliminaries}\label{sec:prelim}

Throughout this paper we work over the real scalar field. Unless otherwise stated, random variables and stochastic processes
are defined on a probability space $(\Om,\F,\P)$ which we consider to be fixed throughout.
On this probability space we fix a filtration $(\F_t)_{t\in [0,T]}$ once and for all. Standard notions from the theory of stochastic processes always refer to this filtration. Whenever we consider stochastic integrals with respect to a (cylindrical) Brownian motion or a more general type of driving process, it is always assumed that it is adapted with respect to this filtration.
The conditional expectation of a random variable $\xi$ with respect to a sub-$\sigma$-algebra $\G\subseteq \F$ will be denoted by $\E_\G(\xi)$.
The progressive $\sigma$-algebra associated with $(\F_t)_{t\in [0,T]}$, i.e., the $\sigma$-algebra generated by sets of the form $B\times A$ with $B\in \mathscr{B}([0,t])$ and $A\in \F_t$, where $t$ ranges over $[0,T]$, is denoted by $\mathscr{P}$.
We will use the subscript $\bF$ to denote the closed subspace of all progressively measurable process in a given space of processes.

When $X$ is a Banach space, under an {\em $X$-valued random variable} we understand a strongly measurable function (i.e., a function which is the pointwise limit of a sequence of simple functions) from $\Omega$ into $X$; for details the reader is referred to \cite{HNVW16, HNVW17}. For the purposes of this article, an {\em $X$-valued process} is a family of $X$-valued random variables indexed by $[0,T]$. Two processes  $(g_t)_{t\in [0,T]}$ and $(h_t)_{t\in [0,T]}$ are said to be {\em modifications} of each other if for al $t\in [0,T]$ we have $g_t = h_t$ almost surely (with exceptional set that may depend on $t$).
A process $(g_t)_{t\in [0,T]}$ with values in $X$ is said to be {\em progressively measurable} if $g$ is strongly measurable
as an $X$-valued function on the measurable space $([0,T]\times \Om,\mathscr{P})$.
It is a deep result in the theory of stochastic processes that every adapted and strongly measurable $X$-valued stochastic process admits a progressively measurable modification; an elementary proof  is offered in \cite{OndSei}.

\subsection{2-Smooth Banach spaces}\label{subsec:2smooth}

A Banach space $X$ is said to have \emph{martingale type} $p\in [1,2]$\index{martingale type $p$} if there exists a constant $C\ge 1$ such that
\begin{equation*}%\label{eq:def-marttype}
  \E \n f_N\n {p}
  \leq C^p \Bigl(\E \n f_0\n ^p+\sum_{n=1}^N\n f_n - f_{n-1}\n^p\Bigl)
\end{equation*}
for all $X$-valued $L^p$-martingales $(f_n)_{n=0}^N$.
A Banach space $X$ is called {\em $(p,D)$-smooth}, where $p\in [1,2]$ and $D\ge 1$, if for all $x,y\in X$ we have
\begin{align*}
	 \n  x+y\n^p+\n  x-y\n^p\leq 2\n  x\n^p+2D^p\n  y\n^p.
\end{align*}
A Banach space is called {\em $p$-smooth} if it is $(p,D)$-smooth for some $D\ge 1$.

By a fundamental result due to Pisier \cite{Pis} every $p$-smooth Banach space has martingale type $p$, and conversely every Banach space with martingale type $p$ admits an equivalent $p$-smooth norm. Moreover,
if $X$ has martingale type $p$ with constant $C$, an equivalent $(p,C)$-smooth norm can be
found;
if $X$ is $(p,D)$-smooth, then $X$ has martingale type $p$ with constant at most $2C$ (and the constant $2$ can be omitted for $p=2$, see Remark \ref{rem:Rad}). Detailed proofs of these facts can be found in \cite{Pis-Mart,Wen, Woy}.

The class of $2$-smooth Banach space is of particular interest from the point of view of stochastic analysis. It includes all Hilbert spaces (with $D=1$, by the parallelogram identity) and
the spaces $L^p(\mu)$ with $2\le p<\infty$ (with $D=\sqrt{p-1}$, see \cite[Proposition 2.1]{Pin} and Proposition \ref{prop:LpX-2mooth} below). The reason for being interested in $2$-smooth spaces rather than spaces with martingale type $2$ is as follows.  Martingale type $2$ is preserved under passing to equivalent norms, but this is not the case for $2$-smoothness. In the results to follow, semigroups and evolution families of {\em contractions} (i.e., operators of norm $\le 1$) play a distinguished role. Since contractivity need not be preserved under passing to equivalent norms, such a distinguished role cannot be expected in the setting of martingale type $2$ spaces. In this connection the following interesting question seems to be an open: if $X$ has martingale type $2$ and supports a $C_0$-semigroup (or $C_0$-evolution family), does there exist an equivalent $(2,D)$-smooth norm with respect to which the semigroup (or evolution family) is contractive?

In what follows we recall some useful properties of $2$-smooth Banach spaces that will be needed in this paper.

If $X$ is $(2,D)$-smooth, then by \cite[Lemma 2.1]{NeeZhu} and its proof the function
\[\rho(x):= \n x\n^2\]
is Fr\'echet differentiable on $X$ and its derivative is Lipschitz continuous. Conversely, if $\rho$ is twice Fr\'echet differentiable and $\rho''(x)(y,y)\leq 2D^2\|y\|^2$ at every $x\in X$, then $X$ is $(2,D)$-smooth (see \cite{Pin} for a more general version of this converse). Unlike in finite dimensions, Lipschitz continuity does not imply almost everywhere differentiability (the latter even being meaningless in the absence of a reference measure). One way to get around this is to consider the functions
\[\rho_{x,y}(t):= \rho(x+ty) = \n x+ ty\n^2. \]
The following lemma is implicit in \cite{Pin}. For the reader's convenience we include a proof.

\begin{proposition}\label{prop:mian-p2} For any Banach space $X$ and constant $D\ge 1$ the following assertions are equivalent:
\begin{enumerate}[\rm(1)]
 \item\label{it:mian-p2-1} $X$ is $(2,D)$-smooth;
 \item\label{it:mian-p2-2} for all $x,y\in X$ the function
 $\rho_{x,y}(t):= \rho(x+ty) = \n x+ty\n^2$ is differentiable on $\R$, its derivative is Lipschitz continuous, and satisfies
 $$ \rho_{x,y}'(t) - \rho_{x,y}'(s) \le 2D^2(t-s)\n y\n^2, \quad s,t\in \R, \ t\ge s.$$
\end{enumerate}
\end{proposition}
\begin{proof}
\eqref{it:mian-p2-1}$\Rightarrow$\eqref{it:mian-p2-2}: \
Fix $x,y\in X$. The differentiability of $\rho_{x,y}(t)= \n x+ty\n^2$ follows from the Fr\'echet differentiability of $\rho$, and by the chain rule we have $\rho_{x,y}'(t) = \lb y,\rho'(x+ty)\rb$. Lipschitz continuity of $\rho'$ follows from \cite[Lemma V.3.5]{Deville} and implies the Lipschitz continuity of $\rho_{x,y}'$.
It follows that the second derivative $\rho_{x,y}''(t)$ exists for almost every $t\in\R$, and in the points where it exists it is given by
\begin{align*} \rho_{x,y}''(t)
 & = \lim_{h\to 0} \frac1{h^2}((\rho_{x,y}(t+h) + \rho_{x,y}(t-h) - 2\rho_{x,y}(t))
 \\ & = \lim_{h\to 0}  \frac1{h^2}( \n (x + ty)+hy\n^2 + \n (x + ty)-hy\n^2 - 2\n x+ty\n^2).
 \end{align*}
Therefore, by $2$-smoothness, $ \rho_{x,y}''(t) \le 2D^2 \n y\n^2$ in these points. This implies that $\rho_{x,y}'(t) - \rho_{x,y}'(s) \le 2D^2(t-s)\n y\n^2$ for all $t\ge s$.

\smallskip
\eqref{it:mian-p2-2}$\Rightarrow$\eqref{it:mian-p2-1}: \
For all $x,y\in X$ we have
\begin{align*}
\n x + y\n^2 + \n x -y\n^2 - 2\n x\n^2 & = \int_0^1 \rho_{x,y}'(t)\ud t - \int_{-1}^0 \rho_{x,y}'(t)\ud t
\\ & = \int_0^1 \rho_{x,y}'(t)- \rho_{x,y}'(t-1)\ud t \le 2D^2\n y\n^2.
\end{align*}
\end{proof}

As an application we prove the following vector-valued analogue of \cite[Proposition 2.1]{Pin}. It will be needed in the proof of Proposition \ref{prop:stochasticellinftyn}, which in turn is applied in Section \ref{sec:numerical}.

\begin{proposition}\label{prop:LpX-2mooth}
Let $(S,\mathscr{A},\mu)$ be a measure space and $X$ be a $(2,D)$-smooth Banach space. Then for all $2\le p<\infty$ the space $L^p(S;X)$ is $(2,\sqrt{p-2+D^2})$-smooth.
\end{proposition}

Notice that $D\ge 1$ implies $p-2+D^2\le D^2(p-1)$, so in particular $L^p(S;X)$ is $(2,D\sqrt{p-1})$-smooth.

\begin{proof}
The proof is based on the equivalence in Proposition \ref{prop:mian-p2}.
For Banach spaces $X$ with the property that
 $\|\cdot\|^2$ is twice continuously Fr\'echet differentiable the proof can be somewhat simplified.

Throughout the proof we use $\n\cdot\n$ and $\n \cdot\n_p$ to denote the norms of $X$ and $L^p(S;X)$, respectively. Thus if $f\in L^p(S;X)$, then $\n f\n$ is the function $s\mapsto \n f(s)\n$ in $L^p(S)$.

 As in \cite[Theorem V.1.1]{Deville} one checks that the functions
 $$\psi_p(x):= \n x\n^p, \quad \Psi_{p}(g):= \n g\n_p^p,$$
 are Fr\'echet differentiable and
 \begin{align}\label{eq:Psipprime}
 \lb f,\Psi_{p}'(g)\rb = \int_S \lb f,\psi_p'(g)\rb\ud \mu, \qquad f,g\in L^p(S;X),
 \end{align}
 where the duality $\lb\cdot,\cdot\rb$ between $X$ and $X^*$ is applied pointwise on $S$.
For $q\in \R$ let
 \begin{align*}
w_{q;x,y}(t) & := \n x+ty\n^q, \qquad x,y\in X; \\
W_{q;f,g}(t) & := \n f+tg\n_p^q, \qquad f,g\in L^p(S;X).
 \end{align*}
The Fr\'echet differentiability of $\psi_p$ and $\Psi_p$ implies the differentiability of $w_{q;x,y}$ and $W_{q;f,g}$ (except possibly at $t=0$ when $x=0$ and $y\not=0$, respectively $f=0$ and $g\not=0$).
Denoting derivatives with respect to $t$ by $\partial_t$,
for $q\not=0$ the chain rule gives
\begin{align*}
\partial_t w_{q;x,y}(t) & = \frac{q}{2} \n x+ty\n^{q-2}\partial_t w_{2;x,y}(t) = q \n x+ty\n^{q-2}\lb y,\psi_{2}'(x+ty)\rb \\
\partial_t W_{q;f,g}(t) & = \frac{q}{2} \n f+tg\n_p^{q-2}\partial_t W_{2;f,g}(t) = q \n f+tg\n_p^{q-2}\lb g,\Psi_{2}'(f+tg)\rb,
\end{align*}
where $\Psi_{2}(g):= \n g\n_p^2$.
Also, $$\psi_p'(x) = \frac{p}{2}\n x\n^{p-2}\psi_2'(x), \quad \Psi_p'(f) = \frac{p}{2}\n f\n_p^{p-2}\Psi_2'(f).$$
Combining these identities with \eqref{eq:Psipprime}, we obtain
\begin{equation}\label{eq:partialW2fg}
\begin{aligned}
\frac12\partial_t W_{2;f,g}(t)
& = \lb g,\Psi_2'(f+tg)\rb
\\ & = \frac2p \n f+tg\n_p^{2-p} \lb g,\Psi_p'(f+tg)\rb
\\ & = \frac2p \n f+tg\n_p^{2-p} \int_S  \lb g, \psi_p'(f+tg)\rb\ud \mu
\\ & = \n f+tg\n_p^{2-p} \int_S \n f+tg\n^{p-2} \lb g, \psi_2'(f+tg)\rb\ud \mu
\\ & = \frac1p \n f+tg\n_p^{2-p} \int_S \partial_t w_{p;f,g}(t)\ud \mu.
\end{aligned}
\end{equation}
Since $X$ is $2$-smooth and Lipschitz functions are almost everywhere differentiable, for all $x,y\in X$ the function $w_{2;x,y}$ is twice differentiable almost everywhere by Proposition \ref{prop:mian-p2}. The exceptional set may depend on the pair $(x,y)$, however, so in order to be able to differentiate the right-hand side of \eqref{eq:partialW2fg} under the integral we will consider {\em simple} functions $f,g\in L^p(S;X)$ from this point onward. Then the right-hand side of \eqref{eq:partialW2fg}
is differentiable for almost all $t\in\R$ and
\begin{align*}
 \ & \partial_t^2 W_{2;f,g}(t)
 \\ & \qquad = \frac2p \partial_t (\n f+tg\n_p^{2-p}) \int_S \partial_t w_{p;f,g}(t)\ud \mu +
 \frac2p  \n f+tg\n_p^{2-p} \partial_t \int_S \partial_t w_{p;f,g}(t)\ud \mu
\\ &  \qquad = \frac2p \partial_t ((W_{2;f,g}(t))^{1-\frac{p}{2}}) \int_S \partial_t w_{p;f,g}(t)\ud \mu +
 \frac2p \n f+tg\n_p^{2-p}\!  \int_S \partial_t^2 w_{p;f,g}(t)\ud \mu
 \\ & \qquad  = \Big(\frac{2}{p}-1\Big)\n f+tg\n_p^{-p}\partial_t W_{2;f,g}(t)\int_S \partial_t w_{p;f,g}(t)\ud \mu
\\ &  \qquad \qquad  + \frac2p\n f+tg\n_p^{2-p}\int_S \partial_t^2 w_{p;f,g}(t)\ud \mu
 \\ &  \qquad \stackrel{(*)}{=} \frac2p\Big(\frac{1}{p}-\frac12\Big)\n f+tg\n_p^{2-2p}\Bigl(\int_S \partial_t w_{p;f,g}(t)\ud \mu\Bigr)^2
\\ &  \qquad \qquad  + \frac2p\n f+tg\n_p^{2-p}\int_S \partial_t^2 w_{p;f,g}(t)\ud \mu
\\ &
 \qquad  \stackrel{(**)}{\le}  \frac2p\n f+tg\n_p^{2-p}\int_S \partial_t^2 w_{p;f,g}(t)\ud \mu,
\end{align*}
where $(*)$ follows from \eqref{eq:partialW2fg} and $(**)$ from the assumption $2\le p<\infty$.
Now
\begin{align*}
 \partial_t^2 w_{p;x,y}(t)
 & = p \partial_t [\n x+ty\n^{p-2} \lb y, \psi_2'(x+ty)\rb]
\\ & =  p\partial_t(\n x+ty\n^{p-2})\lb y, \psi_2'(x+ty)\rb + p\n x+ty\n^{p-2} \partial_t \lb y, \psi_2'(x+ty)\rb
\\ & = 2p\Big(\frac{p}{2}-1\Big)\n x+ty\n^{p-4}\lb y, x+ty\rb^2 + \frac{p}{2}\n x+ty\n^{p-2} \partial_t^2 w_{2;x,y}(t)
\\ & \le 2p\Big(\frac{p}{2}-1\Big)\n x+ty\n^{p-2}\n y\n^2 + pD^2\n x+ty\n^{p-2} \n y\n^2.
\end{align*}
Applying this with $x = f(\cdot)$ and $y = g(\cdot)$ we obtain
\begin{align*}
\partial_t^2 W_{2;f,g}(t)
 & \le 2(p-2 +D^2) \n f+tg\n_p^{2-p}\int_S  \n f+tg\n^{p-2} \n g\n^2\ud \mu.
\end{align*}
By H\"older's inequality with $r = p/(p-2)$ and $r' = p/2$ we obtain that $W_{2;f,g}$ is twice differentiable almost everywhere and
\begin{align}\label{eq:partial2W2fg} \partial_t^2 W_{2;f,g}(t) \le 2(p-2+D^2)\n g\n_p^2.
\end{align}
Since $f$ and $g$ are simple, the $2$-smoothness of $X$ and Proposition \ref{prop:mian-p2} imply that $t\mapsto \partial_t W_{2;f,g}$ is Lipschitz continuous.
Therefore it follows from \eqref{eq:partial2W2fg} that $t\mapsto \partial_t W_{2;f,g}$ is Lipschitz continuous with Lipschitz constant $2(D^2+p-2)\n g\n_p^2$. The proof of the implication
\eqref{it:mian-p2-2}$\Rightarrow$\eqref{it:mian-p2-1} of Proposition \ref{prop:mian-p2} then gives the inequality
\begin{align*}
	 \n  f+g\n^2+\n  f-g\n^2\leq 2\n  f\n^p+ 2(D^2+p-2)\n  g\n^2
\end{align*}
for simple $f,g\in L^p(S;X)$. The inequality for general  $f,g\in L^p(S;X)$ follows by approximation.
\end{proof}

\begin{remark}
By Pisier's characterisation of $2$-smoothness in terms of the modulus of uniform smoothness \cite{Pis},
the fact that $2$-smoothness of $X$ implies the  $2$-smoothness of $L^p(\mu;X)$ for all $2\le p<\infty$
follows from \cite{Figiel76}. A quantitative version is proved in \cite[Corollary 2.3]{Naor}
where it is shown that if the modulus of uniform smoothness of a Banach space satisfies  $ \varrho_X(\tau) \le s \tau^2$ for all $\tau>0$, then the modulus of uniform smoothness of $L^p(\mu;X)$ satisfies
\begin{align}\label{eq:Naor}
\varrho_{L^p(\mu,X)}(\tau)\le (4s+4p)\tau^2, \quad \tau>0.
\end{align}
By Pisier's result, this implies that  $L^p(\mu;X)$ is $(2,E)$-smooth for some $E\ge 1$, but the bound for $E$ obtained this way is worse than ours. We will show this by demonstrating that our Proposition \ref{prop:LpX-2mooth} gives a slight improvement of the constant \eqref{eq:Naor}.
Indeed, by \cite[Proposition 3.1.2]{Woy}, the bound $\varrho_X(\tau) \le s \tau^2$ for $\tau>0$
implies that $X$ is $(2,\sqrt{1+4s})$-smooth.
Consequently Proposition \ref{prop:LpX-2mooth} implies that
$L^p(\mu;X)$ is $(2,\sqrt{p-1+4s})$-smooth. Another application of \cite[Proposition 3.1.2]{Woy}
then gives that
\begin{align*}
\varrho_{L^p(\mu;X)}(\tau) \le (4s+p-1)\tau^2, \quad \tau>0.
\end{align*}
\end{remark}

Following \cite{Pin} we will use Proposition \ref{prop:mian-p2}
to derive some further useful inequalities for the function
$$w (t):= w_{x,y}(t):= (\rho(x+ty))^{1/2} = \n x+ ty\n, $$
where $x$ and $y$ are fixed elements in a $(2,D)$-smooth Banach space.
Evidently  $w$ is Lipschitz continuous with $|w(t) - w(s)|
\le |t-s|\n y\n $, so $w$ is almost everywhere differentiable with
\begin{align}\label{eq:up} |w'(t)| \le \n y\n.
\end{align}

We start from the elementary observation that $\sinh a \le a\cosh a$ for $a\ge 0$.
Hence when $w''w \ge 0$, Proposition \ref{prop:mian-p2} implies the almost everywhere inequalities
\begin{align*} (\cosh w)'' & = (w')^2\cosh w + w''\sinh w \\ & \le
((w')^2 + w''w)\cosh w = \tfrac12 (w^2)''\cosh w \le D^2 \n y\n^2\cosh w,
\end{align*}
whereas if $w''w<0$, then \eqref{eq:up} implies
$$ (\cosh w)'' = (w')^2\cosh w + w''\sinh w \le (w')^2\cosh w \le \n y\n^2\cosh w.$$
Combining these inequalities we obtain the almost everywhere inequality
\begin{align}\label{eq:coshpp}
(\cosh w)'' \le D^2 \n y\n^2\cosh w.
\end{align}

The next lemma was obtained in \cite[Proposition 2.5 and the proof of Theorem 3.2]{Pin}. We present a more direct argument which avoids the smoothing procedure and reduction to the finite dimensional setting used in \cite[Lemma 2.2, Lemma 2.3, and Remark 2.4]{Pin}.

\begin{lemma}\label{lem:folklorePinelis}
Let $X$ be a $(2,D)$-smooth Banach space
and let $\xi,\eta\in L^2(\Omega;X)$. Let $\G\subseteq \F$ be a sub-$\sigma$-algebra. If $\xi$ is strongly $\G$-measurable and $\E_{\G}{\eta}= 0$, then
\begin{align*}
\E_{\mathscr{G}}(\|\xi+\eta\|^2) &\leq \|\xi\|^2 + D^2 \E_{\mathscr{G}}(\|\eta\|^2).
\end{align*}
If, moreover, $\xi,\eta\in L^\infty(\Omega;X)$, then
\[\E_{\mathscr{G}}(\cosh(\|\xi+\eta\|))\leq \big(1+D^2 \E_{\mathscr{G}}(e^{\|\eta\|}-1 - \|\eta\|)\big) \cosh(\|\xi\|).\]
\end{lemma}
\begin{proof}
Fix $x,y\in X$. As before we let $\rho_{x,y}(t): =  (w_{x,y}(t))^2= \|x+ty\|^2 = \rho(x+ty)$ for $t\in \R$.
Then $\rho_{x,y}$ is continuously differentiable and $\rho_{x,y}'$ is Lipschitz continuous with constant  $2D^2\|y\|^2$ by Proposition \ref{prop:mian-p2}.
Taylor's formula then gives
\begin{align*}\|x+ty\|^2 = \rho_{x,y}(t) & = \rho_{x,y}(0) + t \rho_{x,y}'(0) + \int_0^t \rho_{x,y}'(s)-\rho_{x,y}'(0) \ud s
\\ & \leq \|x\|^2 + t \lb \rho'(x), y\rb+ t D^2\|y\|^2.
\end{align*}
Setting $x= \xi(\om)$, $y=\eta(\om)$, $t=1$, and taking conditional expectations, we obtain
\[\E_{\mathscr{G}}(\|\xi+\eta\|^2)\leq \|\xi\|^2 + \E_{\mathscr{G}}(\lb \rho'(\xi), \eta\rb)+ D^2\E_{\mathscr{G}}(\|\eta\|^2).\]
It remains to note that
$\E_{\mathscr{G}}\lb \rho'(\xi), \eta\rb = \lb \rho'(\xi), \E_{\mathscr{G}}\eta\rb = 0$.

For the second assertion note that the function $\zeta:\C\to \C$ defined by $\zeta(z) := \cosh(z^{1/2})$, is entire. Let $u(t) : = \cosh(\|x+ty\|) = \cosh(w(t))= \zeta(\rho_{x,y}(t))$. By \eqref{eq:coshpp},
\begin{align*}
 u(t) & = u(0) + t u'(0)
 + \int_0^t (t-s) u''(s) \ud s
\\ & \leq \cosh(\|x\|) + t \rho_{x,y}'(0)\zeta'(\|x\|^2) \lb \rho'(x), y\rb
+ D^2 \n y\n^2\!\!\int_0^t (t-s)\cosh(\|x+sy\|)\ud s.
\end{align*}
Since $\cosh(\|x+sy\|) \leq \cosh(\|x\|+s\|y\|)\leq e^{s\|y\|}\cosh(\|x\|)$ for $s\ge 0$, the integral on the right-hand side satisfies
\begin{align*}
 \n y\n^2\int_0^t (t-s) \cosh(\|x+sy\|) \ud s& \leq  \cosh(\|x\|) \n y\n^2\int_0^t (t-s) e^{s\|y\|}\ud s
\\ & = \cosh(\|x\|) (e^{t\|y\|} - 1 - t\|y\|).
\end{align*}
Combining the estimates with $x= \xi(\om)$, $y=\eta(\om)$, $t=1$, and taking conditional expectations, we obtain
\begin{align*}
\ & \E_{\mathscr{G}}(\cosh(\|\xi+\eta\|))
\\ & \quad \leq \cosh(\|\xi\|)
+\rho_{x,y}'(0) \zeta'(\|\xi\|^2) \E_{\mathscr{G}}(\lb \rho'(\xi), \eta\rb)
+ D^2 \E_{\mathscr{G}}(e^{\|\eta\|}-1 - \|\eta\|) \cosh(\|\xi\|).
\end{align*}
The result follows from this by using once more that $\E_{\mathscr{G}}(\lb \rho'(\xi), \eta\rb) = 0$.
\end{proof}

\begin{remark}\label{rem:Rad}
Applying the first part of this lemma iteratively to Rademacher sums, we obtain the folklore result that $(2,D)$-smoothness implies martingale type $2$ with constant $D$.
\end{remark}

\subsection{Stochastic integration in $2$-smooth Banach spaces}\label{subsec:SI}

Let $\H$ a Hilbert space. An {\em $\H$-isonormal process} on $\Om$ is a
mapping $\mathscr{W}: \H\to L^2(\Om)$ with the following two properties:
\begin{enumerate}
\item[\rm(i)] For all $h\in \H$ the random variable $\mathscr{W}h$ is Gaussian;
\item[\rm(ii)] For all $h_1,h_2\in \H$ we have
$\E (\mathscr{W}h_1 \cdot\mathscr{W}h_2) = (h_1|h_2)$.
\end{enumerate}

It is easy to see that every $\H$-isonormal process is linear and that
for all $h_1,\dots,h_N\in \H$ the $\R^N$-valued random variable
$(\mathscr{W}h_1,\dots, \mathscr{W}h_N)$ is jointly Gaussian.
For more details the reader is referred to \cite{HNVW4, Nua}.

If $H$ is another Hilbert space, an {\em $H$-cylindrical Brownian motion} indexed by $[0,T]$ is an isonormal process $W: L^2(0,T;H) \to L^2(\Om)$.
Following common practice we write $$W_t h:= W(\one_{(0,t)}\otimes h), \qquad t\in [0,T], \ h\in H.$$
For each $h\in H$, the scalar-valued process
$Wh= (W_th)_{t\in [0,T]}$ is then a Brownian motion, which is standard if and only if $h$ has norm one. Two such Brownian motions $Wh_1$ and $Wh_2$ are independent if and only if $h_1$ and $h_2$ are orthogonal in $H$.
We say that $W$ is {\em adapted} to the
filtration $(\F_t)_{t\in [0,T]}$ on $\Om$ if $W(f\otimes h)\in L^2(\Om,\F_t)$ for all $f\in L^2(0,T)$ supported in $(0,t)$ and all $h\in H$. In what follows we always assume that $H$-cylindrical Brownian motions are adapted to $(\F_t)_{t\in [0,T]}$.

The space of finite rank operators from a Hilbert space $H$ into a Banach space $X$ is denoted by $H\otimes X$.
Every finite rank operator $T\in H\otimes X$ can be represented in the form $T = \sum_{n=1}^N h_n\otimes x_n$ with $(h_n)_{n=1}^N$ orthonormal in $H$ and
$(x_n)_{n=1}^N$ a sequence in $X$. We then define
\begin{align}\label{eq:def-radonif} \n T\n_{\gamma(H,X)}^2 = \E \Big\n \sum_{n=1}^N \gamma_n x_n\Big\n^2,
\end{align}
where $(\gamma_n)_{n=1}^N$ is a sequence of independent standard Gaussian random variables. It is an easy consequence of the preservation of joint Gaussianity under orthogonal transformations that the norm $\n \cdot\n_{\gamma(H,X)}$ is well defined.
The completion of $H\otimes X$ with respect to this norm is denoted by $\gamma(H,X)$. The natural inclusion mapping from $H\otimes X$ into $\calL(H,X)$ extends to a contractive inclusion mapping $\gamma(H,X)\subseteq \calL(H,X)$.
A linear operator $T\in\calL(H, X)$ is said to be {\em $\gamma$-radonifying} if it belongs to $\gamma(H,X)$. For $1\le p<\infty$ the Kahane--Khintchine inequalities guarantee that replacing $L^2$-norms by $L^p$-norms in \eqref{eq:def-radonif} gives an equivalent norm on $\gamma(H,X)$. The space $\gamma(H,X)$, when endowed with this equivalent norm, will be denoted by $\gamma_p(H,X)$.

For Hilbert spaces $K$ we have $$\gamma(H,K) =\calL_2(H,K)$$ isometrically, where $\calL_2(H,K)$ is the space of Hilbert--Schmidt
operators from $H$ to $K$. For $1\le p<\infty$ and any Banach space $X$ the identity mapping on $H\otimes L^p(\mu)$
extends to an isometric isomorphism of Banach spaces
\begin{align}\label{eq:gammaLq} \gamma_p(H, L^p(\mu)) \simeq L^p(\mu;H).
\end{align}
For $H = L^2(\nu)$ this identifies $\gamma(L^2(\nu), L^p(\mu))$ with the space
$L^p(\mu;L^2(\nu))$ of `square functions' using terminology from harmonic analysis. For more details the reader is referred to \cite[Chapter 9]{HNVW17}.

A stochastic process $\Phi :[0,T]\times \Om\to \calL(H,X)$ is called an {\em adapted finite rank step process} if there exist $0=s_0<s_1<\ldots<s_n=T$, random variables $\xi_{ij}\in L^\infty(\Omega,\F_{s_{j-1}})\otimes X$ (the subspace of $L^\infty(\Omega;X)$ of strongly $\F_{s_{j-1}}$-measurable random variables taking values in a finite-dimensional subspace of $X$)
for $i=1, \ldots, m$ and $j=1, \ldots, n$, and an orthonormal system $h_1,\dots, h_m$ in $H$ such that
\begin{equation}\label{eq:simple}
\Phi  = \sum_{j=1}^{n} \one_{(s_{j-1},s_{j}]} \sum_{i=1}^m h_i\otimes \xi_{ij}.
\end{equation}
For such processes the stochastic integral with respect to the $H$-cylindrical Brownian motion $W$
is defined by
\[\int_0^t \Phi_s \ud W_s := \sum_{j=1}^n  \sum_{i=1}^m  (W_{s_{j}\wedge t}-W_{s_{j-1}\wedge t})h_i \otimes \xi_{ij}, \  \ t\in [0,T].\]
Since $t\mapsto W_th$, being a Brownian motion, has a continuous modification, it follows that $t\mapsto \int_0^t\Phi_s \ud W_s$ has a continuous modification. Such modifications will always be used in the sequel.
It was shown by Neidhardt in his PhD thesis \cite{Nei} (see also \cite{Dett89}, \cite{NVW13})
that if $\Phi$ is an adapted finite rank step process, then
\begin{align}\label{eq:Neid}
\E \Big\n \int_0^T \Phi_t \ud W_t\Big\n^2\le D^2 \|\Phi\|_{L^2(\Omega;L^2(0,T;\gamma(H,X)))}^2.
\end{align}
By \eqref{eq:Neid}, standard localisation arguments, and Doob's inequality, the stochastic integral can be extended to arbitrary progressively measurable processes $\Phi: [0,T]\times \Om\to\gamma(H,X)$ for which the $L^2(0,T;\gamma(H,X))$-norm is finite almost surely and the resulting stochastic integral process $(\int_0^t \Phi_s \ud W_s)_{t\in [0,T]}$
has a continuous modification. At this juncture it is useful to observe that a process $\Phi: [0,T]\times \Om\to\gamma(H,X)$
is progressively measurable (as a process with values in the Banach space $\gamma(H,X)$) if and only if
$\Phi h: [0,T]\times \Om\to X$ is progressively measurable (as a process with values in $X$) for all $h\in H$; this follows from
\cite[Example 9.1.16]{HNVW17}.

The following version of the classical Burkholder inequality is the result of contributions of many authors \cite{BD, Brz97, Brz3, Dett89, Dett91, Ondrejat04}.

\begin{proposition}\label{prop:Seid}
 Let $X$ be a $(2,D)$-smooth Banach space, let $W$ be an adapted $H$-cylindrical Brownian motion on $\Om$, and let $0<p<\infty$. For all adapted finite rank step process $\Phi: [0,T]\times \Om\to\gamma(H,X)$ we have
$$
\E \sup_{t\in [0,T]}\Big\n \int_0^t \Phi_s \ud W_s\Big\n^p\le C_{p,D}^p \|\Phi\|_{L^p(\Omega;L^2(0,T;\gamma(H,X)))}^p,
$$
where $C_{p,D}$ is a constant depending only on $p$ and $D$.
\end{proposition}
By using Pinelis's version of the Burkholder--Rosenthal inequalities \cite{Pin}, Seidler \cite{Sei} has shown that the constant $C_{p,D}$ has the same asymptotic behaviour for $p\to\infty$ as in the scalar-valued setting, i.e.,
$$C_{p,D} = C_D O(\sqrt{p}) \ \ \hbox{ as $p\to\infty$}.$$
As a special case of our main result we will recover Seidler's result, with $C_{p,D} = 10 D \sqrt{p}$ if $2\le p<\infty$, by setting $S(t,s) \equiv I$ in Theorem \ref{thm:contractionS-new}.

As a consequence of Proposition \ref{prop:Seid} we obtain the following result, which will be useful in the error analysis of numerical schemes for SPDEs in Section \ref{sec:numerical}.

\begin{proposition}\label{prop:stochasticellinftyn}
Let $X$ be a $(2,D)$-smooth Banach space and let $0< p< \infty$. Let $\Phi := (\Phi^{(k)})_{k=1}^n$ be a finite sequence in $ L^p_{\bF}(\Omega;L^2(0,T;\gamma(H,X)))$ and set
\[I^{\Phi}_n := \Big(\E \sup_{t\in [0,T], k\in \{1,\ldots,n\}}\Big\|\int_0^t \Phi_s^{(k)} \ud W_s\Big\|^p\Big)^{1/p}.\]
Then
\begin{align}
I^{\Phi}_n & \leq C_{p,D} \sqrt{\log n} \|\Phi\|_{L^p(\Omega;L^2(0,T;\gamma(H,\ell^\infty_n(X))))} & \text{if $n\geq 3$,} \label{eq:stochasticellinftyn-1}
\\ I^{\Phi}_n & \leq K_{p,D} \log n \|\Phi\|_{L^p(\Omega;L^2(0,T;\ell^\infty_n(\gamma(H,X))))} & \text{if $n\geq 8$.}\label{eq:stochasticellinftyn-2}
\end{align}
If $2\le p<\infty$, these estimates holds with $C_{p,D} = 10D\sqrt{2e p}$ and $K_{p,D} = 10D e \sqrt{p}$.
\end{proposition}
The bound \eqref{eq:stochasticellinftyn-2} is simpler to use, but \eqref{eq:stochasticellinftyn-1} will give a better result in the applications later on.
\begin{proof}
The method of proof is inspired by \cite{DGVW}.
The idea is to view the sequence $\Phi = (\Phi^{(k)})_{k=1}^n$ as an $\ell^q_n(X)$-valued process
for a clever choice of $q= q(n) \in [2,\infty)$.

We begin with the proof of \eqref{eq:stochasticellinftyn-1}.
Since $\ell^q_n(X)$ is $(2,D\sqrt{q})$-smooth by Proposition \ref{prop:LpX-2mooth}, by Proposition \ref{prop:Seid} we have
\begin{align*}
I^{\Phi}_n \leq
\Big(\E \sup_{t\in [0,T]}\Big\|\int_0^t \Phi_s \ud W_s\Big\|^p_{\ell^q_n(X)} \Big)^{1/p}
 & \leq C_{p,q,D}  \|\Phi\|_{L^p(\Omega;L^2(0,T;\gamma(H,\ell^q_n(X))))}
\\ & \leq C_{p,q,D} n^{1/q} \|\Phi\|_{L^p(\Omega;L^2(0,T;\gamma(H,\ell^\infty_n(X))))},
\end{align*}
and if $2\le p<\infty$ we may take $C_{p,q,D} = 10D\sqrt{pq}$.
The estimate \eqref{eq:stochasticellinftyn-1} follows from this by taking $q=2\log n$, which belongs to the interval $[2, \infty)$ if $n\ge 3$.

To prove \eqref{eq:stochasticellinftyn-2} we argue in the same way, but this time we use that for a sequence $\Gamma:=(\Gamma_k)_{k=1}^n$ with $\Gamma_k\in \gamma(H,X)$,
\begin{align*}
\|\Gamma\|_{\gamma(H,\ell^q_n(X))}& \leq \|\Gamma\|_{\gamma_q(H,\ell^q_n(X))}
\\ & = \|\Gamma\|_{\ell^q_n(\gamma_q(H,X))}\leq n^{1/q} \|\Gamma\|_{\ell^\infty_n(\gamma_q(H,X))}\leq n^{1/q} \sqrt{q} \|\Gamma\|_{\ell^\infty_n(\gamma(H,X))},
\end{align*}
applying the Kahane--Khintchine inequalities (see \cite[Theorem 6.2.6]{HNVW17}) in the last step. Now \eqref{eq:stochasticellinftyn-2} follows by taking $q= \log n$.
\end{proof}

\begin{remark} The same method of proof can be used to show that if $X$ is $(2,D)$-smooth, then $\ell^\infty_n(X)$ has martingale type $2$ with constant $\sqrt{D^2 -2 + 2\log n}$ if $n\geq 3$.
\end{remark}

\section{Extending Pinelis's Burkholder--Rosenthal inequality}\label{sec:PBR}

On the probability space  $(\Om,\F,\P)$ we consider a finite filtration $(\F_j)_{j= 0}^k$ and denote by $\E_j:= \E_{\F_j}$ the
conditional expectation with respect to $\F_{j}$.
When $ (f_j)_{j= 0}^k$ is an $X$-valued martingale with respect to $(\F_j)_{j= 0}^k$, we denote by
$(df_j)_{j=1}^k$ its difference sequence, i.e.,
$df_j := f_j - f_{j-1}$.
We further define the non-negative random variables $f_j\ss$ (for $0\le j\le k$) and $df_j\ss$ and $s_j(f)$ (for $1\le j\le k)$ by
\begin{align*}
 f_j\ss  := \max_{0\le i\le j} \n f_i\n , \quad df_j\ss :=  \max_{1\le i\le j} \n df_i\n,  \quad  s_j(f):= \Bigl(\sum_{i=1}^j \E_{i-1} \n df_{i}\n^2\Bigr)^{1/2},
\end{align*}
and we set $f\ss := f\ss_k$, $df\ss := df\ss_k$, and $s(f) := s_k(f)$.

If $\G$ is a sub-$\sigma$-algebra of $\F$, we call the $X$-valued random variables $\xi$ and $\eta$
{\em conditionally equi-distributed given $\G$} if for all Borel sets $B\subseteq X$ we have
$$ \E_\G \one_{\{\xi \in B\}} = \E_\G\one_{\{\eta \in B\}}.$$
As in \cite[Lemma 4.4.5]{HNVW16} one sees that this is equivalent to the requirement that
\begin{align}\label{eq:cond-equi}
\E(f(\xi)|\G) = \E(f(\eta)|\G)
\end{align}
for all measurable functions $f:X \to X$ such that $f(\xi), f(\eta)\in L^1(\Omega;X)$.

An adapted $X$-valued sequence $(\xi_j)_{j=1}^k$ is called {\em conditionally symmetric given $(\F_j)_{j=0}^k$}
if for all Borel sets $B\subseteq X$ and $1\le j\le k$ the random variables
 $\xi_j$ and $-\xi_j$ are conditionally equi-distributed given $\F_{j-1}$.
 Taking $f(x) = \one_{\{\|x\|\leq r\}}x$ in \eqref{eq:cond-equi}, it follows that for conditionally symmetric sequences we have
$ \E_{j-1} (\one_{\{\n\xi_j\n \leq  r\}}\xi_j) = - \E_{j-1} (\one_{\{\n\xi_j\n \leq  r\}}\xi_j) $, i.e.,
\begin{align}\label{eq:cond-symm} \E_{j-1} (\one_{\{\n\xi_j\n \leq r\}}\xi_j) = 0.
\end{align}

A {\em random operator} on $X$ is a mapping $V:\Om\to \calL(X)$ such that $\om\mapsto V(\om)x$ is strongly measurable for all $x\in X$, and a {\em random contraction} on $X$ is a {\em random operator} on $X$ whose range consists of contractions.

The main result of this section is the following extension of Pinelis's version of the Rosenthal--Burkholder inequality \cite{Pin}. Recently, other extensions of some of Pinelis's estimates for $p$-smooth Banach spaces have been obtained in \cite{Luo}.

\begin{theorem}\label{thm:Pinelis}
Let $X$ be a $(2,D)$-smooth Banach space. Suppose that $ (f_j)_{j= 0}^k$ is an adapted sequence of $X$-valued random variables, $ (g_j)_{j=0}^k$ is an $X$-valued martingale,
$(V_j)_{j=1}^k$ is a sequence of random contractions on $X$ which is strongly predictable (i.e., each $V_jx$ is strongly $\F_{j-1}$ measurable for all $x\in X$), and assume that we have $f_0=g_0=0$ and
\[f_j = V_{j} f_{j-1} + dg_j, \qquad j=1, \ldots, k.\]
Then for all $2\le p<\infty$ we have
\[ \n f^\star\n_p \le 30p \n dg\ss\n_p + 40D\sqrt{p} \n s(g)\n_p.\]
If, moreover, $(g_j)_{j=0}^k$ has conditionally symmetric increments, then
\[ \n f^\star\n_p \le 5p \n dg\ss\n_p + 10D\sqrt{p} \n s(g)\n_p.\]
\end{theorem}

Here and in the rest of the paper, $\n \cdot\n_p$ is the norm of $L^p(\Om)$.
The proof of Theorem \ref{thm:Pinelis} closely follows that of \cite[Theorem 4.1]{Pin}
(which, up to the value of the constants, corresponds to taking $V_j = I$ and $g_j = f_j$).
We point out that even in the case $p=2$, Theorem \ref{thm:Pinelis} is not obvious because the additional predictable sequence $(V_j)_{j=1}^{k}$ destroys the martingale structure of $f$.

The proof in \cite{Pin} is written up rather concisely and therefore we shall present the proof of Theorem \ref{thm:Pinelis} in full detail. At the same time this provides the opportunity to give more precise information on the constants.

We need some auxiliary results, the first of which is a classical `good $\la$' inequality (see \cite[Lemma 7.1]{Burk}).

\begin{lemma}\label{lem:burk}
 Suppose that $g$ and $h$ are non-negative random variables and suppose that $\beta > 1$, $\delta> 0$, and $\eps > 0$ are such that for all $\la>0$ we have
\begin{align*} \P(g > \beta\la,\, h < \delta\la) < \eps \P(g > \la).
\end{align*}
If $1\le p<\infty$ and $\beta^p\eps <1$, then
\begin{align*} \E g^p \le \frac{(\beta/\delta)^p}{1-\beta^p \eps} \E h^p.
\end{align*}
\end{lemma}

The next lemma is a minor extension of \cite[Theorem 3.4]{Pin}.

\begin{lemma}\label{lem:tail}
Suppose that $(g_j)_{j=0}^k$ is a martingale with values in a $(2,D)$-smooth Banach space $X$ with $g_0=0$ and let $(h_j)_{j=0}^{k-1}$ be an adapted sequence of random variables with values in $X$. Set $$f_0:=0, \qquad f_j := h_{j-1} + dg_j, \quad 1\le j\le k,$$
and assume that $\|h_{j}\|\leq \|f_{j}\|$ almost surely for all $0\le j\le k-1$. Suppose further that
$\n dg\ss\n_\infty \le a$ and $\n s(g)\n_\infty\le b/D$ for some $a>0$ and $b>0$. Then for all $r> 0$ we have
$$ \P(f\ss \ge r) \le 2\Bigl(\frac{eb^2}{ra}\Bigr)^{r/a}.$$
\end{lemma}
\begin{proof}
We begin by noting that the almost sure conditions $f_0 = 0$, $\n h_{j-1}\n \le \n f_{j-1}\n$, $f_j := dg_j + h_{j-1}$,
and $dg_j\le a$ imply that the random variables $h_{j-1}$ and $f_j$, $j=1,\dots,k$, are essentially bounded and $h_0=0$ almost surely.

Fix $\lambda>0$ and $1\le j\le k$. By Lemma \ref{lem:folklorePinelis},
\begin{align*}
\E_{j-1}\cosh(\la \n f_j\n) &=  \E_{j-1}\cosh(\la \n h_{j-1} + dg_j\n)
\\ & \leq \Bigl(1+ D^2 \E_{j-1} (e^{\la \n dg_j\n} -1 -\la \n dg_j\n) \Bigr)\cosh(\la \n h_{j-1}\n)
\\ & \leq \Bigl(1+ D^2 \E_{j-1} (e^{\la \n dg_j\n} -1 -\la \n dg_j\n) \Bigr)\cosh(\la \n f_{j-1}\n)
\\ & =:(1+e_j)\cosh(\la \n f_{j-1}\n).
\end{align*}
Note that the random variables $e_j$ are non-negative.
This means that the sequence $(G_j)_{j=0}^k$ defined by
$$G_0=1, \quad G_j:= \Bigl(\prod_{i=1}^j(1+e_i)\Bigr)^{-1}\cosh(\la \n f_{j}\n), \quad j=1,\dots,k,
$$
is a positive supermartingale.
Fix $r>0$ and set $\tau:= \min\{1\le j\le k: \, \n f_j\n\ge r\}$ on the set $\{f\ss \ge r\} = \{\max_{1\le j\le k} \n f_j\n \ge r\}$ and $\tau := \infty$ on its complement.
By the optional sampling theorem, the sequence $(G_{\tau\wedge j})_{j=0}^k$ is a positive supermartingale.
It follows that
$ \E \one_{\{\tau\le k\}}G_\tau \le \E G_{\tau\wedge k} \le \E G_0 = 1.$
Therefore, by the inequality $\cosh u > \frac12e^u$ and Chebyshev's inequality,
\begin{align*}
\P(f\ss \ge r) =  \P(\tau\le k)
& =\P\Bigl(\tau\le k,\, G_\tau \ge  \Bigl\n\prod_{j=1}^k (1+e_j)\Bigr\n_\infty^{-1}\cosh(\la r)\Bigr)
\\ & \le \P\Bigl(\tau\le k,\,G_\tau \ge  \frac12\Bigl\n\prod_{j=1}^k (1+e_j)\Bigr\n_\infty^{-1}e^{\la r}\Bigr)
\\ & \le 2 \exp(-\la r) \Big\n\prod_{j=1}^k (1+e_j)\Bigr\n_\infty \E\one_{\{\tau\le k\}}G_\tau
\\ & \le 2 \exp(-\la r) \Big\n\prod_{j=1}^k (1+e_j)\Bigr\n_\infty \le 2 \exp\Bigl(-\la r + \Bigl\n \sum_{j=1}^k e_j\Bigr\n_\infty\Bigr),
\end{align*}
the last inequality being elementary.

 The function defined by $\psi(0) := \frac12$ and $\psi(u) := (e^u -1-u)/u^2$ for $u\not=0$
 is increasing, and therefore for all $\la>0$ we have
\begin{align*} \E_{j-1} (e^{\la \n dg_j\n} -1 - \la \n dg_j\n) \le \frac{1}{a^2} (e^{\la a} -1 - \la a)\E_{j-1}\n dg_j\n^2 .
\end{align*}
Combining this with the definition of the random variables
$e_j$ and the assumption $\n s(g)\n_\infty\le b/D$, we obtain the pointwise inequalities
\begin{align*} \sum_{i=1}^k e_i
& = D^2\sum_{i=1}^k  \E_{j-1} (e^{\la \n dg_j\n} -1 -\la \n dg_j\n)
\\ & \le  \frac{D^2}{a^2} (e^{\la a} -1 - \la a)
\sum_{i=1}^k \E_{j-1}\n dg_j\n^2
\le \frac{b^2}{ a^2}(e^{\la a} -1 - \la a).
\end{align*}
Taking the supremum norm and substituting the result into above tail estimate for $f^\star$
we arrive at
\begin{align*}
 \P(f\ss \ge r)
 & \le 2 \exp\Bigl(-\la r + \frac{b^2}{ a^2}(e^{\la a} -1 - \la a)\Bigr).
\end{align*}
Up to this point the choice of $\la>0$ was arbitrary.
Optimising the choice of $\la>0$ leads to the estimate
\begin{align*}
 \P(f\ss \ge r)
 & \le 2 \exp\Bigl(\frac{r}{a} - \Bigl(\frac{r}{a}+ \frac{b^2}{a^2}\Bigr)\ln \Bigl(1+\frac{ra}{b^2}\Bigr)\Bigr)
\end{align*}
which, by elementary estimates, implies the inequality in the statement of the lemma.
\end{proof}

The next lemma gives a sufficient condition in order that Lemma \ref{lem:burk} can be applied
and extends \cite[Lemma 4.2]{Pin}. Terminology is as in Theorem \ref{thm:Pinelis}.

\begin{lemma}\label{lem:checkburk}
Let $X$ be a $(2,D)$-smooth Banach space $X$.
Suppose that $(g_j)_{j=0}^k$ is a martingale with values in $X$ with $g_0=0$ such that each $dg_j$ is $\F_{j-1}$-conditionally symmetric,
the sequence of random operators $(V_j)_{j=1}^k$ on $X$ is strongly predictable and contractive. Let $(f_j)_{j=0}^k$ be the sequence of random variables defined by
$$ f_0 := 0, \qquad f_j := V_{j} f_{j-1} +  dg_j, \quad j= 1, \ldots, k.$$
Then for all $\la,\delta_1, \delta_2>0$ and $\beta> 1+\delta_2$ we have
$$ \P( f\ss > \beta\la, \, w \le \la) \le \eps \P(f\ss>\la),$$
where
$$ w = (\delta_2^{-1}dg\ss) \vee (\delta_1^{-1}Ds(g)), \quad
\eps = 2\Bigl(\frac{e\delta_1^2}{N\delta_2^2}\Bigr)^N, \quad N = \frac{\beta-1-\delta_2}{\delta_2}.$$
\end{lemma}
\begin{proof} Fix $\la,\delta_1, \delta_2>0$ and $\beta> 1+\delta_2$.
Setting $\ov g_0:= 0$ and
\[\ov g_j  := \sum_{i=1}^j \one_{\{\n dg_i\n \le \delta_2 \la\}} dg_i,\quad j=1,\dots,k,\]
by \eqref{eq:cond-symm} we have $\E_{j-1}d\ov g_j = 0$. Set $\overline{f_0}:= f_0=0$, $h_0:=0$, and
$$\overline{f}_j := V_{j}\overline{f}_{j-1} + d\overline{g}_j , \quad h_j :=  V_{j} h_{j-1}+\one_{\{\mu<j\le \tau\wedge \nu\}}d\overline{g}_j, \quad j=1,\dots,k, $$
where the stopping times $\mu$, $\nu$, and $\tau$ are defined by
 \begin{align*}
 \mu  & := \inf\{0\le j\le k:\, \n \ov f_j\n >\la\}, \\
 \nu  & := \inf\{0\le j\le k:\, \n \ov f_j\n >\beta\la\}, \\
 \tau & := \inf\{0\le j\le k-1:\, s_{j+1}(\ov g) > \delta_1 D^{-1} \la\};
 \end{align*}
we set $\mu := \infty$,  $\nu:=\infty$, and $\tau:=\infty$ if the respective sets over which the infima are taken are empty.
Note that the sequence $(h_j)_{j=0}^k$ is adapted. Notice that $h_j = 0$ on the set $\{j\le \mu\}$; in particular $h_\mu = 0$.

On the set $\{w \le \la\}$ we have $dg\ss \le \delta_2\la$ and in particular $\n dg_i\n \le \delta_2\la$ and therefore $dg_i =d\ov g_i$ for all $i=0,\dots,k$, so $\ov f_j = f_j$ for all $j=0\dots k$. It follows that
$$ \P( f\ss > \beta\la, \, w \le \la) = \P(\ov f\ss > \beta\la, \, w \le \la).$$
It also follows that $s(\ov g) \le \delta_1D^{-1}\la$, so $\tau = \infty$.

On the set $\{\ov f\ss > \beta\la\}$ we have $\mu\le \nu\le k$,
$\n \ov f_{\mu-1}\n \le \la$, and $\n \ov f_{\nu}\n > \beta\la$.
Consequently, for any contraction $S$ on $X$, on the set  $ \{\ov f\ss > \beta\la, \,w \le \la\}$ we have
\begin{align*}
\| \ov f_{\nu} - S \ov f_{\mu}\|\geq \|\ov f_{\nu}\| - \|\ov f_{\mu}\|\geq \|\ov f_{\nu}\| - \|\ov f_{\mu-1}\| -\|d\ov{g}_{\mu}\| > \beta\la - \la - \delta_2\lambda.
\end{align*}
On this set we also have
\begin{equation} \label{eq:identityhnumuf}
\begin{aligned}
h_{\nu} &= 0   \ \ \hbox{if $\mu=\nu$}, \\
h_{\nu} &= \ov f_{\nu} - V_{\nu,\mu}\ov f_{\mu} \ \ \hbox{if $\mu<\nu$, where} \ \ V_{\nu,\mu}= V_{\nu}\circ\ldots\circ V_{\mu+1}.
\end{aligned}
\end{equation}
The first identity in \eqref{eq:identityhnumuf} follows from
$ h_\mu  = V_{\mu}h_{\mu-1}= \dots = V_{\mu}\circ\dots\circ V_1 h_0 = 0$, recalling that $h_0=0$.
The second identity follows from $h_\mu=0$ and induction pointwise on $\Om$, noting that if $\mu\leq n < n+1\le \nu$,
then
\begin{align*}
h_{n+1} = V_{n+1} h_{n}+d\overline{g}_{n+1} = V_{n+1} (\ov f_{n} - V_{n,\mu}\ov f_{\mu})+d\overline{g}_{n+1} = \ov f_{n+1} - V_{n+1,\mu} \ov f_{\mu},
\end{align*}
where we used the definitions of $h$ and $f$, the linearity of $V_{n+1}$, and the induction hypothesis. Therefore, on the set $ \{\ov f\ss > \beta\la, \,w \le \la\}$, we obtain
\begin{align*}
h\ss  \ge \n h_{\nu}\n
& = \n \ov f_{\nu} - V_{\nu,\mu} \ov f_{\mu}\n > (\beta-1-\delta_2)\la.
\end{align*}
We have shown that
$$\P(\ov f\ss > \beta\la, \, w \le \la)  = \P(\ov f\ss > \beta\la, \, w \le \la) \le \P(h\ss > (\beta-1-\delta_2)\la).$$

Let $0\le n\le k$ be such that $\P(\Omega_n)>0$, with $\Omega_n := \{\mu=n\}$. We claim that (a) the random variables $\one_{\{\mu<j\le \tau\wedge \nu\}}d\overline{g}_j$ form a martingale difference sequence on the probability space $(\Omega_n, \F|_{\Omega_n}, \P_n)$, where $\F|_{\Omega_n} := \{F\cap \Omega_n:\, F\in \mathscr{F}\}$ and $\P_n:=\P/\P(\Omega_n)$, and (b) for this martingale difference sequence the conditions of Lemma \ref{lem:tail} are satisfied on the probability space $\Omega_n$, with $f_j$, $g_j$, and $h_j$ replaced by the restrictions  to $\Om_n$ of $h_j$,  $\gamma_j:= \one_{\{\mu<j\le \tau\wedge \nu\}}d\overline{g}_j$, and $ V_{j+1} h_j$ respectively, and with $a = \delta_2\la$, and $b = \delta_1\la$.

Indeed, fix $1\le j\le k$. If $j\le n$, then $j\le \mu$ on $\Om_n$ and therefore $\gamma_j = \one_{\{\mu<j\le \tau\wedge \nu\}}d\overline{g}_j =0$ on $\Om_n$.
If $j>n$, then $\{\mu<j\le \tau\wedge \nu\}\cap \Omega_n = \{j\le \tau\wedge \nu\}\cap \Omega_n$ is $\F_{j-1}$-measurable as a subset of $\Om$ and $\F_{j-1}|_{\Omega_n}$-measurable as a subset of $\Om_n$ and consequently for all $F\in \F_{j-1}|_{\Omega_n}\subseteq \F_{j-1}$ we obtain
\[\int_{F} \gamma_j \ud \P_n = \frac{1}{\P(\Omega_n)}\int_{F\cap \{j\le \tau\wedge \nu\}} d\ov g_j  \ud \P = 0\]
since $\E_{j-1}d\ov{g}_j=0$. This proves part (a) of the claim.

Turning to part (b) of the claim, the condition $d{\gamma}\ss \leq \delta_2\lambda$ of Lemma \ref{lem:tail} is immediate from the definition,  and the adaptedness of $V_{j+1} f_{j}$ as well as the pointwise inequalities $\|V_{j+1} f_{j}\| \leq \|f_{j}\|$ are also clear. The pointwise inequality $s({\gamma})\leq \delta_1D^{-1}\lambda$ on $\Om_n = \{\mu = n\}$ follows from
\begin{align*}
s(\gamma) & = \Big(\sum_{j=1}^{k} \E_{j-1}^{n}(\one_{\{\mu<j\leq \tau\wedge \nu\}} \|d\ov g_j\|^2)\Big)^{1/2}
\\ & = \Big(\sum_{j=n+1}^{k} \E_{j-1}^{n}(\one_{\{j\leq \tau\wedge \nu\}} \|d\ov g_j\|^2)\Big)^{1/2}
\\ & \stackrel{(*)}{=} \Big(\sum_{j=n+1}^{k} \E_{j-1}(\one_{\{j\leq \tau\wedge \nu\}} \|d\ov g_j\|^2)\Big)^{1/2}
\\ & = \Big(\sum_{j=n+1}^{\tau\wedge \nu\wedge k} \E_{j-1}(\|d\ov g_j\|^2)\Big)^{1/2}\leq s_{\tau\wedge k}(\ov g)\leq \delta_1D^{-1}\lambda,
 \end{align*}
where $(*)$ follows from the $\F_{j-1}|_{\Om_n}$-measurability of $\{j\leq \tau\wedge \nu\}\cap \Om_n$ for $j>n$ and the last step uses the definition of $\tau$.

Putting together the various inequalities and applying Lemma \ref{lem:tail} on the space $\Omega_n$ as indicated above, taking $r = (\beta-1-\delta_2)\lambda$, and using that $h^\star =0$ on $\{\mu = \infty\}$, by definition of $\eps$ we arrive at
\begin{align*}
\P( f\ss > \beta\la, \, w \le \la) &\le \P(h\ss > (\beta-1-\delta_2)\la)
\\ & = \sum_{n\geq 0} \P(\mu = n) \P_n(h\ss > (\beta-1-\delta_2)\la)
\\ & \le \sum_{n\geq 0} \P(\mu = n) \cdot 2 \Bigl(\frac{e (\delta_1\la)^2}{(\beta-1-\delta_2)\la(\delta_2\la)}\Bigr)^{(\beta-1-\delta_2)/\delta_2}
\\ & = \sum_{n\geq 0} \P(\mu = n) \cdot 2 \Bigl(\frac{e\delta_1^2}{N\delta_2^2}\Bigr)^{N}
= \eps \P(\mu<\infty) = \eps\P(f\ss >\la).
\end{align*}
\end{proof}

\begin{proof}[Proof of Theorem \ref{thm:Pinelis}]

{\em Step 1.} \ We first consider the conditional symmetric case.
Combining Lemmas \ref{lem:burk} (with $g = f\ss$ and $h = w$) and \ref{lem:checkburk} (with the choice of $\eps$, $N$ and $w$ made there) we arrive at the estimate
\begin{align*}
 \n f\ss\n_p & \le \frac{\beta}{(1-\beta^p\eps)^{1/p}}
 \n (\delta_2^{-1}dg\ss) \vee (\delta_1^{-1}D s(g))\n_p,
\end{align*}
valid for all choices of $\la>0$, $\delta_1, \delta_2>0$, $\beta> 1+\delta_2$
satisfying $ \beta^p \eps<1$.

With the choices
$$  \delta_1 := \frac1{4\sqrt{p}}, \quad \delta_2 := \frac1{2p}, \quad \beta := 2 + \delta_2 = 2+\frac1{2p}$$
we have
$N  = \frac{\beta-1-\delta_2}{\delta_2} =2p\geq 4$,
$\beta^{p/N} = (2+\frac1{2p})^{1/2} \le (\frac94)^{1/2} = \frac{3}{2}$ and $\varepsilon = 2(e/8)^N$, so
\begin{align*}
(\beta^p\eps)^{1/N} = \beta^{p/N} \cdot 2^{1/N} \frac{e}{8} \leq \frac32\cdot 2^{1/4}\cdot \frac{e}{8}=:\theta \approx 0.60611\hdots <1,
\end{align*}
so Lemma \ref{lem:burk} can be applied with these choices. This gives $1-\beta^p \eps \ge 1-\theta^N \ge 1-\theta^4
\approx 0.8650\hdots\,$, so $\beta/(1-\beta^p \eps)^{1/p} \le \frac94\cdot (1-\theta^4)^{1/2} \approx 2.0926\hdots$
and consequently
$$
\n f\ss\n_p \le \frac{\beta}{(1-\beta^p\eps)^{1/p}}\Bigl(
2p\n dg\ss\n_p + 4\sqrt{p} D\n s(g)\n_p\Bigr)\leq 5p
\n dg\ss\n_p + 10 D\sqrt{p} \n s(g)\n_p.
$$
This completes the proof in the conditional symmetric case.

\smallskip
{\em Step 2}. \ The general case will be reduced to the conditional symmetric case. This is a variation of a standard symmetrisation argument (cf.\ the proof of \cite[Theorem 4.1]{Hitc}). In view of the rather intricate setting and in order to obtain explicit constants, we present some details.

Using the terminology of \cite[Chapter 6]{GinePena}, let $(d\wt{g}_j)_{j=0}^k$ be the decoupled tangent sequence of $(dg_j)_{j=0}^k$ on a possibly enlarged probability space. There exists a $\sigma$-algebra $\mathscr{G}$ such that the sequence $(d\wt{g}_j)_{j=0}^k$ is $\mathscr{G}$-conditionally independent and such that
\[\P(d\wt{g}_j\in \cdot|\mathscr{G}) =\P(d\wt{g}_j\in \cdot|\F_{j-1}) = \P(dg_j\in \cdot|\F_{j-1}).\]
Moreover we may assume that $\mathscr{G}  = \mathscr{F}_k$, trivially extending the latter $\sigma$-algebra to the larger probability space (see \cite[p. 294]{GinePena}).
Let $\wt{f}_0 := f_0=0$ and
$\wt{f}_j := V_j \wt{f}_{j-1} + d\wt{g}_j$. Setting $F_j := f_j-\wt{f}_j$ and $G_j := g_j - \wt{g}_j$, we have
$F_0 = 0$ and $F_j = V_j F_{j-1} + dG_j$. The differences $dG_j$ are conditionally symmetric. Therefore, by the symmetric case of Theorem \ref{thm:Pinelis},
\begin{align*}
\n f^\star\n_p \leq \n F^\star\n_p + \n \wt{f}^\star\n_p
\leq 5p \n dG\ss\n_p + 10D\sqrt{p} \n s(G)\n_p + \n \wt{f}^\star\n_p.
\end{align*}
We estimate each of the three terms on the right-hand side.

As in \cite[Lemma 1 and p. 227]{Hit88},
\[\|dG^\star\|_p\leq \|dg^\star\|_p + \|d\wt{g}^\star\|_p\leq 3\|dg^\star\|_p.\]
To estimate $s(G)$ we note that $s(G)\leq s(g) + s(\wt{g}) = 2s(g)$, where we used that $\E_{j-1} \|d\wt{g}_j\|^2 =\E_{j-1} \|d{g}_j\|^2$ (see \cite[Lemma  4.4.5]{HNVW16}). Thus
\begin{align}\label{eq:splittingGg}
5p \n dG\ss\n_p + 10D\sqrt{p} \n s(G)\n_p\leq 15p \n dg\ss\n_p + 20D\sqrt{p} \n s(g)\n_p
\end{align}
To estimate $\n \wt{f}^\star\n_p$, let $(d\overline{g}_j)_{j=1}^k$ be yet another decoupled tangent sequence of $(dg_j)_{j=1}^k$ on a further enlarged probability space. This sequence can be chosen in such a way that $(d\overline{g}_j)_{j=1}^k$ and $(d\wt{g}_j)_{j=1}^k$ are ${\mathscr{G}}$-conditionally independent with $\mathscr{G}$ as before.
Let $\overline{f}_0 := f_0=0$ and
$\overline{f}_j := V_j \overline{f}_{j-1} + d\overline{g}_j$. Then also
$(\overline{f}_j)_{j=0}^k$ and $(\wt{f}_j)_{j=0}^k$ are ${\mathscr{G}}$-conditionally independent.
Therefore, by Jensen's inequality and the fact that $\E_{\mathscr{G}} \overline{f}_j = 0$ (which follows by induction using $\E_{\mathscr{G}}  d\overline{g}_j =0$),
\begin{align*}
\E_{{\mathscr{G}}}\|\wt{f}^\star\|^p  = \E_{{\mathscr{G}}}\|(\wt{f}_j)_{j=0}^k\|_{\ell^\infty_k(X)}^p \leq \E_{{\mathscr{G}}}\|(\wt{f}_j)_{j=0}^k - (\overline{f}_j)_{j=0}^k\|_{\ell^\infty_k(X)}^p = \E_{{\mathscr{G}}}|\overline{F}^\star|^p,
\end{align*}
where $\overline{F}_j = \wt{f}_j - \overline{f}_j$ and $\overline{G}_j = \wt{g}_j - \overline{g}_j$. Then $F_0 = 0$ and $\overline{F}_j = V_j \overline{F}_{j-1} + d\overline{G}_j$. As before, $(\overline{G}_j)_{j=1}^n$ is conditionally symmetric and therefore, by the symmetric case of Theorem \ref{thm:Pinelis},
\begin{align*}
\n \wt{f}^\star\n_p \leq \|\overline{F}^\star\|_p &\leq 5p \n d\overline{G}\ss\n_p + 10D\sqrt{p} \n s(\overline{G})\n_p
 \leq 15p \n dg\ss\n_p + 20D\sqrt{p} \n s(g)\n_p,
\end{align*}
where the last step is the same as \eqref{eq:splittingGg}.

The desired inequality is obtained by combining all estimates.
\end{proof}

\begin{remark}
choices of the parameters $\beta$, $\delta_1$ and $\delta_2$ lead to related inequalities, with a different behaviour of the constants in $p$. In particular, as in \cite[Theorem 4.1]{Pin} one can prove that there exists a constant $C$ such that for all $p\in [2, \infty)$
\[ \n f^\star\n_p \le \frac{C p}{\log p}(\n dg\ss\n_p + D \n s(g)\n_p),\]
and the latter growth is known to be optimal in the scalar case (see \cite{Hitc}).
\end{remark}

The next result extrapolates Theorem \ref{thm:Pinelis} to exponents  $0<p<2$. By using a variation of the method in \cite[pp. 38-39]{Burk}, an estimate is obtained without the term $\|dg^*\|_p$.

\begin{corollary}\label{cor:Pinelis}
Let $X$ be a $(2,D)$-smooth Banach space. Suppose that $ (f_j)_{j= 0}^k$ is an adapted sequence of $X$-valued random variables, $ (g_j)_{j=0}^k$ is an $X$-valued martingale,
$(V_j)_{j=1}^k$ is a sequence of random contractions on $X$ which is strongly predictable (i.e., each $V_jx$ is strongly $\F_{j-1}$ measurable for all $x\in X$), and assume that we have $f_0=g_0=0$ and
\[f_j = V_{j} f_{j-1} + dg_j, \qquad j=1, \ldots, k.\]
Then for all $0<p<2$ we have
\[ \n f^\star\n_p \le (300D)^{2/p}\n s(g)\n_p.\]
If, moreover, $(g_j)_{j=0}^k$ has conditionally symmetric increments, then
\[ \n f^\star\n_p \le (100D)^{2/p} \n s(g)\n_p.\]
\end{corollary}
\begin{proof}
By Doob's maximal inequality   and the fact that $X$ has martingale type $2$ with constant $D$ (by Remark \ref{rem:Rad})
\[\|dg^\star\|_2\leq 2 \|g^\star\|_2 \le 4\|g\|_2\leq 4D \|s(g)\|_2.\]
Therefore, Theorem \ref{thm:Pinelis} implies
\begin{align}\label{eq:estfstarp2}
\|f^\star\|_2\leq (
4A+B)D \|s(g)\|_2,
\end{align}
where $(A,B) = (10, 10\sqrt{2})$ if $g$ has conditionally symmetric increments and $(A,B) = (60, 40\sqrt{2})$ in the general case.

For non-negative random variables $Z$ and exponents $0<q<1$ we have the identity
\begin{align}\label{eq:Z}\E |Z|^q = q(1-q) \int_0^\infty \E(Z\wedge \lambda) \lambda^{q-2} \ud \lambda.
\end{align}
Setting $K = (
4A+B)D+1$, we claim that
\[\E(|f^\star|^2\wedge \lambda) \le K^2 \E (s(g)^2\wedge \lambda), \ \ \ \  \lambda>0.\]
Once this has been verified, upon taking $q=p/2$, $Z = |f^\star|^2$, and then $Z= s(g)$ in \eqref{eq:Z}, we obtain
\begin{align*}
\E|f^\star|^p & = q(1-q) \int_0^\infty \E(|f^\star|^2\wedge \lambda) \lambda^{q-2} \ud \lambda
\\ & \leq K^2 q(1-q) \int_0^\infty \E(|s(g)|^2\wedge \lambda) \lambda^{q-2} \ud \lambda =K^2 \E |s(g)|^p
\end{align*}
and the result follows.

To prove the claim, set $\tau := \inf\{0\leq n\leq k-1: \sum_{j=1}^{n+1} \E_{j-1} \|dg_j\|^2\geq \lambda\}$, with the convention that $\tau := k$ if the set is empty. Let the adapted sequence of random variables $(F_{j})_{j=0}^k$ be defined by $F_0 := 0$ and
\[F_j := W_j F_{j-1} + dG_j,\quad j=1,\dots, k,\]
where $W_j := V_j$ if $0\le j\le \tau$, $W_j := I$ if $j>\tau$,  and $dG_j := \one_{\{0\le j\le \tau\}} dg_j$. One checks that $f_{j\wedge \tau} = F_j$ for all $j=0,\ldots, k$. Applying \eqref{eq:estfstarp2} to $F$ gives
\begin{align*}
\E \sup_{0\leq j\leq k}\|f_{j\wedge \tau}\|^2 = \E |F^{\star}|^2
& \leq (4A+B)^2D^2 \E s(G)^2
\\ & = (4A+B)^2D^2 \E \sum_{0\leq j\leq k} \one_{\{0\le j\le \tau\}}  \E_{j-1} \|dg_j\|^2
\\ & \leq  (4A+B)^2D^2\E(s(g)^2\wedge \lambda).
\end{align*}
Since $|f^{\star}|^2\wedge \lambda \leq \sup_{0\leq j\leq k}\|f_{j\wedge \tau}\|^2 + \lambda \one_{\{\tau<k\}}$, we obtain
\begin{align*}
\E(|f^\star|^2\wedge \lambda)  \leq \E \sup_{0\leq j\leq k}\|f_{j\wedge \tau}\|^2 + \E(\one_{\{\tau<k\}}\lambda ) \leq K^2\E(s(g)^2\wedge \lambda),
\end{align*}
which gives the claim.
\end{proof}

\section{Maximal inequalities for stochastic convolutions}\label{sec:main}

A family $(S(t,s))_{0\le s\le t\le T}$ of bounded operators on a Banach space
$X$ is called a {\em $C_0$-evolution family} if:
\begin{enumerate}[(1)]
\item $S(t,t) = I$  for all $t\in [0,T]$;
\item $S(t,r) = S(t,s) S(s,r)$ for all $0\le r\le s\le t\le T$;
\item the mapping $(t,s) \to S(t,s)$ is strongly
continuous on the set $\{0\le s\le t\le T\}$.
\end{enumerate}

$C_0$-Evolution family typically arise as the solution operators for the linear time-dependent problem $u'(t) = A(t)u(t)$
in much the same way as $C_0$-semigroups solve the time-independent problem $u'(t) = Au(t)$. The reader is referred to \cite{EN, Pazy, Ta1} for systematic treatments.
If $(S(t))_{t\ge 0}$ is a $C_0$-semigroup on $X$, then $S(t,s):= S(t-s)$ defines a $C_0$-evolution family
$(S(t,s))_{0\le s\le t\le T}$ for every $0<T<\infty$.

\subsection{The main result}
The following theorem is the main result of this paper.

\begin{theorem}\label{thm:contractionS-new}
Let $(S(t,s))_{0\leq s\leq t\leq T}$ be a $C_0$-evolution family of contractions on a $(2,D)$-smooth Banach space $X$ and let $W$ be an adapted $H$-cylindrical Brownian motion on $\Omega$. Then for every $g\in L_{\bF}^0(\Omega;L^2(0,T;\gamma(H,X)))$
the process $(\int_0^t S(t,s)g_s\ud W_s)_{t\in [0,T]}$ has a continuous modification which satisfies, for all $0<p<\infty$,
\[\E\sup_{t\in [0,T]}\Big\n \int_0^t S(t,s)g_s\ud W_s\Big\n^p\leq C_{p,D}^p \|g\|_{L^p(\Omega;L^2(0,T;\gamma(H,X)))}^p,\]
with a constant $C_{p,D}$ depending only on $p$ and $D$.
For $2\le p<\infty$ the inequality holds with $C_{p,D} = 10D\sqrt{p}$.
\end{theorem}

The stochastic integral is well defined by \eqref{eq:Neid}.
By rescaling, more generally it may be assumed that there exists a $\la\geq 0$ such that
that $$\|S(t,s)\| \le e^{\la(t-s)}, \qquad 0\leq s\leq t\leq T.$$
The estimate of the theorem then holds with constant $C_{p,D}$ replaced with $e^{\lambda T} C_{p,D} $.
\begin{proof}
The proof is split into four steps. In the first two steps we prove the theorem for $2\le p<\infty$, in the third step we
consider the case $0<p<2$, and in the fifth the pathwise continuity assertion for $p=0$.

\smallskip
{\em Step 1}. \
Fix a partition $\pi:= \{r_0,\dots,r_N\}$, where $0= r_0<r_1<\ldots <r_N=T$, and let $(K(t,s))_{0\leq s\leq t\leq T}$ be a family of contractions on $X$ with the following properties:
\begin{enumerate}[\rm(i)]
 \item\label{it:K1} $K(t,\cdot)$ is constant on $[r_{j-1}, r_j)$ for all $t\in [0,T]$ and $j = 1,\ldots, N$;
 \item\label{it:K2} $K(\cdot, s)$ is strongly continuous for all $s\in [0,T]$;
 \item\label{it:K3} $S(t,r) K(r,s) = K(t,s)$ for all $0\le s\le r\le t\le T$.
\end{enumerate}
Let $g\in L_{\bF}^p(\Omega;L^2(0,T;\gamma(H,X)))$ and define the process $(v_t)_{t\in [0,T]}$ by
$$v_t := \int_0^t K(t,s) g_s \ud W_s, \quad t\in [0,T].$$
Properties \eqref{it:K1} and \eqref{it:K2} imply that the process $(v_t)_{t\in [0,T]}$ is well defined and has a modification with continuous paths. Indeed, for $t\in [r_{j-1}, r_j]$
\begin{align*}
\int_0^t K(t,s) g_s \ud W_s = \sum_{k=1}^{j-1} K(t,r_{k-1}) \int_{r_{k-1}}^{r_{k}} g_s\ud W_s + K(t,r_{j-1})\int_{r_{j-1}}^{t} g_s\ud W_s,
\end{align*}
which can be seen to have a continuous modification.
Working with such a modification, we will first prove that for all $2\le p<\infty$ we have
\begin{align}\label{eq:maxest-Kg}
\Big\|\sup_{t\in [0,T]}\n v_t\n\Big\|_p\leq 10D\sqrt{p}\|g\|_{L^p(\Omega;L^2(0,T;\gamma(H,X)))}.
\end{align}
By a limiting argument it suffices to consider $p>2$.

For the proof of \eqref{eq:maxest-Kg}, by density we may assume that $g$ is as in \eqref{eq:simple}, i.e.,
\begin{equation*}
g  = \sum_{j=1}^{k} \one_{(s_{j-1},s_{j}]} \sum_{i=1}^\ell h_i\otimes \xi_{ij},
\end{equation*}
where $0= s_0<s_1<\ldots <s_k=T$ and $h_i$ and $\xi_{ij}$ are as in \eqref{eq:simple}. Refining $\pi$ if necessary, we may assume that $s_j\in \pi$ for all $j=0,\dots,k$. We prove \eqref{eq:maxest-Kg} in two steps.

\smallskip
{\em Step 1a}. \
Let $\pi' = \{t_0, t_1, \ldots, t_m\}\subseteq [0,T]$ be another partition. It suffices to prove the bound
\begin{align}\label{eq:toprovemaximalineq}
\Big\|\sup_{t\in \pi'}\n v_t\n\Big\|_p \le a_{\pi'} + 10D\sqrt{p}
\|g\|_{L^p(\Omega;L^2(0,T;\gamma(H,X)))}
\end{align}
with $a_{\pi'} = o(\hbox{mesh}(\pi'))$ as mesh$(\pi')\to 0$.
Refining $\pi'$ if necessary, we may assume that $\pi\subseteq\pi'$.

For fixed $j=1,\ldots,m$ we have, by property  \eqref{it:K3},
\begin{align*}f_{j} := v_{t_j} & = S(t_j, t_{j-1}) v_{t_{j-1}} + \int_{t_{j-1}}^{t_j} K(t_j,s) g_s \ud W_s
 \\ & =: V_{j} f_{j-1} + dG_j,
\end{align*}
where we set $V_{j} := S(t_j, t_{j-1})$ and $dG_j:=\int_{t_{j-1}}^{t_j} K(t_j,s) g_s \ud W_s$. We further set $f_0:=0$ and $G_0:=0$.
By using the symmetry of normally distributed random variables as in \cite[Proposition 4.4.6]{HNVW16} it is seen that
the difference sequence $(dG_j)_{j=1}^m$ is conditionally symmetric. Therefore, by Theorem \ref{thm:Pinelis},
\begin{align}\label{eq:Pinelisvariant}
\n f^\star\n_p \le 5p \n dG\ss\n_p + 10D\sqrt{p} \n s(G)\n_p,
\end{align}
where $f = (f_j)_{j=0}^m$ and $G = (G_j)_{j=0}^m$.

\smallskip
{\em Step 1b}. \
For all $q\in [2, \infty)$ and all $1\le j\le m$, the independence of $W_{t_{j}}-W_{t_{j-1}}$ and $\F_{t_{j-1}}$ implies (see \cite[9.10]{Wil})
\begin{align*}
\E_{j-1}\|dG_j\|^q &= \E_{j-1}\Big\|\sum_{i=1}^\ell (W_{t_{j}}-W_{t_{j-1}})h_i K(t_j, t_{j-1})g_{t_{j-1}} h_i\Big\|^q
\\ & \leq \E_{j-1}\Big\|\sum_{i=1}^\ell (W_{t_{j}}-W_{t_{j-1}})h_i g_{t_{j-1}} h_i\Big\|^q
\\ & = \wt\E \Big\|\sum_{i=1}^\ell (t_{j} - t_{j-1})^{1/2} \wt{\gamma}_{ij} g_{t_{j-1}} h_i \Big\|^q
\\ & = (t_{j} - t_{j-1})^{q/2} \|g_{t_{j-1}}\|_{\gamma_q(H,X)}^q,
\end{align*}
where $(\wt{\gamma}_{ij})_{i\geq 1,j\geq 1}$ is a doubly indexed Gaussian sequence on an independent probability space $(\wt\Om,\wt\F,\wt \P)$ and $\gamma_q(H,X)$ denotes the space $\gamma(H,X)$ endowed with the equivalent $L^q$-norm as discussed in Subsection \ref{subsec:SI}.
We used that $K(t_j, s) = K(t_j, t_{j-1})$ and $g_s = g_{t_{j-1}}$ for $s\in [t_{j-1}, t_j)$. Consequently,
\begin{equation}\label{eq:est-q}
\begin{aligned}
\sum_{j=1}^m \E_{j-1}\|dG_j\|^q
& = \sum_{j=1}^m (t_{j} - t_{j-1})^{q/2} \|g_{t_{j-1}}\|_{\gamma_q(H,X)}^q
\\ & \leq (\text{mesh}(\pi))^{\frac{q}{2}-1} \|g\|_{L^2(0,T;\gamma_q(H,X))}^q.
\end{aligned}
\end{equation}
Applying \eqref{eq:est-q} with $q=p$ and taking expectations, we obtain
\begin{align*}
\n dG\ss\n_p^p & \leq \Big\|\Big(\sum_{j=1}^m \|dG_j\|^p \Big)^{1/p}\Big\|_p^p
 \\ & = \E \sum_{j=1}^m \|dG_j\|^p \leq (\text{mesh}(\pi))^{\frac{p}{2}-1} \E\|g\|_{L^p(0,T;\gamma_p(H,X))}^p.
\end{align*}
Applying \eqref{eq:est-q}  with $q=2$, we obtain
\[\n s(G)\n_p \leq \|g\|_{L^p(\Omega;L^2(0,T;\gamma(H,X)))}.\]
Substituting these bounds into \eqref{eq:Pinelisvariant}, we obtain
\begin{align*}
\Big\|\sup_{t\in \pi'}\n v_t\n\Big\|_p  = \n f\ss \n_p &
\leq 5p \n dG\ss\n_p + 10D\sqrt{p} \n s(G)\n_p
\\ & \leq 5p\, (\text{mesh}(\pi'))^{\frac{1}{2}-\frac{1}{p}} \|g\|_{L^p(\Omega;L^p(0,T;\gamma_p(H,X)))}
\\ & \qquad + 10D\sqrt{p}\|g\|_{L^p(\Omega;L^2(0,T;\gamma(H,X)))}.
\end{align*}
Since $p>2$, this proves \eqref{eq:toprovemaximalineq} for finite rank adapted step processes $g$.

\smallskip
{\em Step 2}. \ Fix $g\in L_{\bF}^p(\Om;L^2(0,T;\gamma(H,X)))$ and $n\in \N$. Set $\sigma_n(s) := j 2^{-n}T$ for $s\in [j2^{-n}T, (j+1)2^{-n}T)$ and define $S_n(t,s) := S(t,\sigma_n(s))$ and
\[v^{(n)}_t := \int_0^t S_n(t,s) g_s \ud W_s.\]
The assumptions  \eqref{it:K1}--\eqref{it:K3} in Steps 1 and 2 apply to $K(t,s) = S_n(t,s)$, $N = 2^n$, and $r_j = j2^{-n}T$. By what has been shown in these steps, the process $v^{(n)}$ has a continuous modification. Moreover, noting that for $n\geq m$ we have $$v^{(n)}_t - v^{(m)}_t = \int_0^t S_n(t,s)(I- S(\sigma_n(s),\sigma_m(s))) g_s \ud W_s,$$ from Step 1 we obtain
\begin{align*}
\Big\|\sup_{t\in [0,T]} \n v^{(n)}-v^{(m)}\n\Big\|_p & \leq 10D\sqrt{p}  \big\|(I- S(\sigma_n(\cdot),\sigma_m(\cdot)))g\big\|_{L^p(\Omega;L^2(0,T;\gamma(H,X)))}.
 \end{align*}
Since the right-hand side
tends to zero by the dominated convergence theorem, $(v^{(n)})_{n\geq 1}$ is a Cauchy sequence in $L^p(\Omega;C([0,T];X))$ and hence converges to some $\wt v$ in $L^p(\Omega;C([0,T];X))$. On the other hand, for all $t\in [0,T]$ we have
$$v^{(n)}_t\to \int_0^t S(t,s) g_s \ud W_s =: u_t$$ with convergence in $L^2(\Omega;X)$. Therefore, $\wt v$ is the required continuous modification of $u$.  Applying Step 1 again we obtain
\begin{align*}
\Big\|\sup_{t\in [0,T]} \n u_t\| \Big\|_{p} = \lim_{n\to \infty} \Big\|\sup_{t\in [0,T]} \n v^{(n)}_t\|\Big\|_{p}\leq 10D\sqrt{p} \|g\|_{L^p(\Omega;L^2(0,T;\gamma(H,X)))}.
\end{align*}

{\em Step 3}. \ In the case $0<p<2$ one can argue in the same way as in the previous steps, using Corollary \ref{cor:Pinelis} instead of Theorem \ref{thm:Pinelis}.
The estimate \eqref{eq:Pinelisvariant} simplifies as the term $\|dG^*\|_p$ does not appear anymore. Alternatively, one could use a standard extrapolation argument involving Lenglart's inequality
\cite[Proposition IV.4.7]{RY}.

\smallskip
{\em Step 4}. \ The continuity assertion for $p=0$ follows  by a standard localisation argument.
\end{proof}

As a consequence of Theorem \ref{thm:contractionS-new}, a simple optimisation argument in the exponent $p$ gives the following exponential tail estimate (see \cite[Corollary 4.4]{NV20a} for details).

\begin{corollary}[Exponential tail estimate]\label{cor:expontail}
If, in addition to the conditions of Theorem \ref{thm:contractionS-new}, we have $g\in L^\infty(\Omega;L^2(0,T;\gamma(H,X)))$, then
\[\P\Bigl(\sup_{t\in [0,T]}\Big\n \int_0^t S(t,s)g_s \ud W_s\Big\n\geq r\Bigr) \leq 2\exp\Bigl(-\frac{r^2}{2\sigma^2}\Bigr), \qquad r>0,\]
where $\sigma^2 = 100eD^2\n g\n_{L^\infty(\Omega;L^2(0,T;\gamma(H,X)))}^2$.
\end{corollary}

This method to derive exponential tail estimates only uses that the constant $C_{p,X}$ in the maximal estimate  has order $O(\sqrt{p})$ for $p\to \infty$. By the same method, similar exponential tail estimates can therefore be deduced from all other results in this paper where the constant is of asymptotic order $O(\sqrt{p})$ .

\begin{remark}\label{rem:Itoformexp}
Under additional assumptions on the evolution family (which are satisfied in the case of $C_0$-semigroups of contractions), a variant of It\^o's formula can be used to give an alternative proof of the estimate of Corollary \ref{cor:expontail} with sharper variance $\sigma^2 = 2D^2\n g\n_{L^\infty(\Omega;L^2(0,T;\gamma(H,X)))}^2$ (see \cite[Theorem 5.6]{NV20a}).
\end{remark}

\subsection{The non-contractive case} We briefly discuss two sets of sufficient conditions for the existence of continuous versions and the validity of maximal estimates for general (i.e., not necessarily contractive) $C_0$-evolution families $(S(t,s))_{0\leq s\leq t\leq T}$. The first of these replaces the condition `$g\in L_{\bF}^0(\Om;L^2(0,T;\gamma(H,X)))$' by `$g\in L_{\bF}^0(\Om;L^q(0,T;\gamma(H,X)))$ for some $q>2$'. Under this stronger assumption, a maximal inequality for general $C_0$ semigroups on Hilbert spaces was obtained  by Da Prato, Kwapie\'n, and Zabczyk \cite{DPKZ} by the so-called factorization method. It was extended to $C_0$-evolution families on Hilbert by Seidler \cite{Sei93}. His proof extends {\em mutatis mutandis} to give the following result, which is taken from \cite{NV20a} where a further discussion is to be found.

\begin{proposition}[Additional time regularity]\label{prop:fact} Let $(S(t,s))_{0\le s\le t\le T}$ be a $C_0$-evol\-ution family on a $(2,D)$-smooth Banach space $X$ and let $2<q<\infty$. For all $g\in L_{\bF}^0(\Omega;L^q(0,T;\gamma(H,X)))$ the process $(\int_0^t S(t,s)g_s \ud W_s)_{t\in [0,T]}$ has a continuous modification which satisfies, for all $0<p\le q$,
\begin{align*}
\E\sup_{t\in [0,T]}\Big\n \int_0^t S(t,s)g_s \ud W_s\Big\n^p \leq C_{p,q,D,T}^p C_{S,T}^p\|g\|_{L^p(\Omega;L^q(0,T;\gamma(H,X)))}^p,
\end{align*}
where $C_{S,T} := \sup_{0\le s\le t\le T}\|S(t,s)\|$.
\end{proposition}

In the second result we assume that $g$ has additional space regularity. Although this may not seem surprising,  we have not been able to find a reference for this in the literature, and for this reason we provide a detailed proof. The result will play a role in Theorem \ref{thm:generalapprox1}, where convergence rates for time discretisation schemes are studied under space regularity assumptions on $g$.

When $A$ is generator of a $C_0$-semigroup on the Banach space $X$, for $\nu\in (0,1)$ we denote by $X_{\nu,\infty} = :(X,\Dom(A))_{\nu,\infty}$ the real interpolation space between $X$ (see \cite{Lun} for more details).

\begin{proposition}[Additional space regularity]\label{prop:maximalineqspacereg}
Let $A$ be the generator of a $C_0$-semigroup $S=(S(t))_{t\geq 0}$ on a $(2,D)$-smooth Banach space $X$ and let
$0<\nu<1$. For all $g\in L_{\bF}^0(\Omega;L^2(0,T;\gamma(H,X_{\nu,\infty})))$
the process $(\int_0^t S(t-s)g_s\ud W_s)_{t\in [0,T]}$, as an $X$-valued process, has a continuous modification which satisfies, for all $0<p<\infty$,
\[\E\sup_{t\in [0,T]}\Big\n \int_0^t S(t-s)g_s\ud W_s\Big\n^p\leq C_{p,D,T,\nu}^p C_{S,T}^p \|g\|_{L^p(\Omega;L^2(0,T;\gamma(H,X_{\nu,\infty})))}^p,\]
where $C_{S,T} = \sup_{0\leq t\leq T} \|S(t)\|$.
\end{proposition}
\begin{proof}
By localisation and Lenglart's inequality, it suffices to prove the continuity and maximal estimate for $p>\frac{1}{2\nu}$.

We have
\[\int_0^t S(t-r)g_r\ud W_r = \int_0^t (S(t-r)-I)g_r\ud W_r + \int_0^t g_r\ud W_r =:u_t+v_t.\]
By Proposition \ref{prop:Seid}, $v$ has a continuous version satisfying the required maximal estimate, so it remains to prove the same for $u$. For this we will use the Kolmogorov--Chentsov continuity criterion \cite[Theorem I.2.1]{RY}.

For $0\leq s\leq t\leq T$ we have
\begin{align*}
\|S(t) -S(s)\|_{\calL(X,X)} \leq 2 C_{S,T}, \qquad  \|S(t) -S(s)\|_{\calL(\Dom(A),X)} \leq C_{S,T} |t-s|.
\end{align*}
Therefore, by interpolation,
\begin{equation}\label{eq:normStSs}
\begin{aligned}
\|S(t)- S(s)\|_{\calL(X_{\nu,\infty},X)} & \leq \|S(t)- S(s)\|_{\calL(X_{\nu,\infty},X)}\\ & \leq 2C_{S,T} |t-s|^{\nu}.
\end{aligned}
\end{equation}
Next, for $0\leq s\leq t\leq T$ we have
\begin{align*}
u_t-u_s = \int_0^s (S(t-r)-S(s-r))g_r \ud W_r + \int_s^t (S(t-r)-I)g_r \ud W_r.
\end{align*}
Taking $L^p(\Omega;X)$-norms, from Proposition \ref{prop:Seid} we obtain
\begin{align*}
\E\Big\|\int_0^s (S(t-r)-S(s-r))g_r \ud W_r\Big\|^p&\leq
C_{p,D}^p \E\|(S(t-\cdot)-S(s-\cdot))g\|_{L^2(0,s;\gamma(H,X))}^p
\\ & \leq (K |t-s|^{\nu})^p \E\|g\|_{L^2(0,T;\gamma(H,X_{\nu,\infty}))}^p,
\end{align*}
where $K = 2C_{S,T} C_{p,D} $. Similarly,
\begin{align*}
\E\Big\|\int_s^t (S(t-r)-I)g_r \ud W_r\Big\|^p&\leq
C_{p,D}^p \E\|(S(t-\cdot)-I) g\|_{L^2(s,t;\gamma(H,X))}^p
\\ & \leq (K|t-s|^{\nu})^p \E\|g\|_{L^2(0,T;\gamma(H,X_{\nu,\infty}))}^p.
\end{align*}
It follows that
\begin{align*}
\E\|u_t-u_s\|^p\leq K^p |t-s|^{\nu p} \E\|g\|_{L^2(0,T;\gamma(H,X_{\nu,\infty}))}^p.
\end{align*}
Now we will use the assumption $p>\frac{1}{2\nu}$, which allows us to apply the Kolmogorov--Chentsov continuity criterion. It implies that for $0<\delta<2\nu-\frac1p$ the process $u$ has a ($\delta$-H\"older) continuous version which satisfies
\begin{align*}
\E \|u\|_{C^{\delta}([0,T];X)}^p\leq
K^p C_{p,T,\delta,\nu}^p\E\|g\|_{L^2(0,T;\gamma(H,X_{\nu,\infty}))}^p.
\end{align*}
Together with the bound $\sup_{t\in [0,T]}\|u(t)\|\leq T^{\delta}\|u\|_{C^{\delta}([0,T];X)}$ and the estimate for $v$, this implies the maximal inequality in the statement of of the proposition.
\end{proof}

\begin{remark}
 The same result holds if we replace $X_{\nu,\infty}$ by any Banach space which continuously embeds into $X_{\nu,\infty}$. In particular this implies to complex interpolation spaces and fractional domain spaces.
\end{remark}

\subsection{Martingales as integrators: Hilbert spaces}\label{subsec:Hilbert}

In the remainder of this section we consider stochastic convolutions driven by an $L^2$-martingale $(M_t)_{t\in [0,T]}$ with values in a separable Hilbert space $H$. For details on stochastic integration in this setting we refer to \cite{MetPel80,Met} and the summary in \cite{HauSei2}. We will use a couple of notions from the theory of stochastic processes that have not been introduced in Section \ref{sec:prelim} but are otherwise completely standard; see for instance \cite{Kal, RY}.

In the present subsection we also let $X$ be a Hilbert space; the case where $X$ is a $2$-smooth Banach space is discussed in the next subsection. By a standard argument involving the essential separability of the ranges of strongly measurable functions, there is no loss of generality in assuming $X$ to be separable. This is relevant as we cite some results from the literature which are stated for separable spaces.

For details on the concepts we introduce below we refer to \cite[Chapter 4]{Met}, where proofs of the various claims made below can be found. We denote by $\lb M_t\rb_{t\in [0,T]}$ the predictable quadratic variation of $M$, and by $\lb\!\lb M_t\rb\!\rb_{t\in [0,T]}$ the predictable tensor quadratic variation of $M$ taking values in the space of trace class operators $\calL_1(H)$.
The covariance process $(Q_{M,t})_{t\in [0,T]}$ is defined as the Radon--Nikod\'ym derivative $Q_M = \frac{{\rm d}\lb\!\lb M\rb\!\rb}{{\rm d}\lb M\rb}$ (note that $\calL_1(H)$ has the Radon--Nikod\'ym property: this space is separable and is canonically isometric to the  dual of the space of compact operators on $H$; see \cite[Theorems 1.3.21, D.2.6]{HNVW16}). Then $Q_M$ is positive and trace class with ${\rm tr}(Q_M)= 1$ almost everywhere on $[0,T]\times \Omega$. For processes $g:[0,T]\times\Omega\to \calL(H,X)$ which are predictable in the strong operator topology, one has
\begin{align}\label{eq:L2integral}
\E\Big\|\int_0^T g_t \ud M_t\Big\|^2 = \E \int_0^T \|g_t Q_{M,t}^{1/2}\|_{\calL_2(H,X)}^2 \ud \lb M\rb_t,
\end{align}
whenever the right-hand side of \eqref{eq:L2integral} is finite.
Moreover, the predictable quadratic variation is given by
\begin{align}\label{eq:quadraticvar}
\Big<\int_0^\cdot g_s \ud M_s\Big>_t = \int_0^t \|g_s Q_{M,s}^{1/2}\|_{\calL_2(H,X)}^2 \ud \lb M\rb_s.
\end{align}
In these identities, $\calL_2(H,X)$ denotes the space of Hilbert--Schmidt
operators from $H$ to $X$.

The following theorem shows that the main result of \cite{Kot82} also holds with a strong type estimate instead of a weak estimate. A similar result was obtained in \cite{Kot84} under additional assumptions on the evolution family $(S(t,s))_{0\leq s\leq t\leq T}$. The result also covers the Poisson case; this can be seen in the same way as in \cite[Section 3]{HauSei2}.

\begin{theorem}\label{thm:generalmart}
Let $(S(t,s))_{0\leq s\leq t\leq T}$ be a $C_0$-evolution family of contractions on a Hilbert space $X$ and let $M$ be a continuous (respectively, c\`adl\`ag) local $L^2$-martingale with values in $H$. Let
$g:[0,T]\times\Omega\to \calL(H,X)$ be a process such that $g(h)$ is predictable for all $h\in H$ and $$\int_0^T \|g_t Q_{M,t}^{1/2}\|_{\calL_2(H,X)}^2 \ud \lb M\rb_t<\infty \ \ \hbox{almost surely.}$$
Then the process $(\int_0^t S(t,s)g_s\ud M_s)_{t\in [0,T]}$ has a continuous (respectively, c\`adl\`ag) modification. Moreover, if $0<p\le 2$, then
\[\E\sup_{t\in [0,T]}\Big\n \int_0^t S(t,s)g_s\ud M_s\Big\n^p\leq C^p_p \E \Big(\int_0^T \|g_t Q_{M,t}^{1/2}\|_{\calL_2(H,X)}^2 \ud \lb M\rb_t\Big)^{p/2},\]
where $C_p$ is a constant depending only on $p$. For $p=2$ the inequality holds with $C=300$.
\end{theorem}
This result can be extended to a larger class of processes $g$ by a density argument, but the description of the space is quite technical. The interested reader is referred to \cite{HauSei2, Met}.

\begin{proof}
By Lenglart's theorem and a localisation argument as in Theorem \ref{thm:contractionS-new} it suffices to consider $p=2$. Moreover, by localisation we may assume that $M$ is a continuous (respectively, c\`adl\`ag) $L^2$-martingale. By approximation it furthermore suffices to consider adapted step processes $g$. We will focus on the continuous case, the c\`adl\`ag case being similar. Only the required changes in the proof of Theorem \ref{thm:contractionS-new} will be indicated.

First of all, $\|g\|_{L^2(0,T;\gamma(H,X))}$ must be replaced by $\int_0^T \|g_t Q_{M,t}^{1/2}\|_{\calL_2(H,X)}^2 \ud \lb M\rb_t$ throughout. With this adjustment, up to \eqref{eq:Pinelisvariant} the proof is {\em verbatim} the same. By Theorem \ref{thm:Pinelis} with $p=2$ we find that
\begin{align*}
\n f^\star\n_2 & \le 60 \n dG\ss\n_2 + 40\sqrt{2} \n s(G)\n_2.
\end{align*}
Noting that $\n dG\ss\n_2\leq 2\|G\ss\|_2 \leq 4\|G\|_2 = 4 \n s(G)\n_2$ by Doob's maximal inequality
and combining the above with \eqref{eq:L2integral} and the bound $\|K(t_j, s)\|\leq 1$, we obtain
\begin{align*}
\n f^\star\n_2^2 \leq C^2\n s(G)\n_2^2 &= C^2\sum_{j=1}^m \E \|dG_j\|^2 \\ &= C^2 \sum_{j=1}^m \E \int_{t_{j-1}}^{t_j}\|K(t_j,s) g(s)Q_M^{1/2}\|_{\calL_2(H,X)}^2 \ud \lb M\rb_s \\&  \leq C^2\E \int_{0}^{T}\|g(s)Q_M^{1/2}\|_{\calL_2(H,X)}^2 \ud \lb M\rb_s,
\end{align*}
where $C = 240+40\sqrt{2} < 300$.
\end{proof}

\begin{remark}\label{rem:HScase}
Let us explain how to extend Theorem \ref{thm:generalmart} to arbitrary $2\le p< \infty$ in the case of continuous local martingales. In particular this extends \cite[(1.13)]{HauSei2} to the case of evolution families.

If $M$ is a continuous local martingale with values in $H$, then Theorem \ref{thm:generalmart} extends to exponents $2\le p<\infty$ with $C_p = 40\sqrt{p}$. As an immediate consequence, Corollary \ref{cor:expontail} holds with $W$ replaced by $M$ and with
\[\sigma^2 = 1600 e \|gQ_{M}^{1/2}\|_{L^\infty(\Omega;L^2(0,T;\calL_2(H,X)))}^2 \]
The proof is similar to those of Theorems \ref{thm:contractionS-new} and \ref{thm:generalmart}, but some modifications are required which we sketch below.

By a stopping time argument we may assume that $\|M\|$ and $\lb M\rb$ are uniformly bounded on $[0,T]\times \Omega$. By approximation it can be assumed that $g$ is an adapted finite rank step process. Then up to \eqref{eq:Pinelisvariant} the proof is the same. Theorem \ref{thm:Pinelis} gives that
\begin{align*}
\n f^\star\n_p & \le 30p \n dG\ss\n_p + 40\sqrt{p} \n s(G)\n_p.
\end{align*}
Moreover the following extension of \eqref{eq:L2integral} holds:
\[\E\sup_{t\in [0,T]}\Big\|\int_0^t g_t \ud M_t\Big\|^p \eqsim_p \E \Big(\int_0^T \|g_t Q_{M,t}^{1/2}\|_{\calL_2(H,X)}^2 \ud \lb M\rb_t\Big)^{p/2}.\]
Since $g$ is uniformly bounded it follows that
\begin{align*}
\|dG\ss\|^p_p & \leq \sum_{j=1}^m \E \|dG_j\|^p
\\ & \leq C_g^p \sum_{j=1}^m \E |\lb M\rb_{t_j} - \lb M\rb_{t_{j-1}}|^{{p}/{2}}
\leq C_g^p \E (\sup_{j}|\lb M\rb_{t_j} - \lb M\rb_{t_{j-1}}|^{\frac{p-2}{2}}) |\lb M_T\rb|.
\end{align*}
By dominated convergence the right-hand side tends to zero as the mesh size tends to $0$. The result follows once we have shown that
\[s(G)^2 \to \int_0^T \int_0^T \|g_t Q_{M,t}^{1/2}\|_{\calL_2(H,X)}^2 \ud \lb M\rb_t\]
with convergence in $L^{p/2}(\Omega)$. If we replace $s(G)^2$ by $\widetilde{s}(G)^2 := \sum_{j=1}^m \|dG_j\|^2$ this follows from \eqref{eq:quadraticvar} (as explained in \cite[Section 4]{BurkStoch}, the scalar case considered in \cite{Doleans} extends to the Hilbert space). The proof will be completed by showing that
$$\E|\widetilde{s}(G)^2 - s(G)^2|^{q}\to 0$$
for any $q\in [1, \infty)$.  Without loss of generality we may take $q\geq 2$ and since $g$ is an adapted finite rank step process.
To prove the convergence in $L^q(\Omega)$ we note that by the scalar case of Theorem \ref{thm:Pinelis}, applied with $V_j = I$ and martingale differences $dL_j = \|dG_j\|^2 - \E_{j-1}(\|dG_j\|^2)$,  for all $2\le q<\infty$  we have
\begin{align*}
\|\widetilde{s}(G)^2 - s(G)^2\|_q & \leq 30 q\|dL\ss\|_q + 40\sqrt{q}\|s(L)\|_q
\\ & \leq 60q \|d G\ss\|_{2q}^2 + 80\sqrt{q} \Big\|\Big(\sum_{j=1}^m \E_{j-1} \|dG_j\|^4\Big)^{1/2}\Big\|_q.
\end{align*}
We have already seen that the first term tends to $0$ as the mesh size tends to zero. For the second term we use \cite[Proposition 3.2.8]{HNVW16} and  H\"older's inequality to find that
\begin{align*}
\Big\|\Big(\sum_{j=1}^m \E_{j-1} \|dG_j\|^4\Big)^{1/2}\Big\|_q
& \leq \frac{q^2}{4} \Big\|\Big(\sum_{j=1}^m \|dG_j\|^4\Big)^{1/2}\Big\|_q
\leq  \frac{q^2}{4}\|dG\ss\|_{2q} \|s(G)\|_{2q} \to 0
\end{align*}
as $\text{mesh}(\pi)\to 0$.
\end{remark}

\subsection{Martingales as integrators: $2$-smooth UMD Banach spaces}

As before we let $H$ be a separable Hilbert space and turn to the case where $X$ is a $(2,D)$-smooth Banach space with the UMD property. Discussions of UMD spaces can be found in \cite{HNVW16, Pis-Mart}. Rather than introducing this property here, we content ourselves by mentioning that examples of Banach spaces with this property include Hilbert spaces, $L^p$-spaces with $1<p<\infty$ and most classical function spaces constructed from these.
We will prove an extension of the maximal estimate of the preceding subsection to this setting by using some results from \cite{yaroslavtsev2018burkholder}.
To avoid technicalities with non-predictable quadratic variations we only consider {\em continuous} local martingales with values in $H$. In that case the quadratic variation considered in \cite{yaroslavtsev2018burkholder} coincides with the one of Subsection \ref{subsec:Hilbert} (see \cite[Theorem 20.5]{Met}).

Let $g:[0,T]\times\Omega\to \calL(H,X)$ be a process such that $g(h)$ is predictable for all $h\in H$ and
$$\|g_t Q_{M,t}^{1/2}\|_{\gamma(L^2(0,T;H),d\lb M\rb_t;X)}^2<\infty \ \ \hbox{almost surely}.$$
By \cite[Theorem 4.1]{VerYar} (see also \cite[Corollary 7.4 and Remark 7.6]{yaroslavtsev2018burkholder})
these assumptions enable one to construct a stochastic integral $\int_0^t g_s \ud M_s$ which, for all $0<p<\infty$, satisfies the two-sided estimate
\begin{align}\label{eq:L2integralBanach}
\E\sup_{t\in [0,T]}\Big\|\int_0^t g_s \ud M_s\Big\|^p \eqsim_{p,X} \E \Big(\|g_t Q_{M,t}^{1/2}\|_{\gamma(L^2(0,T;H),d\lb M\rb_t;X)}^2\Big)^{p/2}
\end{align}
whenever the expression on the right-hand side is finite. If in addition $X$ has type $2$ (which holds if $X$ is $2$-smooth), then by \cite[Theorem 6.1]{vNWe05}
\begin{align}\label{eq:L2integralBanach2}\|g_t Q_{M,t}^{1/2}\|_{\gamma(L^2(0,T;H),d\lb M\rb_t;X)}^2 \leq \tau_{2,X}^2\int_0^T \|g_t Q_{M,t}^{1/2}\|_{\gamma(H,X)}^2 \ud \lb M\rb_t,
\end{align}
where $\tau_{2,X}$ is the type $2$ constant of $X$. We will consider processes for which the right-hand side is finite almost surely.

\begin{theorem}\label{thm:generalmart2}
Let $X$ be a $(2,D)$-smooth UMD Banach space.
Let $(S(t,s))_{0\leq s\leq t\leq T}$ be a $C_0$-evolution family of contractions on $X$ and let $M$ be a continuous local martingale with values in $H$. Let $g:[0,T]\times\Omega\to \calL(H,X)$ be a process such that $g(h):[0,T]\times \Om\to X$ is predictable for all $h\in H$ and $$\int_0^T \|g_t Q_{M,t}^{1/2}\|_{\gamma(H,X)}^2 \ud \lb M\rb_t<\infty \ \ \hbox{almost surely}.$$ Then the process $(\int_0^t S(t,s)g_s\ud M_s)_{t\in [0,T]}$ has a continuous modification which satisfies, for all $0<p<\infty$,
\[\E\sup_{t\in [0,T]}\Big\n \int_0^t S(t,s)g_s\ud M_s\Big\n^p\leq C_{p,X}^p \E \Big(\int_0^T \|g_t Q_{M,t}^{1/2}\|_{\gamma(H,X)}^2 \ud \lb M\rb_t\Big)^{p/2},\]
where $C_{p,X}$ is a constant depending only on $p$ and $X$.
\end{theorem}
\begin{proof}
We argue as in Theorem \ref{thm:generalmart} and Remark \ref{rem:HScase}. Since we may assume that $g$ takes values in a finite dimensional subspace of $X$, as in Remark \ref{rem:HScase} it follows that $\|dG^\star\|_p\to 0$ as the mesh$(\pi)\to 0$. It remains to estimate $s(G)$.
By a standard argument \eqref{eq:L2integralBanach} and \eqref{eq:L2integralBanach2} imply
\[\E_{j-1} \|dG_j\|^2 \leq C_{X}^2 \E_{j-1} \Big(\int_{t_{j-1}}^{t_j} \|g_t Q_{M,t}^{1/2}\|_{\gamma(H,X)}^2\ud \lb M\rb_t\Big)=:C_{X}^2\E_{j-1} (\xi_j) ,\]
where $C_X$ is a constant only depending on $X$.
Therefore, by \cite[Proposition 3.2.8]{HNVW16},
\begin{align*}
\|s(G)\|_p^p \leq C_{X}^p \E\Big(\sum_{j=1}^m \E_{j-1} (\xi_j)\Big)^{p/2}\!\!
& \leq (p/2)^{p/2} C_{X}^p \E\Big(\sum_{j=1}^m \xi_j\Big)^{p/2}
\\ & = (p/2)^{p/2} C_{X}^p \E\Big(\int_{0}^{T} \|g_t Q_{M,t}^{1/2}\|_{\gamma(H,X)}^2 \ud \lb M\rb_t \Big)^{p/2}.
\end{align*}
The proof can now be completed as before.

Observe that this method gives the result with $C_{p,X} = \frac{40}{\sqrt{2}}p C_X$ for $p\geq 2$, which is linear in $p$ as $p\to \infty$; this contrasts with the $O(\sqrt{p})$ growth obtained in all other places in the paper.
\end{proof}

The infinite dimensional version of the Dambis--Dubins--Schwarz theorem of \cite[Theorem 4.9]{VerYar} suggests that the correct order of the constant in Theorem \ref{thm:generalmart2} is $O(\sqrt{p})$.

We expect that a large portion of Theorem \ref{thm:generalmart2} extends to the setting of (non necessarily continuous) local martingales if one replaces the predictable quadratic variation $\lb M\rb$ by the process $[M]$ as defined in \cite[Theorem 20.5]{Met}. However, usually it is preferred to work with a predictable quadratic variation. An alternative substitute for predictability has been recently developed in \cite{Dirk14} in the Poisson case and in \cite{DiYar, yaroslavtsev2019local} for general local martingales, but the norms are much more complicated to work with. It would be interesting to see if one can combine our techniques with the estimates in  \cite{Dirk14,DiYar} for $X = L^q$ with $2\le q<  \infty$, or in \cite{yaroslavtsev2019local} for more general Banach spaces $X$.

\section{Applications to time discretisation}\label{sec:numerical}

In this section we will apply our abstract results to prove stability of certain numerical approximations of stochastic evolution equations with additive noise of the form
\begin{equation}\label{eq:SEE}
    \begin{cases}
        \ud u_t &= A(t)u_t\ud t + g_t \ud W_t, \qquad t\in [0,T], \\
        u_0 &= 0.
    \end{cases}
\end{equation}
This setting covers to both parabolic and hyperbolic time-dependent SPDEs; the latter class includes the stochastic wave equation and the Schr\"odinger equation.
To solve \eqref{eq:SEE} numerically one typically uses discretisation in time and space \cite{JenKlo, Lord}. Here we will only consider time discretisation, leaving space-time discretisation and the extension to semi-linear equations with multiplicative noise for a future publication. In that respect the results presented here serve as a proof-of-principle only.
We mainly focus on the splitting scheme and the implicit Euler scheme, although the method is robust and can be applied to other schemes as well.

In what follows, for $n=1,2,\dots$ we set $t_j^{(n)} := jT/n$ and consider the partition
$$\pi^{(n)} := \{t_j^{(n)}: j=0,\ldots, n\}$$ as a discretision of the interval $[0,T]$.
We fix a process $g\in L_{\bF}^0(\Omega;L^2(0,T;\gamma(H,X)))$ and consider the continuous martingale

$$ M_t := \int_0^t g_s \ud W_s, \quad t\in [0,T].$$
For $j=0,\dots, n$ we set
\begin{align}\label{eq:defMg}
 d_j^{(n)} M := M_{t_{j}^{(n)}} - M_{t_{j-1}^{(n)}} = \int_{t_{j-1}^{(n)}}^{t_{j}^{(n)}} g_s \ud W_s.
 \end{align}
In the presence of a $C_0$-evolution family $(S(t,s))_{0\leq s\leq t\leq T}$ we set
$$u_t:= \int_0^t S(t,s) g_s\ud W_s, \quad t\in [0,T].$$
This covers the special case of $C_0$-semigroups by letting $S(t,s) =S(t-s)$.

\subsection{The splitting method}
Our first result gives stability of a time discretisation scheme for the stochastic convolution process involving a $C_0$-evolution family of contractions called the {\em splitting method} (also called the {\em exponential Euler method}). This scheme has already been employed in the proof of Theorem \ref{thm:contractionS-new}. An extension to random evolution families is discussed in Remark \ref{rem:splitting}.

\begin{theorem}[Uniform convergence of the splitting method]\label{thm:splitting}
Let $(S(t,s))_{0\leq s\leq t\leq T}$ be a $C_0$-evolution family of contractions on a  $(2,D)$-smooth Banach space $X$. Let $g\in L_{\bF}^p(\Omega;L^2(0,T;\gamma(H,X)))$ with $0< p<\infty$.
Define, for $n\ge 1$,
\begin{equation*}
\begin{cases}
u_0^{(n)} & :=0, \\
u_j^{(n)} & := S(t_{j}^{(n)}, t_{j-1}^{(n)}) ( u_{j-1}^{(n)} + d_j^{(n)} M), \quad j= 1,\dots,n,
\end{cases}
\end{equation*}
where $d_j^{(n)} M$ is given by \eqref{eq:defMg}.
Then for all $n\ge 1$ we have
\begin{align}\label{eq:estsplitting1}\E\sup_{j=0,\ldots,n}\|u_{t_j^{(n)}} - u_j^{(n)}\|^p\leq C_{p,D}^p \E\|s\mapsto (S(s,\sigma_n(s)) - I)  g_s\|_{L^2(0,T;\gamma(H,X))}^p,
\end{align}
where $\sigma_n(s)= t_{j-1}^{(n)}$ for $s\in [t_{j-1}^{(n)}, t_j^{(n)})$.
In particular, \begin{align}\label{eq:estsplitting2}\limn \E\sup_{j=0,\ldots,n}\|u_{t_j^{(n)}} - u_j^{(n)}\|^p = 0.
               \end{align}
For $2\le p<\infty$ the estimate \eqref{eq:estsplitting1} holds with $C_{p,D} = 10D\sqrt{p}$.
\end{theorem}

The process $u$ has a continuous modification by Theorem \ref{thm:contractionS-new}.
We will not need this modification in the proof, because the suprema in \eqref{eq:estsplitting1} and \eqref{eq:estsplitting2} are taken with respect to finite index sets. This remark applies to all results
in this subsection and the next (in Theorem \ref{thm:generalapprox1} the existence of the continuous modification follows from  Proposition \ref{prop:maximalineqspacereg}).

\begin{proof}
To simplify notation we fix $n\geq 1$ and write $t_j := t_j^{(n)}$, $v_j:= u_j^{(n)}$, and $d_j M:= d_j^{(n)}M$.
By induction one checks that $v_0=0$ and
\begin{align*}
v_{k} =  \sum_{j=1}^k S(t_k, t_{j-1}) d_jM, \qquad k=1,\dots,n.
\end{align*}
Therefore,
\begin{align*}
u_{t_k} - v_{k} &= \sum_{j=1}^k \int_{t_{j-1}}^{t_j}  (S(t_k,s) - S(t_k, t_{j-1})) g_s \ud W_s
\\ & = \sum_{j=1}^k \int_{t_{j-1}}^{t_j}  S(t_k, s) (I - S(s,t_{j-1})) g_s \ud W_s
\\ & = \sum_{j=1}^k \int_{t_{j-1}}^{t_j}  S(t_k, s) (I - S(s,\sigma_n(s)) g_s \ud W_s
\\ & = \int_{0}^{t_k}  S(t_k,s) (I - S(s,\sigma_n(s)) g_s \ud W_s
\end{align*} and hence, by Theorem \ref{thm:contractionS-new},
\begin{align*}
\E\sup_{j=0,\dots,n}\|u_{t_j} - v_{j}\|^p & \leq \E \sup_{t\in [0,T]} \Big\|\int_{0}^{t}  S(t,s) (I - S(s,\sigma_n(s)) g_s \ud W_s\Big\|^p
\\ & \leq C_{p,D}^p\E\|s\mapsto (I - S(\cdot,\sigma_n(\cdot))g_s\|^p_{L^2(0,T;\gamma(H,X))}.
\end{align*}
The assertion $E_n\to 0$ as $n\to \infty$ follows by dominated convergence in combination with the convergence criterion \cite[Theorem 9.1.14]{HNVW17}.
\end{proof}

In the next corollary we obtain explicit convergence rates for processes $g$ taking values in intermediate spaces. In order to make the statement easy to formulate we only consider the case of semigroup generators.

\begin{corollary}[Uniform convergence of the splitting method with decay rate]\label{cor:splittingdecay}
Let $(S(t))_{t\ge 0}$ be a $C_0$-contraction semigroup on a $(2,D)$-smooth Banach space $X$.
As in the preceding theorem, for $n\ge 1$ let
\begin{equation*}
\begin{cases}
u_0^{(n)} & :=0, \\
u_j^{(n)} & := S(t_{j}^{(n)}- t_{j-1}^{(n)}) (u_{j-1}^{(n)} + d_j^{(n)}M), \quad j= 1,\dots,n,
\end{cases}
\end{equation*}
where $d_j^{(n)} M$ is given by \eqref{eq:defMg}. Let $X_{\nu} := (X,\Dom(A))_{\nu,\infty}$ for $\nu\in (0,1)$ and $X_{1} := \Dom(A)$, where $A$ is the generator of the semigroup. If $g\in L_{\bF}^p(\Omega;L^2(0,T;\gamma(H,X_{\nu})))$ with $0< p<\infty$, then for all $n\ge 1$ we have
\[\E\sup_{j=0,\ldots,n}\|u_{t_j^{(n)}} - u_j^{(n)}\|^p\leq \Bigl(2 C_{p,D} \bigl(\frac{T}{n}\bigr)^{\nu}\Bigr)^p  \|g\|_{L^p(\Omega;L^2(0,T;\gamma(H,X_{\nu})))}^p.\]
For $2\le p<\infty$ the inequality holds with $C_{p,D} = 10D\sqrt{p}$.
\end{corollary}

A version of the above result for $C_0$-semigroups which are not necessarily contractive and a general class of discretisation schemes will proved in Theorem \ref{thm:generalapprox1}.
\begin{proof}
Since $\|(I - S(t))x\| \leq 2 \|x\|$ and
\begin{align*}
\|(I - S(t))x\|& \leq \int_0^t \|S(s) Ax\| \ud s \leq t \|A x\|,
\end{align*}
for $0<\nu<1$
by interpolation we obtain
\[\|(I - S(t))x\|\leq 2 t^{\nu} \|x\|_{X_{\nu}}.\]
For $\nu=1$ we have $$\|(I - S(t))x\|\leq t \|x\|_{\Dom(A)} = t \|x\|_{X_1}. $$
The result now follows from Theorem \ref{thm:splitting} and the ideal property (see \cite[Theorem 9.1.10]{HNVW17}).
\end{proof}

\begin{remark}[Pathwise convergence]\label{rem:pathwiseconv}
If we assume $p \nu>1$ in Corollary \ref{cor:splittingdecay}, then for all $\beta\in (0,\nu-\frac1p)$ there exists a random variable $\xi\in L^p(\Omega)$ such that, almost surely,
\[\sup_{j=0,\ldots,n}\|u_{t_j^{(n)}} - u_j^{(n)}\|\leq n^{-\beta} \xi.\]
Indeed, setting $\xi := (\sum_{n\geq 1} n^{\beta p} \sup_{j=0,\ldots,n} \|u_{t_j^{(n)}} - u_j^{(n)}\|^p)^{1/p}$,
by Corollary \ref{cor:splittingdecay} we have
\begin{align*}
\E|\xi|^p \leq \Bigl(2 C_{p,D} \bigl(\frac{T}{n}\bigr)^{\nu}\Bigr)^p \sum_{n\geq 1} n^{\beta p} n^{-\nu p},
\end{align*}
the sum on the right-hand side being convergent since $(\nu-\beta)p>1$.
\end{remark}

\subsection{General time discretisation methods}

We now investigate whether analogues of Theorem \ref{thm:splitting} hold for general time discretisation methods.
Before returning to convergence questions, we consider a stability result for abstract numerical schemes featuring random operators $V_{j,n}$ satisfying an $\F_{t_{j-1}}$-measurability condition. In particular, the operators are allowed to depend on $u$ and $g$ up to time $t_{j-1}$. This makes this result applicable to nonlinear problems.

\begin{proposition}[Stability]\label{prop:stability}
Let $X$ be a $(2,D)$-smooth Banach space and assume that $g\in L_{\bF}^p(\Omega;L^2(0,T;\gamma(H,X)))$ with  $2\le p<\infty$. For $n=1,2,\dots$ and $j=1,\dots, n$ assume that the random contraction $V_{j,n}:\Omega\to \calL(X)$ is such that $V_{j,n}x$ is strongly $\F_{t_{j-1}^{(n)}}$-measurable for all $x\in X$, and define
\begin{equation*}
\begin{cases}
u_0^{(n)} & :=0, \\
u_{j}^{(n)} &: = V_{j,n} (u_{j-1}^{(n)} +  d_j^{(n)}M),   \quad j=1, \ldots, n,
\end{cases}
\end{equation*}
where $d_j^{(n)}M$ is given by \eqref{eq:defMg}. Then
\[\E \sup_{j=0,\ldots,n}\|u_j^{(n)}\|^p\leq K_{p,D}^p \|g\|_{L^p(\Omega;L^2(0,T;\gamma(H,X)))}^p,\]
where $K_{p,D} = \frac{100 D p^{5/2}}{p-1} + \frac{10}{\sqrt{2}} D^2 p$.
\end{proposition}
\begin{proof}
We fix $n\ge 1$ and write $t_j:=t_j^{(n)}$, $d_{j}M := d_j^{(n)}M$,  and $d_{j}\wt{M} := V_{j,n} d_j^{(n)}M$. Theorem \ref{thm:Pinelis} and the contractivity of $V_{j,n}$, and Doob's maximal inequality imply that
\begin{align*}
\Big\|\sup_{j=0,\dots,n}\|u_j^{(n)}\|\Big\|_p &\leq 5p \|d\wt{M}^{\star}\|_p + 10D\sqrt{p} \|s(\wt{M})\|_p
\\ &\leq 5p \|dM^{\star}\|_p + 10D\sqrt{p} \|s(M)\|_p
\\ &\leq 10p \|M^{\star}\|_p + 10D\sqrt{p} \|s(M)\|_p
\\ & \leq \frac{10p^2}{p-1}\|M_T\|_p + 10D\sqrt{p} \|s(M)\|_p.
\end{align*}
We will estimate the terms on the right-hand side separately. By Proposition \ref{prop:Seid},
\begin{align*}
\|M_T\|_p \leq 10D\sqrt{p} \|g\|_{L^p(\Omega;L^2(0,T;\gamma(H,X)))}.
\end{align*}
To estimate $s(M)$, by \eqref{eq:Neid} we have
\begin{align*}
\E_{j-1}\|d_jM\|^2\leq D^2 \E_{j-1} \|g\|_{L^2(t_{j-1},t_{j};\gamma(H,X))}^2=:D^2 \E_{j-1}(\xi_j).
\end{align*}
By the dual of Doob's maximal inequality (see \cite[Proposition 3.2.8]{HNVW16}) and using $p/2\geq 1$
\begin{align*}
\|s(M)\|_p^p & \leq D^p \E \Big(\sum_{j=1}^n \E_{j-1}(\xi_j) \Big)^{p/2}
\\ & \leq
(p/2)^{p/2} D^p \E \Big(\sum_{j=1}^n \xi_j \Big)^{p/2}
 = (p/2)^{p/2} D^p \E\|g\|_{L^2(0,T;\gamma(H,X))}^p.
\end{align*}
The required estimate follows by combining the estimates.
\end{proof}

\begin{remark}
 For $p=2$ the inequality holds with
$K_{2,D} = 40 D + 10\sqrt{2} D^2$. This is because in the case $p=2$ we can use \eqref{eq:Neid} instead of Proposition \ref{prop:Seid}.
\end{remark}

\begin{remark}
In the setting of monotone operators on Hilbert spaces, a related stability result for $p=2$ for the implicit Euler method can be found in \cite[Theorem 2.6]{GyMi07}.
\end{remark}

Returning to the problem of convergence, the convergent numerical schemes which we will consider are given in the following definition.

\begin{definition}\label{def:convergentscheme} Let $X$ be a Banach space.
An $\calL(X)$-valued {\em scheme} is a function $R:[0,\infty)\to\calL(Y,X)$.
If $A$ generates a $C_0$-semigroup $S$ on $X$ and $Y$ us a Banach space continuously and densely embedded in $X$,
an $\calL(X)$-valued scheme $R$ is said to {\em approximate $S$ to order $\alpha>0$ on $Y$} if for all $T>0$ there exists a constant $K\geq 0$ such that for all integers $n\geq 1$ and $t\in [0,T]$ we have
\begin{align}\label{eq:convorderalpha}
\|R(t/n)^n  - S(t)\|_{\calL(Y,X)}\leq K (t/n)^{\alpha}.
\end{align}
A scheme $R$ is said to be {\em contractive} if $\n R(t)\n\le 1$ for all $n\ge 1$ and $t\ge 0$.
\end{definition}

If $R$ approximates $S$ to order $\alpha$ on $Y$ and there exists a constant $C\ge 0$ such that
$$\|R(t/n)^n\|\leq C \ \ \hbox{and} \ \  \|S(t)\|\leq C \ \ \hbox{for all} \ \ n\geq 1, \ \ t\in [0,T],$$ then by real interpolation it approximates $S$ to order $\theta\alpha$ on the real interpolation spaces $(X, Y)_{\theta,\infty}$ for $\theta\in (0,1)$ with estimate
\begin{align*}
\|R(t/n)^n  - S(t)\|_{\calL((X, Y)_{\theta,\infty},X)}\leq  (2C)^{1-\theta} K^{\theta} (t/n)^{\theta\alpha}, \ \ t\ge 0.
\end{align*}
An interesting special case arises when $Y = \Dom(A^m)$. If an $\calL(X)$-valued scheme $R$ approximates $S$ to order $\alpha$ on $\Dom(A^m)$, then $R$ approximates $S$ to order $\theta\alpha$ on $(X,\Dom(A^m))_{\theta,\infty}$.

\begin{proposition}\label{prop:lower}
Let $T>0$ and suppose that there exists a constant $C\geq 0$ such that for all $t\in (0,T]$ and integers $n\geq 1$, $\|R(t/n)^n\|\leq C$ and $\|S(t)\|\leq C$.  Suppose that the $\calL(X)$-valued scheme
$R$ approximates $S$ to order $\alpha$ on $\Dom(A^m)$ for some integer $m\ge 1$, and let $0<\theta<1$. Then
$R$ approximates $S$ to order $\theta\alpha$ on $(X,\Dom(A^m))_{\theta,\infty}$.
\end{proposition}

Since the continuous embedding $\Dom((-A)^{\theta m})\hookrightarrow (X,\Dom(A^m))_{\theta,\infty}$ holds,
we obtain the following: If $\|S(t)\|\leq M e^{\mu t}$ for all $t\ge0$, with $M\ge 1$ and $\mu\in \R$, then
$R$ approximates $S$ to order $\theta\alpha$ on the fractional domain $\Dom((\mu-A)^{\theta m})$.

We will now review some examples of numerical schemes satisfying the conditions of the above definition. Classical references include \cite{BrennerThomee,HershKato} and, for analytic semigroups, \cite{CLPT93}. A new and unified approach to approximation of semigroups which sharpens several classical estimates has been recently developed in \cite{GoTo14,GoHoTo19}.

Part \eqref{it:BrennerThomee} of the next theorem follows from \cite[Theorem 4]{BrennerThomee}; see also \cite{HershKato}. More elaborate versions on interpolation spaces can be found in \cite{Kovacs06}. Part \eqref{it:analyticfolklore1} follows from \cite[Theorem 4.2]{LTW91} by interpolating the stability result \cite[Theorem 5]{CLPT93} using Proposition \ref{prop:lower} (see \cite[Theorem 9.2.3]{Haase} for a direct approach, which also does not rely on $0\in \varrho(A)$).

\begin{theorem}[Time discretisation]\label{thm:BrennerThomee}
Let $r:\C\to \C$ be a rational function such that $|r(z)|\leq 1$ for all $\Re z \le 0$, and assume that there exists an integer $\ell\geq 1$ such that $$|r(z) - e^{z}| = O(z^{\ell+1}) \ \ \hbox{as} \ \ z\to 0.$$
Let $A$ be the generator of a bounded $C_0$-semigroup on $(S(t))_{t\geq 0}$ a Banach space $X$ and set $$R(t) := r(tA), \qquad t\ge 0.$$
\begin{enumerate}[{\rm (1)}]
\item\label{it:BrennerThomee} $R$ approximates $S$ to order $\eta(\ell,k)$ on $\Dom(A^{k})$ for all integers
$k\in \{1, \ldots,\ell+1\}\setminus\{\frac{\ell+1}{2}\}$, where
    \[\eta(\ell,k) = \left\{
                         \begin{array}{ll}
                           k-\frac12, & \hbox{if $k<\frac{\ell+1}{2}$;} \\
                           \frac{k\ell}{\ell+1}, & \hbox{if $\frac{\ell+1}{2}<k\leq \ell+1$.}
                         \end{array}
                       \right.
    \]
\end{enumerate}

If the semigroup is analytic and bounded on a sector, then:
\begin{enumerate}[{\rm (1)}]
\setcounter{enumi}{1}
\item \label{it:analyticfolklore1} $R$ approximates $S$ to order $\nu$ on $\Dom((-A)^{\nu})$ for all $\nu\in (0,\ell]$.
\end{enumerate}
\end{theorem}

\begin{example}[Time discretisation for $C_0$-semigroups]\label{ex:generaldiscr}
Let $A$ be the generator of a bounded $C_0$-semigroup $(S(t))_{t\geq 0}$ on a Banach space $X$. For each of the functions $r$ below we set $$R(t) := r(tA), \qquad t\ge 0.$$ Then $R$ approximates $S$ in each of the following cases:
\begin{enumerate}[{\rm (1)}]
\item splitting: $r(z) = e^{z}$, to any order on $X$.
\item implicit Euler: $r(z) = (1-z)^{-1}$, to order $\alpha$ on $\Dom((-A)^{2\alpha})$ for all $\alpha\in (0,1]$ (see \cite[Theorem 1.3]{GoTo14} or \cite[Corollary 4.4]{Kovacs06}).
\item Crank--Nicholson: $r(z) = (2+z)(2-z)^{-1}$, to order $\nu$ on $\Dom((-A)^{k})$ for points $(k,\nu)$ on the graph of the piecewise linear function connecting the points $(\frac12, 0)$, $ (1,\frac12)$, $(2,\frac43)$, and $(3,2)$ (see \cite[Theorem 1.1 and 4.1]{Kovacs06}). If moreover $R$ is stable (see Proposition \ref{prop:contractionHilbert} for sufficient conditions), then the order is $\nu$ on $\Dom((-A)^{{3\nu}/{2}})$ for any $\nu\in (0,2]$ (see \cite[Corollary 4.4]{Kovacs06}).
\end{enumerate}
\end{example}

\begin{example}[Time discretisation for analytic $C_0$-semigroups]\label{ex:analyticdiscr}
Let $A$ be the generator of a bounded analytic $C_0$-semigroup $(S(t))_{t\geq 0}$ on $X$.
For each of the functions $r$ below we set $$R(t) := r(tA), \qquad t\ge 0.$$ Then $R$ approximates $S$ in each of the following cases:
\begin{enumerate}[{\rm (1)}]
\item splitting: $r(z) = e^{z}$, to any order on $X$.
\item implicit Euler: $r(z) = (1-z)^{-1}$, to order $\nu$ on $\Dom((-A)^{\nu})$ for any $\nu\in (0,1]$.
\item Crank--Nicholson: $r(z) = (1+\frac12z)(1-\frac12z)^{-1}$, to order $2\nu$ on $\Dom(A^{2\nu})$ for any $\nu\in (0,1]$.
\end{enumerate}
\end{example}

If $A$ generates a contractive $C_0$-semigroup $(S(t))_{t\geq 0}$ the splitting method and implicit Euler methods lead to contractive approximants $S_n(t)$. In the following proposition we discuss another class of examples where this holds. It applies to all numerical schemes of the form $R(t) = r(tA)$ considered in Theorem \ref{thm:BrennerThomee} and includes all schemes considered in \cite{BrennerThomee, HershKato}.  We use the notation
$$ \Sigma_\sigma = \{z\in \C\setminus\{0\}: \ |\arg(z)|<\sigma\},$$
where the argument is taken from $(-\pi,\pi]$.

\begin{proposition}\label{prop:contractionHilbert}
Let $A$ be the generator of a $C_0$-semigroup of contractions on a Hilbert space. Suppose that $r:\Sigma_\sigma\to \C$ is holomorphic for some $\frac12\pi<\sigma<\pi$ and satisfies $|r(z)|\leq 1$ for all $\Re z\ge 0$. Then $\|r(-tA)\|\leq 1$ for all $t>0$, where
$r(-tA)$ is defined through the $H^\infty$-calculus of $-A$.
\end{proposition}

The proof is immediate from \cite[Theorem 10.2.24]{HNVW17}.
The proposition is false beyond the Hilbert space setting. Indeed, for the operator $A = {\rm d}/{\rm d}x$ on $X=L^p(\R)$ with $p\neq 2$ or $X=C_0(\R)$, in \cite{BrennerThomee70} it was shown that contractivity of $R(t)$ fails
for a general class of schemes  (see also \cite{CLPT93} for the Crank--Nicholson scheme).

In what follows we restrict ourselves to the semigroup setting, but expect the results to extend to evolution families under suitable additional conditions. In the next theorem we obtain convergence rates for a rather general class of discretisation schemes, which in case of the splitting method turn out to be equal to the ones of Corollary \ref{cor:splittingdecay} up to a logarithmic term. Modulo this term, the theorem extends Corollary \ref{cor:splittingdecay} in two ways:
\begin{itemize}
\item contractivity of $S$ is not needed;
\item the result holds for arbitrary approximation schemes.
\end{itemize}
The proof directly uses Seidler's version of the Burkholder inequality of Proposition \ref{prop:Seid} in combination Proposition \ref{prop:stochasticellinftyn} and works for $C_0$-semigroup and numerical schemes that are not necessarily contractive. The results of Sections \ref{sec:PBR} and \ref{sec:main} are not used. One should carefully note, however, that inhomogeneities $g$ taking values in $\gamma(H,X_{\nu})$ are considered, where $X_{\nu}$ is a suitable intermediate space between $X$ and $\Dom(A^m)$.
The case of inhomogeneities $g$ taking values in $\gamma(H,X)$ will be considered in Theorem \ref{thm:generalapprox2} and does require contractivity.

\begin{theorem}[Convergence rates without contractivity]\label{thm:generalapprox1}
Let $A$ be the generator of a $C_0$-semigroup $S=(S(t))_{t\geq 0}$ on a $(2,D)$-smooth Banach space $X$ and let $R$ be  an $\calL(X)$-valued scheme approximating $S$ to order $\alpha$ on a Banach space $Y$ continuously embedded in $ X_{\alpha}$ for some $\alpha\in (0,1]$, where $X_{\alpha} := (X,\Dom(A))_{\alpha,\infty}$ if $\alpha\in (0,1)$ and $X_{1} := \Dom(A)$. Let $g\in L_{\bF}^p(\Omega;L^2(0,T;\gamma(H,Y)))$ with $0< p<\infty$, and
let $u_t:=\int_0^t S(t-s) g_s\ud W_s$ for $t\in [0,T]$.  Define, for $n\ge 1$,
\begin{equation}\label{eq:iterationg}
\begin{cases}
u_0^{(n)} & :=0, \\ u_j^{(n)} &:= R(T/n)(u_{j-1}^{(n)} + d_j^{(n)}M), \quad j=1,\dots,n,
\end{cases}
\end{equation}
where $d_j^{(n)}M$ is given by \eqref{eq:defMg}.
Then for all $n\geq 3$,
\begin{align}\label{eq:estungeneralscheme}
\E\sup_{j=0,\ldots,n}\|u_{t_j^{(n)}} - u_j^{(n)}\|^p \le \Bigl(LC_{p,D}\frac{\sqrt{\log (n+1)}}{n^{\alpha}}\Bigr)^p \|g\|_{L^p(\Omega;L^2(0,T;\gamma(H,Y)))}^p,
\end{align}
where $L:=(2 K_{\alpha,Y}C_{S,T} +K)T^{\alpha}$, with $K_{\alpha,Y}$ the norm of the embedding $Y\hookrightarrow X_{\alpha}$, $C_{S,T} := \sup_{t\in [0,T]}\|S(t)\|$, and $K$ the constant in \eqref{eq:convorderalpha}.

If $2\le p<\infty$, the estimate holds with $C_{p,D} = 10D\sqrt{2ep}$.
\end{theorem}

Examples of numerical schemes satisfying the conditions of the theorem can be obtained from Examples \ref{ex:generaldiscr} and \ref{ex:analyticdiscr}. Note that the embedding condition $Y\hookrightarrow X_{\alpha}$ is satisfied for the real interpolation spaces $(X, \Dom(A))_{\alpha,r}$ with $1\le r\le \infty$, the complex interpolation spaces $[X, \Dom(A)]_{\alpha}$ and the fractional domain spaces $\Dom((\mu-A)^\alpha)$ for suitable $\mu\in \varrho(A)$ for all $\alpha\in (0,1)$.

As in Remark \ref{rem:pathwiseconv}, \eqref{eq:estungeneralscheme} implies almost sure pathwise convergence of order $n^{-\beta}$, provided that $\alpha p>1$ and $\beta\in (0,\alpha-\frac1p)$.

\begin{proof}
Let $S_n:[0,T]\to \calL(X)$ be given by
\[S_n(t) := R(T/n)^j, \qquad t\in [t^{(n)}_{j-1}, t^{(n)}_j), \ j=1,\ldots,n.\]
With this notation,
\begin{align*}
u_k^{(n)} & = \sum_{j=1}^k R(T/n)^{k-j+1} d_j^{(n)}M
\\ & = \sum_{j=1}^k \int_{t^{(n)}_{j-1}}^{t^{(n)}_j} S_n(t^{(n)}_k-s) g_s \ud W_s = \int_0^{t^{(n)}_k} S_n(t^{(n)}_k-s) g_s \ud W_s.
\end{align*}
Therefore,
\begin{align*}
u(t^{(n)}_k) - u_k^{(n)} = \int_0^{T} \one_{[0,t^{(n)}_k]}(s) (S(t^{(n)}_k-s) -S_n(t^{(n)}_k-s)) g_s \ud W_s.
\end{align*}
By the bound \eqref{eq:stochasticellinftyn-1} in Proposition \ref{prop:stochasticellinftyn}, for $n\ge 3$ we have
\begin{align*}
\ & \Big(\E\sup_{j=0,\ldots,n}\|u(t^{(n)}_k) - u_k^{(n)}\|^p\Big)^{1/p}
\\ & \leq C_{p,D}\sqrt{\log(n+1)} \|(s,k)\mapsto \one_{[0,t^{(n)}_k]}(s) (S(t^{(n)}_k-s) -S_n(t^{(n)}_k\!-s)) g_s\|_{L^p(\Omega;L^2(0,T;\gamma(H,\ell^\infty_n(X))))}
\\ & \leq C_{p,D} \sqrt{\log(n+1)}\sup_{s\in [0,T]} \|S(s) -S_n(s)\|_{\calL(Y,X)} \|g\|_{L^p(\Omega;L^2(0,T;\gamma(H,Y)))},
\end{align*}
where we may take $C_{p,D} = 10D\sqrt{2ep}$ if $2\le p<\infty$.

By \eqref{eq:normStSs},
for $0\leq s\leq t\leq T$
we have
\[\|S(t)- S(s)\|_{\calL(Y,X)}\leq K_{\alpha,Y} \|S(t)- S(s)\|_{\calL(X_{\alpha},X)} \leq 2 K_{\alpha,Y} C_{S,T} |t-s|^\alpha.\]
Hence from the assumption on the numerical scheme  we conclude that for all $s\in [t^{(n)}_{j-1}, t^{(n)}_j)$,
\begin{align*}
 \|S(s) -S_n(s)\|_{\calL(Y,X)}
& = \|S(s)- S(t^{(n)}_{j})+ S(t^{(n)}_{j})- R(T/n)^j\|_{\calL(Y,X)}
\\ & \leq \|S(s)- S(t^{(n)}_{j})\|_{\calL(Y,X)} + \|S(t^{(n)}_{j})- R(T/n)^j\|_{\calL(Y,X)}
\\ & \leq 2K_{\alpha,Y}C_{S,T} (T/n)^{\alpha} + K (t_j^{(n)}/j)^{\alpha}
\\ & \leq  (2K_{\alpha,Y}C_{S,T} + K)T^{\alpha}  n^{-\alpha}.
\end{align*}
\end{proof}

For $C_0$-semigroups of contractions and contractive discretisation schemes, the next theorem provides uniform convergence in time for inhomogeneities $g$ taking values in $\gamma(H,X)$.

\begin{theorem}[Convergence for contractive schemes]\label{thm:generalapprox2}
Let $A$ be the generator of a $C_0$-contraction semigroup $S=(S(t))_{t\geq 0}$ on a $(2,D)$-smooth Banach space $X$. Let $R$ be an $\calL(X)$-valued contractive scheme approximating $S$ to some order $\alpha\in (0,1]$ on $\Dom(A)$.
Let $g\in L_{\bF}^p(\Omega;L^2(0,T;\gamma(H,X)))$ with $2\leq p<\infty$ and
let
$u_t:= \int_0^t S(t,s) g_s\ud W_s$ for $t\in [0,T]$.  Defining $(u^{(n)}_j)_{j=0}^n$ as in the preceding theorem, we have
\[ \lim_{n\to \infty} \E\sup_{j=0,\ldots,n}\|u_{t_j^{(n)}} - u_j^{(n)}\|^p =0.\]
\end{theorem}
\begin{proof}
Let $\ell_{n+1}^\infty(X):= \bigoplus_{j=0}^n X$ with norm $\|(x_0,\dots,x_n)\|:= \max_{j=0,\ldots,n}\n x_j\n$ and $Z_p^{(n)} := L^p(\Omega;\ell_{n+1}^\infty(X))$.
Let $J,J^{(n)}:L^p_{\bF}(\Omega;L^2(0,T;\gamma(H,X)))\to Z_p^{(n)} $ be the linear operators given by
\[(J g)_j = u_{t_j}^{(n)}, \ \  \text{and} \ \ (J^{(n)} g)_j := u_{j}^{(n)},\qquad j=0,\dots,n.\]
By Theorem \ref{thm:contractionS-new} and Proposition \ref{prop:stability}, the operators
$J$ and $J^{(n)}$ are (uniformly) bounded with $\|J\|\leq C_{p,D}$ and $\|J_n\|\leq K_{p,D}$ respectively, the latter constant being defined as in Proposition \ref{prop:stochasticellinftyn}.

To prove convergence in $Z_p^{(n)}$,
fix $\varepsilon>0$ and let $f\in L^p(\Omega;L^2(0,T;\gamma(H,\Dom(A))))$ be such that $\|g-f\|_{L^p(\Omega;L^2(0,T;\gamma(H,X)))}<\varepsilon$.
By the boundedness and linearity of $J$ and $J^{(n)}$,
\begin{align*}
\|J(g) & - J^{(n)}(g)\|_{Z_p^{(n)}} \\ & \leq \|J(g) - J(f)\|_{Z_p^{(n)}} + \|J(f) - J^{(n)}(f)\|_{Z_p^{(n)}} + \|J^{(n)}(f) - J^{(n)}(g)\|_{Z_p^{(n)}}
\\ & \leq (C_{p,D} + K_{p,D}) \varepsilon + \|J(f) - J^{(n)}(f)\|_{Z_p^{(n)}},
\end{align*}
and the last term tends to zero as $n\to \infty$ by Theorem \ref{thm:generalapprox1}. Since $\varepsilon>0$ was arbitrary the result follows.
\end{proof}

\subsection{Applications to SPDE}\label{subsec:applischemes}
We will now apply the results to some simple examples of stochastic PDE and compare the results with results available in the literature. It goes without saying that with additional work more sophisticated problems can be treated. While this will be taken up in  forthcoming work, the objective here is
to treat some model problems in order to see where our methods can be expected to improve the presently available rates.

We begin with the  stochastic heat equation. The results of the next example can be extended to more general uniformly elliptic operators with space-dependent coefficients. As will follow from Section \ref{sec:random}, if one is only interested in the splitting method the coefficients can even be taken progressively measurable in $(t,\omega)$.

\begin{example}[Stochastic heat equation]\label{ex:stochheat}
Consider the inhomogeneous stochastic heat equation on $\R^d$:
\begin{equation}\label{eq:SEE-heat}
\begin{cases}
\ud u_t &= \Delta u_t + \sum_{k\geq 1} g_{t}^{k}\ud W_t^k, \quad t\in [0,T].
\\ u_0 & = 0.
\end{cases}
\end{equation}
We assume that $g = (g^{k})_{k\geq 1}$ belongs to $L^p_{\bF}(\Omega;L^2(0,T;H^{\lambda,q}(\R^d;\ell^2)))$ with $0<p<\infty$, and $W=(W^k)_{k\geq 1}$ is a sequence of independent standard Brownian motions.
We can view $W$ as an $\ell^2$-cylindrical Brownian motion in a natural way by putting, for $h = (k_k)_{k\ge 1}\in \ell^2$,
$W_t h:= W(\one_{(0,t)\otimes h}) := \sum_{k\ge 1} h_k W_k$, noting that the sum on the right-hand side converges in
$L^2(\Om)$.
As is well known, the operator $\Delta$ generates an analytic $C_0$-semigroup of contractions on the Bessel potential
spaces $H^{\lambda,q}(\R^d)$ and $\Dom(\Delta) = H^{\lambda+2,q}(\R^d)$ for all $\lambda\in \R$ and $1<q<\infty$.

Let us now assume that $2\le q<\infty$. By Theorem \ref{thm:contractionS-new}, the mild solution $u$ to the problem \eqref{eq:SEE-heat} has a continuous modification with values in $H^{\lambda,q}(\R^d)$ which satisfies
\begin{align*}
\E \sup_{t\in [0,T]}\|u_t\|_{H^{\lambda,q}(\R^d)}^p \leq C_{p,q}^p \E\|g\|_{L^2(0,T;H^{\lambda,q}(\R^d;\ell^2))}^p,
\end{align*}
where we may take $C_{p,q} = 10\sqrt{p}(q-1)$ if $2\le p<\infty$.
Here we used that $H^{\lambda,q}(\R^d)$ is $(2,\sqrt{q-1})$-smooth by Proposition \ref{prop:LpX-2mooth} and that
\[\|g_t\|_{\gamma(\ell^2,H^{\lambda,q}(\R^d))}\leq \|g_t\|_{\gamma_q(\ell^2,H^{\lambda,q}(\R^d))} = \|\gamma\|_q \|g_t\|_{H^{\lambda,q}(\R^d;\ell^2)}\]
by H\"older's inequality and \cite[Proposition 9.3.2]{HNVW17}),
where $\gamma$ is a standard Gaussian random variable (whose moments satisfy $\|\gamma\|_q\leq \sqrt{q-1}$).

We consider the approximation scheme \eqref{eq:iterationg} for the
splitting (S), implicit Euler (IE), and Crank-Nicholson (CN) schemes
 discussed in Example \ref{ex:analyticdiscr}. Each of them leads to a sequence of approximate solutions $(u_j^{(n)})_{j=0}^n$, $n\ge 1$, for which we define the approximation errors
\[E_{n,\beta} := \Big(\E\sup_{j=0,\ldots,n}\|u_{t_j^{(n)}} - u_j^{(n)}\|^p_{H^{\lambda-2\beta,q}(\R^d)}\Big)^{1/p}.\]
These numbers also depend on $p,q,\lambda$ and $d$, but the rates in the estimates below will be independent of these parameters. By Theorem \ref{thm:generalapprox2}, $E_{n,0}\to 0$ for (S) and (IE). For $q=2$, (CN) is contractive by Proposition \ref{prop:contractionHilbert} and again we obtain $E_{n,0}\to 0$.
Moreover, we can give rates of convergence for each of these methods. These are given in Table \ref{table:heat}
for the errors $E_{n,\beta}$ with $\beta\in (0,1]$ (up to constants depending on $p,q$). The assertions follow from Example \ref{ex:analyticdiscr}, and Corollary \ref{cor:splittingdecay} and Theorem \ref{thm:generalapprox1} applied with $X = H^{\lambda-2\beta, q}(\R^d)$, $\Dom(\Delta) = H^{\lambda-2\beta+2, q}(\R^d)$ and $Y = H^{\lambda, q}(\R^d) = [X,\Dom(\Delta)]_{\beta}$.
\begin{table}[ht]
\begin{tabular}{ |c|c|c|c| }
\hline
Scheme &  $\beta$ & $q$ & Error $E_{n,\beta}$ \\
\hline
splitting  &  $(0,1]$ & $[2, \infty)$ & $n^{-\beta}$\\
\hline
implicit Euler  & $(0,1]$ & $[2, \infty)$ & $n^{-\beta}(\log(n+1))^{1/2}$ \\
\hline
Crank-Nicholson & $(0,1]$ & $[2, \infty)$ & $n^{-\beta}(\log(n+1))^{1/2}$ \\
\hline
\end{tabular}

\medskip
\caption{Approximation errors for the stochastic heat equation.}\label{table:heat}
\end{table}

Up to a logarithmic term the convergence rates are the same for the three schemes, independently of $p\in (0, \infty)$. Although (S) and (CN) have better orders of convergence, the convergence rate of the approximation errors $E_{n,\beta}$ cannot exceed $\beta$ due to limitations in Corollary \ref{cor:splittingdecay} and Theorem \ref{thm:generalapprox1}.
\end{example}

We next consider a simple non-parabolic equation. Here, higher order schemes give better rates of convergence. Other non-parabolic examples, including wave equation on $\R^d$ (for $q=2$), can be treated similarly.

\begin{example}[Stochastic transport equation]\label{ex:stochtrans}
Consider the following transport equation on $\R$:
\begin{equation}\label{eq:transport}
\begin{cases}
\ud u_t &= \partial_x u_t + \sum_{k\geq 1} g_{t}^{k}\ud W_t^k, \quad t\in [0,T],
\\ u_0 & = 0.
\end{cases}
\end{equation}
Here $g\in L^p_{\bF}(\Omega; L^2 (0,T;H^{\lambda,q}(\R^d;\ell^2)))$ with $0< p<\infty$. It is well known that $\partial_x$ generates a $C_0$-contraction semigroup on $H^{\lambda,q}(\R)$ for all $\lambda\in \R$ and $1\le q<\infty$.

Let us now assume that $2\le q<\infty$. As before, by Theorem \ref{thm:contractionS-new}, the mild solution $u$ to the problem \eqref{eq:transport} has a continuous modification with values in $H^{\lambda}(\R)$ which satisfies
\begin{align*}
\E \sup_{t\in [0,T]}\|u_t|_{H^{\lambda,q}(\R)}^p\leq C_{p,q}^p \E\|g\|_{L^2(0,T;H^{\lambda,q}(\R;\ell^2))}^p,
\end{align*}
where may take $C_{p,q} = 10\sqrt{p}(q-1)$ if $2\le p,\infty$.
As before, for $\beta\geq 0$ let
\[E_{n,\beta} := \Big(\E\sup_{j=0,\ldots,n}\|u_{t_j^{(n)}} - u_j^{(n)}\|^p_{H^{\lambda-\beta,q}(\R)}\Big)^{1/p}.\]
By Theorem \ref{thm:generalapprox2} we have $E_{n,0}\to 0$ for (S) and (IE), and if $q=2$ the same holds for (CN) by Proposition \ref{prop:contractionHilbert}.

Table \ref{table:transport} gives the estimates for the errors $E_{n,\beta}$ for suitable intervals for $\beta$ (up to constants depending on $p,q$). The assertions follow from Example \ref{ex:generaldiscr} (using Proposition \ref{prop:contractionHilbert} for (CN) if $q=2$), Corollary \ref{cor:splittingdecay}, and
Theorem \ref{thm:generalapprox1} applied with $X = H^{\lambda-\beta, q}(\R)$, $\Dom(A^m) = H^{\lambda-\beta+m, q}(\R)$ and $Y = H^{\lambda, q}(\R) = [X,\Dom(A^m)]_{\beta/m}$ for $m=1$ for (S), $m=2$ for (IE), and $m=3$ for (CN). Note that $\phi(8/5) = 1$; since the convergence rate cannot exceed 1, there is no point in considering values $\beta>\frac85$.

\begin{table}[ht]
\begin{tabular}{ |c|c|c|c| }
\hline
Scheme &  $\beta$ & $q$ & Error $E_{n,\beta}$ \\
\hline
splitting &  $(0,1]$ & $[2, \infty)$ & $n^{-\beta}$\\
\hline
implicit Euler  & $(0,2]$ & $[2, \infty)$ & $n^{-\beta/2}(\log(n+1))^{1/2}$ \\
\hline
Crank-Nicholson & $(0,\frac32]$ & $q=2$ & $n^{-2\beta/3}(\log(n+1))^{1/2}$ \\
\hline
Crank-Nicholson & $(\frac12,\frac85]$ & $q\in (2, \infty)$ & $n^{-\phi(\beta)}(\log(n+1))^{1/2}$ \\
\hline
\end{tabular}
\medskip
\caption{Approximation errors for the stochastic transport equation, where $\phi$ is  the piecewise linear function connecting the points $(\frac12, 0)$, $(1,\frac12)$, and $(2,\frac43)$.}\label{table:transport}
\end{table}
\end{example}

Our final example concerns the Schr\"odinger equation.

\begin{example}[Stochastic Schr\"odinger equation]\label{ex:stochschr}
Consider the following heat equation on $\R^d$:
\begin{equation*}
\begin{cases}
\ud u_t &= i\Delta u_t + \sum_{k\geq 1} g_{t}^{k}\ud W_t^k, \quad t\in [0,T].
\\ u_0 & = 0.
\end{cases}
\end{equation*}
We assume that $g\in L^p_{\bF}(\Omega; L^2 (0,T;H^{\lambda}(\R^d;\ell^2)))$ for some $0< p<\infty$, where $H^{\lambda}(\R^d) = H^{\lambda,2}(\R^d)$. It is well known that $i\Delta$ generates a unitary $C_0$-group on $H^{\lambda}(\R^d)$ for all $\lambda\in \R$. As before, by Theorem \ref{thm:contractionS-new}, the mild solution $u$ to the problem \eqref{eq:transport} has a continuous modification with values in $H^{\lambda}(\R^d)$ which satisfies
\begin{align*}
\E\sup_{t\in [0,T]}\|u_t\|_{H^{\lambda}(\R^d)}^p\leq C_p^p\E\|g\|_{L^2(0,T;H^{\lambda}(\R^d;\ell^2))}^p,
\end{align*}
where we may take $C_p = 10 \sqrt{p}$ if $2\le p<\infty$.

As before let
\[E_{n,\beta} := \Big(\E\sup_{j=0,\ldots,n}\|u_{t_j^{(n)}} - u_j^{(n)}\|^p_{H^{\lambda-2\beta}(\R^d)}\Big)^{1/p}.\]
By Theorem \ref{thm:generalapprox2}, $E_{n,0}\to 0$ for (S), (IE), and (CN) (using Proposition \ref{prop:contractionHilbert} for the latter).

Table \ref{table:schrodinger} gives the estimates for the errors $E_{n,\beta}$ (up to constants depending on $p$) for suitable intervals for $\beta$. The assertions follow from Example \ref{ex:generaldiscr}, Corollary \ref{cor:splittingdecay}, and
Theorem \ref{thm:generalapprox1} applied with $X = H^{\lambda-2\beta}(\R^d)$ and $Y = H^{\lambda, q}(\R^d) = [X,\Dom(A^m)]_{\beta/m}$ for $m=1$ for (S), $m=2$ for (IE), and $m=3$ for (CN).

\begin{table}[ht]
\begin{tabular}{ |c|c|c| }
\hline
Scheme &  $\beta$ &  Error $E_{n,\beta}$ \\
\hline
splitting &  $(0,1]$ & $n^{-\beta}$\\
\hline
implicit Euler  & $(0,2]$  & $n^{-\beta/2}(\log(n+1))^{1/2}$ \\
\hline
Crank-Nicholson & $(0,\frac32]$ & $n^{-2\beta/3}(\log(n+1))^{1/2}$ \\
\hline
\end{tabular}

\medskip
\caption{Approximation errors for the stochastic Schr\"odinger equation.}\label{table:schrodinger}
\end{table}
\end{example}

We are aware of only few papers dealing with convergence uniformly in time in infinite dimensions.
In \cite{GyKr03}
the splitting method is considered for (possibly degenerate) parabolic problems with gradient noise.
The inhomogeneities have to be uniformly bounded in time. The same methods are considered in \cite{CoxNee13} for semi-linear stochastic parabolic problems. No contractivity of the semigroups needs to be assumed and convergence in H\"older norms is obtained under $L^p$-integrability conditions in time with $p>2$. See Table \ref{tab:parabolic} for a comparison of the convergence rates.

In \cite{CoxNee13} (in the setting of UMD spaces) and \cite{GyMi07} (in the setting of monotone operators on Gelfand triples $V\hookrightarrow X\hookrightarrow V^*$), the implicit Euler scheme was considered with uniform convergence in time, but these results seem not to be comparable to ours due to the fact that an additional discretisation of the noise term is allowed. In the latter reference, convergence rates of order $n^{-\nu}$ are obtained under the assumption that the solution $u$ belong to $C^{\nu}([0,T];L^2(\Omega;V)) \cap L^2(\Omega;L^\infty(0,T;V))$.
Results on uniform convergence in time (and sometimes even convergence in H\"older norms in time) for schemes involving space and time discretisation can be found in many papers, including \cite{CoxHau12,CoxHau13,CHNJW,Gyo99,GyMi09,Yoo00,Jent09,PeSi05}. Results concerning uniform convergence in case of white noise and discretisation in time only can be found in \cite{BCH, BG, GyNu95, GyNu97}. Some results are with explicit rates and some are not, but the schemes considered in these papers are different.

In the parabolic setting, results on convergence of the form
\begin{equation}\label{eq:pointwiseconvscheme}
\sup_{j=0,\ldots,n} \E \|u(t_j^{(n)}) - u_j^{(n)}\|^p\to 0
\end{equation}
(notice the reversed order of supremum and expectation) with explicit rates, which can even be faster than $1/n$, can be found in \cite{CoxNee10,JenKlo,Lord} and references therein.

\begin{table}[ht]
\begin{tabular}{ |c|c|c|c|c| }
\hline
paper & Scheme &  $\beta$ & $g\in L^r$ in time & Error $E_{n}$ \\
\hline
present & splitting &  $(0,1]$ & $r=2$ & $n^{-\beta}$\\
\hline
\cite{GyKr03} & splitting &  $2$ & $r=\infty$ & $n^{-1}$\\
\hline
\cite{CoxNee13} & splitting  & $(-\frac12,\frac12)$ & $r>2$  & $n^{-\frac12-\beta+\frac1r+\varepsilon}$ \\
\hline
\end{tabular}

\medskip
\caption{Comparison of rates in the parabolic setting.}\label{tab:parabolic}
\end{table}
For non-parabolic problems no systematic results seem to be available on uniform convergence in time. In \cite{Wang} uniform convergence with explicit rates has been obtained for a nonlinear wave equation with the splitting scheme. The fact that the underlying semigroup is a group allows us to write
\[\int_0^t S(t-s) g_s d W_s = S(t)\int_0^t S(-s) g_s d W_s\]
and uniform convergence can be obtained from standard maximal estimates for martingales. In \cite{FiTaTe} the authors obtain uniform convergence results in case the semigroup admits a dilation to a group.
Our results do not rely on the above identity and therefore are applicable in the case of arbitrary contractive $C_0$-semigroups, and the convergence holds with the same rate. Even more is true: for arbitrary $C_0$-semigroups and general numerical schemes the same convergence rates can be obtained up to a logarithmic factor.

\section{Maximal inequalities for random stochastic convolutions}\label{sec:random}

In this section we consider the time-dependent problem
\begin{equation}\label{eq:SEE-random}
    \begin{cases}
        \ud u_t &= A(t)u_t\ud t + g_t \ud W_t, \qquad t\in [0,T], \\
        u_0 &= 0.,
    \end{cases}
\end{equation}
with {\em random} operators $A(t)$. More precisely we assume that $(A(t,\om))_{(t,\om)\in [0,T]\times \Om}$ is an adapted family of closed operators acting in $X$ which satisfy suitable conditions, to be made precise below, guaranteeing the generation of an adapted evolution family. We will assume throughout that $W$ is an adapted $H$-cylindrical Brownian motion on $\Om$. and that
$g:[0,T]\times \Om\to \gamma(H,X)$ is progressively measurable; recall that this is equivalent to the requirement that $g(h): [0,T]\times \Om\to X$ is progressively measurable for all $h\in H$. Many of the results of this section are expected to extend to more general martingales.

\subsection{The forward stochastic integral}
In analogy with the non-random case one expects that \eqref{eq:SEE-random} admits a mild solution given as before by the stochastic convolution process $\int_0^t S(t,s)g_s\ud W_s$. This stochastic integral, however, cannot be defined as an It\^o stochastic integral because the random variables $S(t,s)x$ are only assumed to be $\F_t$-measurable rather than $\F_s$-measurable and consequently the integrand will not be progressively measurable in general.

To overcome this problem we use the forward stochastic integral, introduced and studied by Russo and Vallois \cite{RussoVallois} in the scalar-valued setting. Following \cite{LeonNual,ProVer14,ProVer15} we define its vector-valued analogue as follows.
Fix an orthonormal basis $(h_k)_{k\geq 1}$ of $H$. For processes $\Phi\in L^0(\Om;L^2(0,T;\gamma(H,X)))$ and $n=1,2,\dots$ define
\begin{align*}
I^-(\Phi,n) := n\sum_{k=1}^n \int_0^T \Phi_s h_k (W_{(s+1)/n} - W_s) h_k\ud s.
\end{align*}
The process $\Phi$ is {\em forward stochastically integrable} if the sequence $(I^-(\Phi,n))_{n\geq 1}$ converges in probability.
If this is the case, the limit is independent of the choice of orthonormal basis and is called the {\em forward stochastic integral} of $\Phi$. We write
\[\int_0^T \Phi_s\ud W_s^- := I^{-}(\Phi) := \limn I^-(\Phi,n).\]
Notice that $\Phi$ is not assumed to be progressively measurable. It is easy to see that if $\Phi$
is a finite rank step process, then $\Phi$ is forward integrable. In case $\Phi$ is progressively measurable and integrable in the It\^o sense, then the forward stochastic integral exists and coincides with the It\^o integral (see \cite[Proposition 3.2]{ProVer15}).

In order to apply the forward integral to our problem we make following Hypothesis:

\begin{hypothesis}\label{hyp:condSadapted}
The family $(S(t,s,\om))_{0\le s\le t\le T,\,\om\in\Om}$ is an adapted $C_0$-evolution family of contractions on $X$, i.e.,
\begin{enumerate}
 \item[\rm(i)]\label{it:condSadapted1} $(S(t,s,\om)_{0\le s\le t\le T}$ is a $C_0$-evolution family of contractions for every $\om\in \Om$;
 \item[\rm(ii)]\label{it:condSadapted2} $S(t,s,\cdot)x$ is strongly $\F_t$-measurable for all $0\leq s\leq t\leq T$ and $x\in X$.
\end{enumerate}
Furthermore we assume:
\begin{enumerate}
\item[\rm(iii)]\label{it:condSadapted3} $Y$ is a Banach space, continuously embedded in $X$, and for almost all $\om\in \Om$ we have $S(t,\cdot,\om)y \in W^{1,1}(0,t;X)$ for all $t\in (0,T]$ and $y\in Y$ and
\[\|(S(t,\cdot,\om)y)\|_{W^{1,1}(0,t;X)}\leq C(\om)\|y\|_Y\]
for some function $C:\Omega\to [0,\infty)$ independent of $y\in Y$ and $0 \leq s\leq t\leq T$.
\end{enumerate}
\end{hypothesis}

We have the following sufficient condition for forward integrability (see \cite[Corollary 5.3]{ProVer15}, which extends to the current setting).

\begin{proposition}\label{prop:forwardequiv}
Suppose that Hypothesis \ref{hyp:condSadapted} holds, with $X$ a $2$-smooth Banach space, and let $g:[0,T]\times\Om\to\gamma(H,Y)$ be a finite rank adapted step process. Then process $(S(t,s)g_s)_{s\in [0,t]}$ is forward integrable on $[0,t]$ and almost surely we have
\begin{align}\label{eq:pathwisemild}
\int_0^t S(t,s)g_s\ud W_s^- =  S(t,0)\int_0^t g_s \ud W_s + \int_0^t \partial_s S(t,s) \int_s^t g_r \ud W_r \ud s.
\end{align}
Moreover, the process $(\int_0^t S(t,s)g_s\ud W_s^-)_{t\in [0,T]}$ has a continuous modification.
\end{proposition}
The right-hand side of \eqref{eq:pathwisemild} is well defined by the hypothesis and the assumption that $g$ takes values in $Y$. By the almost sure pathwise continuity of $\int_0^\cdot g_s \ud W_s$, the forward integral in \eqref{eq:pathwisemild} admits a continuous modification.

\begin{remark}
In the setting where $S$ is generated by an adapted family $(A(t))_{t\in [0,T]}$ satisfying suitable parabolicity assumptions, the right-hand side of \eqref{eq:pathwisemild} is called the {\em pathwise mild solution} of \eqref{eq:SEE-random}. Pathwise mild solutions were introduced and extensively studied in \cite{ProVer14}. In the parabolic case, $\partial_s S(t,s)$ typically extends to a bounded operator on $X$ and $\|\partial_s S(t,s)\|\leq C(t-s)^{-1}$, where $C$ depends on $\omega\in \Omega$.  Since $\int_0^\cdot g_r \ud W_r$ is almost surely H\"older continuous under $L^p(0,T)$-integrability assumptions on $g$ with $p>2$, the right-hand side of \eqref{eq:pathwisemild} exists pathwise as a Bochner integral.

It is quite difficult to prove estimates for the forward integral directly. A major advantage of using the right-hand side of \eqref{eq:pathwisemild} is that one can obtain estimates using only It\^o and Bochner integrals.
\end{remark}

\subsection{The maximal inequality}
We will now extend the maximal estimate of Theorem \ref{thm:contractionS-new} to random evolution families, replacing the It\^o stochastic integral of that theorem by the forward stochastic integral. The precise sense in which the forward integral constitutes a solution of the problem \eqref{eq:SEE-random} will be addressed subsequently in Theorem \ref{thm:contractionS-adaptedSDE}.
Even without the supremum on the left-hand side, the estimate in Theorem \ref{thm:contractionS-adapted} is new.

\begin{theorem}\label{thm:contractionS-adapted}
Suppose that Hypothesis \ref{hyp:condSadapted} holds,with $X$ a $2$-smooth Banach space,  and let $g:[0,T]\times\Om\to\gamma(H,Y)$ be a finite rank adapted step process. Then for all $0<p<\infty$ we have
\begin{align*}
\E\sup_{t\in [0,T]}\Big\|\int_0^t S(t,s) g_s \ud W_s^-\Big\|^p \leq C_{p,D}^p \|g\|_{L^p(\Omega;L^2(0,T;\gamma(H,X)))}^p,
\end{align*}
where the constant $C_{p,D}$ only depends on $p$ and $D$. For $2\le p<\infty$ the inequality holds with $C_{p,D} = 10D\sqrt{p}$.
\end{theorem}
\begin{proof}

The proof is similar to that of Theorem \ref{thm:contractionS-new}, but with some extra technicalities which justify a detailed presentation.

\smallskip
{\em Step 1}. \
Let $g:[0,T]\times\Om\to \gamma(H,X)$ be an adapted finite rank step process, say
\begin{equation*}
g  = \sum_{j=1}^{k} \one_{(s_{j-1},s_{j}]} \sum_{i=1}^\ell h_i\otimes \xi_{ij}
\end{equation*}
as in \eqref{eq:simple}. For the moment there is no need to insist that $g$ be $Y$-valued; this will only be needed in the last step of the proof.

Fix $0<\delta<T$ and set $S^{\delta}(t,s) := S((t-\delta)^+,(s-\delta)^+)$ for $0\leq s\leq t\leq T$.
Fix a partition $\pi:= \{r_0,\dots,r_N\}$, where $0= r_0<r_1<\ldots <r_N=T$, and
let $(K(t,s,\omega))_{0\leq s\leq t\leq T,\,\omega\in \Omega}$ be a family of contractions on $X$ with the following properties:
\begin{enumerate}[\rm(i)]
 \item\label{eq:Kgd1} $K(t,\cdot,\om)$ is constant on $[r_{j-1},r_j)$ for all $t\in [0,T]$, $\om\in\Om$, and $j=1, \ldots, N$;
 \item\label{eq:Kgd2} $K(\cdot, s,\om)$ is strongly continuous for all $s\in [0,T]$ and $\om\in\Om$;
 \item\label{eq:Kgd3} $S^\delta(t,r,\om) K(r,s,\om) = K(t,s,\om)$  for all $0\leq r\leq s\leq t\leq T$ and $\om\in\Om$;
 \item\label{eq:Kgd4} $K(t,s,\cdot)x$ is strongly $\F_{(t-\delta)^+}$-measurable for all $0\le s\leq t\leq T$.
\end{enumerate} By refining $\pi$ we may assume that $|r_j-r_{j-1}|\leq \delta$ for $j=1,\ldots,N$ and that
$s_j\in \pi$ for all $j=0,\ldots,k$.

Define the process $(v_t)_{t\in [0,T]}$ by
\begin{align}\label{eq:K}
v_t := \int_0^t K(t,s) g_s \ud W_s^{-},
\end{align}
this forward integral being well defined since the integrand is a finite rank step process.
For $t\in [0,r_1]$ the above integral coincides with the It\^o integral since $K(t,s,\cdot)$ is strongly $\F_{0}$-measurable.
By \eqref{eq:Kgd3}, for $r_{j-1} \le s\le t< r_j$ we have
\begin{align}\label{eq:splitSdelta}
v_t = S^{\delta}(t,s) v_{s} + \int_{s}^{t} K(t,r) g_r \ud W_r,
\end{align}
where the stochastic integral is again an It\^o integral since the random variable $K(t,r,\cdot)$ does not depend on $r\in [s,t]\subseteq [r_{j-1},r_j)$ by \eqref{eq:Kgd1} and is strongly $\F_{r_{j-1}}$-measurable by \eqref{eq:Kgd4}
and the inclusion $\F_{(t-\delta)^+}\subseteq \F_{r_{j-1}}$ (using that $(t-\delta)^+\le r_{j-1}$).
Properties \eqref{eq:Kgd1} and \eqref{eq:Kgd2} imply that $v$ has a modification with continuous paths. Working with such a modification, we will first prove that for all $2\le p<\infty$ one has
\begin{align*}
\Big\|\sup_{t\in [0,T]}\n v_t\|\Big\|_p \leq 10D\sqrt{p}\|g\|_{L^p(\Omega;L^2(0,T;\gamma(H,X)))}.
\end{align*}
By a limiting argument it suffices to consider exponents $2<p<\infty$.

Let $\pi' = \{t_0,t_1, \ldots, t_m\}\subseteq [0,T]$ be another partition. It suffices to prove
\begin{align}\label{eq:toprovemaximalineqAadapted}
\Big\|\sup_{t\in \pi'} \n v_t \n\Big\| \le a_\pi + 10D\sqrt{p}
\|g\|_{L^p(\Omega;L^2(0,T;\gamma(H,X)))}
\end{align}
with $a_\pi = o(\hbox{mesh}(\pi))$ as mesh$(\pi)\to 0$.
Refining $\pi'$ if necessary, we may assume that $\pi'\subseteq \pi$ and that mesh$(\pi')<\delta$.

For fixed $j=1,\ldots,m$ we have, by \eqref{eq:splitSdelta},
\begin{align*}
f_{j} := v_{t_j} & = S^{\delta}(t_j, t_{j-1}) v_{t_{j-1}} + \int_{t_{j-1}}^{t_j} K(t_j,s) g_s \ud W_s
 \\ & =: V_{j} f_{j-1} + dG_j,
\end{align*}
where we set $V_{j} := S^{\delta}(t_j, t_{j-1})$ and $dG_j:= \int_{t_{j-1}}^{t_j} K(t_j,s) g_s \ud W_s$. We further set $f_0:=0$
and $G_0:=0$. As in the proof of Theorem \ref{thm:contractionS-new} the sequence $(dG_j)_{j=1}^m$ is conditionally symmetric
and an application of Theorem \ref{thm:Pinelis} gives
\begin{align*}
\n f^\star\n_p \le 5p \n dG\ss\n_p + 10D\sqrt{p} \n s(G)\n_p.
\end{align*}
Proceeding as in Step 1b of the proof of Theorem \ref{thm:contractionS-new} we obtain \eqref{eq:toprovemaximalineqAadapted}.

\smallskip
{\em Step 2}. \
Fix $n\in \N$ and set $\sigma_n(s) := j 2^{-n}T$ for $s\in [j2^{-n}T, (j+1)2^{-n}T)$. Set $S_n^{\delta}(t,s) := S((t-\delta)^+,\sigma_n((s-\delta)^+))$ and define
$v^{(n)}_t$ as in \eqref{eq:K} with $K(t,s) = S_n^\delta(t,s)$.
The assumptions  \eqref{eq:Kgd1}-- \eqref{eq:Kgd4} in Step 1 apply to $K(t,s) = S_n^\delta(t,s)$, $N = 2^n$, and $r_j = j2^{-n}T$. By what has been shown in this step, the process $v_t$ has a continuous modification. Moreover, for $n\geq m$ the process
\[ v^{(n)}_t - v^{(m)}_t =
 S^{\delta}(t,s) (v_{s}^{(n)}- v_{s}^{(m)}) + \int_{s}^{t} K(t,r)
(I- S(\sigma_n(r),\sigma_m(r))) g_r\ud W_r\]
is strongly progressively measurable. Moreover,
\begin{align*}
\ & \Big\|\sup_{t\in [0,T]} \n v^{(n)}-v^{(m)}\n\Big\|_p
\\ & \qquad\qquad \leq 10D\sqrt{p} \big\|(I- S(\sigma_n((\cdot-\delta)^+),\sigma_m((\cdot-\delta)^+)))g\big\|_{L^p(\Omega;L^2(0,T;\gamma(H,X)))}.
 \end{align*}
Since the right-hand side
tends to $0$ by dominated convergence, $(v^{(n)})_{n\geq 1}$ is a Cauchy sequence with respect to the norm of  $L^p(\Omega;C([0,T];X))$ and hence converges to some $\wt v^{\delta} \in L^p(\Omega;C([0,T];X))$. By Step 1,
\begin{align}\label{eq:wtvdelta}
\Big\|\sup_{t\in [0,T]} \n \wt{v}_t^{\delta}\| \Big\|_{p} = \lim_{n\to \infty} \Big\|\sup_{t\in [0,T]} \n v^{(n)}_t\|\Big\|_{p}\leq 10D\sqrt{p} \|g\|_{L^p(\Omega;L^2(0,T;\gamma(H,X)))}.
\end{align}

We will show next that $ \wt v_t^{\delta}=\int_0^t S^{\delta}(t,s) g_s \ud W_s^{-}$ almost surely for each $t\in [0,T]$. To this end let $\pi'' = \{t_0, \ldots, t_M\}$ with $0=t_0<\ldots<t_M=T$ with mesh$(\pi'')<\delta$. We define an $X$-valued process $(v_t^{\delta})_{t\in [0,T]}$ by setting $v^{\delta}_0 := 0$ and, recursively,
\[v_t^{\delta} = S^{\delta}(t,t_{j-1}) v_{t_{j-1}} + \int_{t_{j-1}}^{t} S^{\delta}(t,s) g_s \ud W_s, \qquad t\in (t_{j-1},t_{j}].\]
The stochastic integral is well defined since for all $t_{j-1}\leq s\leq t\leq t_j$ the random variable $S^{\delta}(t,s) = S((t-\delta)^+, (s-\delta)^+)$ is strongly $\F_{t_{j-1}}$-measurable. Using the elementary properties of forward integrals
we can rewrite this definition as the forward integral
\begin{align}\label{eq:recursive} v^{\delta}(t) = \int_0^t S^{\delta}(t,s) g_s \ud W_s^{-}, \qquad t\in [0,T].
\end{align}
We claim that for each $t\in [0,T]$ we have $v^{\delta}(t) = \wt v^{\delta}(t)$ almost surely. Indeed, by \eqref{eq:Neid},
\begin{align*}
\Big\|\int_{t_{j-1}}^{t} S^{\delta}_n(t,s) g_s & \ud W_s  - \int_{t_{j-1}}^{t} S^{\delta}(t,s) g_s \ud W_s\Big\|_{L^2(\Omega;X)} \\ & \leq D\|(S^{\delta}_n(t,s)- S^{\delta}(t,s))g_s\|_{L^2(\Omega;L^2(0,t;\gamma(H,X)))}\to 0
\end{align*}
as $n\to \infty$ by dominated convergence. Therefore, the terms in the recursive identities \eqref{eq:recursive} converge to the correct limit and the claim is proved.

\smallskip
{\em Step 3}. \ We will next show that
$$ \lim_{\delta\downarrow 0} \int_0^t S^{\delta}(t,s) g_s  \ud W_s^- = \int_0^t S(t,s) g_s  \ud W_s^-$$
in $L^0(\Om;C([0,T];X))$.
This will be done by providing an alternative formula for $\int_0^t S^{\delta}(t,s) g_s \ud W_s^{-}$ in which we can let $\delta\downarrow 0$.
Here it will be important that $g$ takes values in $Y$.

Fix $t\in (0,T]$.
Since $\|\partial_s (S(t,s)y)\|_X\leq C\|y\|_Y$ with a constant $C$ independent of $0<s<t\leq T$, it follows from Proposition \ref{prop:forwardequiv} that the forward stochastic convolution integral $u_t := \int_0^t S(t,s) g_s \ud W_s^{-} $ exists and is almost surely equal to
\[ S(t,0) \int_0^t g_{r} \ud W_r + \int_0^t \partial_s S(t,s) \int_s^t g_r \ud W_r \ud s.\]
Similarly,
\begin{align*}
v^{\delta}(t)& = S((t-\delta)^+,0) \int_0^t g_{r} \ud W_r + \int_0^t \partial_s S((t-\delta)^+,(s-\delta)^+) \int_s^t g_r \ud W_r \ud s
\\ & = S((t-\delta)^+,0) \int_0^t g_{r} \ud W_r + \int_{0}^{(t-\delta)^+} \partial_s S((t-\delta)^+,s) \int_{s+\delta}^t g_r \ud W_r \ud s.
\end{align*}
Letting $\delta\downarrow 0$, by the piecewise strong continuity of $t\mapsto \partial_s S(t,s)$ on $Y$ and dominated convergence we obtain that $v^{\delta}(t)\to u(t)$ almost surely.

By dominated convergence one also obtains that $u$ has a continuous modification. To prove the maximal estimate for this modification it suffices to show that for any finite set $\pi\in [0,T]$,
\begin{align*}
\Big\|\sup_{t\in \pi} \n u_t \n\Big\| \le 10D\sqrt{p}
\|g\|_{L^p(\Omega;L^2(0,T;\gamma(H,X)))}.
\end{align*}
Using that \eqref{eq:wtvdelta} and $v^{\delta}(t) = \wt{v}^{\delta}(t)$ for $t\in \pi$, this follows from Fatou's lemma:
\begin{align*}
\Big\|\sup_{t\in \pi} \n u_t \n\Big\|_p \leq \liminf_{\delta\downarrow 0}\Big\|\sup_{t\in \pi} \n v_t^{\delta} \n\Big\|  \le 10D\sqrt{p}
\|g\|_{L^p(\Omega;L^2(0,T;\gamma(H,X)))}.
\end{align*}

{\em Step 4}. \ The case $0<p<2$ follows again by using Corollary \ref{cor:Pinelis} instead of Theorem \ref{thm:Pinelis}, or by an extrapolation argument involving Lenglart's inequality.
\end{proof}

If the embedding $Y\hookrightarrow X$ is dense we can use the maximal inequality of the theorem to see that for all $0<p<\infty$ the mapping
$$ g\mapsto \int_0^t S(t,s) g_s  \ud W_s^-$$
has a unique extension to a continuous linear operator
$$ J_p: L_{\bF}^p(\Om;L^2(0,T;\gamma(H,X)))\to L^p(\Om;C([0,T];X)).$$
Moreover, by a standard localisation argument,
$J$ has a unique extension to a continuous linear operator
$$ J: L_{\bF}^0(\Om;L^2(0,T;\gamma(H,X)))\to L^0(\Om;C([0,T];X)).$$
It is not guaranteed, however, that for general $g\in L_{\bF}^0(\Om;L^2(0,T;\gamma(H,X)))$ the process $Jg$ is given by
a forward stochastic convolution again, nor is this clear if we replace $L^0$ and $J$ by $L^p$ and $J_p$. The same problem occurs if we use the right-hand side in the identity in Proposition \ref{prop:forwardequiv}.

Since $J_p$ satisfies the same estimate as in Theorem \ref{thm:contractionS-adapted}, we immediately obtain an extension of the exponential tail estimate of Corollary \ref{cor:expontail} in the current setting. As in Remark \ref{rem:Itoformexp} under more restrictive conditions on the random evolution family, but with better bound on the variance $\sigma^2$ a similar result was obtained in \cite[Remark 5.8]{NV20a}.

The next theorem addresses the question in what sense $J_pg$ and $Jg$ ``solve'' the problem \eqref{eq:SEE-random}. Some additional assumptions are needed to establish the precise relation between the random evolution family $S$ and the random operator $A$.

\begin{hypothesis}\label{hyp:condSadapted-randomgenerator} Hypothesis \ref{hyp:condSadapted} is satisfied. Furthermore, the random operator family $A:[0,T]\times \Omega\to \calL(Y,X)$ has the property that $Ay$ is strongly progressively measurable for all $y\in Y$.
Furthermore the following conditions hold:
\begin{enumerate}[\rm(i)]
\item\label{it:condSadaptedgenerator3b}
For almost all $\om\in \Om$ we have $S(t,\cdot,\om)y \in W^{1,1}(0,t;X)$ for all $t\in [0,T]$ and $y\in Y$, and for almost all  $s\in [0,t]$ we have $\partial_s S(t,s)y=-S(t,s)A(s) y$ and
\[\|S(t,s)A(s)y\|_{X}\leq C\|y\|_Y,\]
where $C:\Omega\to [0,\infty)$ is independent of $y\in Y$ and $0\leq s<t\leq T$.

\item\label{it:condSadaptedgenerator3a}
For almost all $\om\in \Om$ we have $S(\cdot,s,\om)y \in W^{1,1}(s,T;X)$ for all $s\in [0,T]$ and $y\in Y$, and for almost all  $t\in [s,T]$ we have $\partial_t S(t,s)y=A(t)S(t,s) y$ and
\[\|A(t)S(t,s)y\|_{X}\leq C\|y\|_Y,\]
where $C:\Omega\to [0,\infty)$ is independent of $y\in Y$ and $0\leq s<t\leq T$.

\item\label{it:condSadaptedgenerator4} There exists a dense subspace $F\subseteq X^*$ such that $F\subseteq \Dom(A(t,\omega)^*)$
    for all $(t,\omega)\in [0,T]\times\Omega$, and
    almost surely the mapping $(t,\omega)\mapsto \lb x, A(t,\omega)^*x^*\rb$ belongs to $L^\infty(0,T)$ for all $x\in X$ and $x^*\in F$.
\end{enumerate}
\end{hypothesis}

In the proof below we will combine \eqref{it:condSadaptedgenerator4} with the observation that if $f:(0,T)\to X$ is integrable
and $g:(0,T)\to X^*$ has the property that $\lb x,g\rb\in L^\infty(0,T)$ for all $x\in X$, then the function  $t\mapsto\lb f(t),g(t)\rb$ is integrable and
$$\int_0^T |\lb f(t),g(t)\rb|\le \n f\n_1 \sup_{\n x^*\n\le 1} \n \lb x,g\rb\n_\infty,$$ the supremum on the right-hand side being finite by a closed graph argument. Indeed, this estimate is clear for simple functions $f$ and the general case follows by approximation.

Under the above hypothesis a process $u\in L_{\bF}^0(\Omega;L^1(0,T;X))$ is called a {\em weak solution} of \eqref{eq:SEE-random} if for all $x^*\in F$, a.s. for all $t\in [0,T]$,
\[\lb u_t,x^*\rb = \int_0^t \lb  u_s, A(s)^* x^*\rb \ud s +\int_0^t  g_s^* x^* \ud W_{s}.\]
In many situations weak solutions are known to be unique. However, we will not address this issue here.

\begin{theorem}\label{thm:contractionS-adaptedSDE}
Suppose that Hypothesis \ref{hyp:condSadapted-randomgenerator} holds, with $X$ a $2$-smooth Banach space, and assume in addition that the embedding $Y\hookrightarrow X$ is dense. Then for every $g\in L_{\bF}^0(\Omega;L^2(0,T;\gamma(H,X)))$ the process $Jg$ is a weak solution to \eqref{eq:SEE-random}.
\end{theorem}
\begin{proof}
We proceed in three steps.

\smallskip
{\em Step 1}. \  First let $g:[0,T]\times \Omega\to \calL(H,Y)$ be an adapted finite rank step processes and write $v^g_t = \int_0^t g \ud W$.
From Proposition \ref{prop:forwardequiv}, Theorem \ref{thm:contractionS-adapted} and Hypothesis \ref{hyp:condSadapted-randomgenerator}\eqref{it:condSadaptedgenerator3b} it is immediate that
\begin{align}\label{eq:maxestadptedcase}
\E\sup_{t\in [0,T]}\|u^g_t\|^p \leq C_{p,D}^p \|g\|_{L^p(\Omega;L^2(0,T;\gamma(H,X)))}^p,
\end{align}
and
\[u_t^g = S(t,0) v_t^g - \int_0^t S(t,s)A(s) (v_t^g - v_s^g) \ud s.\]
We check next that $u^g$ is a weak solution. For this we use a variation of the argument in \cite[Theorem 4.9]{ProVer14}.
For all $x\in Y$,
\begin{align}\label{eq:propSA}
\int_0^t S(t,s) A(s) x \ud s= -x + S(t,0)x \ \ \ \text{and} \ \  \int_r^t A(s) S(s,r) x \ud r= S(t,r)x - x
\end{align}
Therefore, applying the first part of \eqref{eq:propSA} with $x = v_t^g$, we obtain
\begin{align}\label{eq:ugwithoutSterm}
u_t^g = v_t^g + \int_0^t S(t,s)A(s) v_s^g \ud s.
\end{align}
To claim that $u^g$ is a weak solution it remains to check that
\[\left<\int_0^t S(t,s)A(s) v_s^g \ud s, x^*\right> = \int_0^t \lb u^g_s, A(s)^*x^*\rb \ud s.\]
Note that the integral on the right-hand side is well defined as a Lebesgue integral almost surely.
To prove the claim we note that by  \eqref{eq:ugwithoutSterm}, Fubini's theorem and the second part of \eqref{eq:propSA} (or rather, its weak version $\int_r^t \lb S(s,r) x, A^*(s) x^*\rb  \ud r = \lb S(t,r) x, x^*\rb  - \lb x, x^*\rb$, the point being that in the argument below the vector $x=A(t) v^g_r$ need not belong to $Y$),
\begin{align*}
  \int_0^t \lb &u^g_s, A(s)^*x^*\rb \ud s \\ & = \int_0^t \lb v_s^g,A(s)^* x^*\rb \ud s  + \int_0^t \int_0^s \lb S(s,r)A(r) v_r^g,A(s)^* x^*\rb \ud r \ud s
\\ & = \int_0^t \lb v_s^g,A(s)^* x^*\rb \ud s  + \int_0^t \int_r^t \lb S(s,r)A(r) v_r^g,A(s)^* x^*\rb \ud s\ud r
\\ & = \int_0^t \lb v_s^g,A(s)^* x^*\rb \ud s  + \int_0^t \lb S(t,r)A(r) v_r^g,x^*\rb \ud r - \int_0^t \lb A(r) v_r^g,x^*\rb \ud r
\\ & = \int_0^t \lb S(t,r)A(r) v_r^g,x^*\rb \ud r,
\end{align*}
which gives the required identity.

\smallskip
{\em Step 2}. \  Let  $g\in L^p_{\mathscr{P}}(\Omega;L^2(0,T;\gamma(H,X)))$ with $0<p<\infty$ and choose a sequence of $Y$-valued adapted finite rank step processes $(g^{(n)})_{n\geq 1}$ such that $g^{(n)}\to g$ in $L^p(\Omega;L^2(0,T;\gamma(H,X)))$. Then from \eqref{eq:maxestadptedcase} applied to $g^{(n)}-g^m$ we obtain that $(u^{g^{(n)}})_{n\geq 1}$ is a Cauchy sequence and therefore converges to some $u$ in $L^p(\Omega;C([0,T];X))$. By Step 1, $u^{g^{(n)}}$ is a weak solution and thus
\[\lb u_t^{g^{(n)}},x^*\rb = \int_0^t \lb  u_s^{g^{(n)}}, A(s)^* x^*\rb \ud s +\int_0^t  (g_s^{(n)})^* x^* \ud W_{s}.\]
Letting $n\to \infty$ in this identity we conclude that $u^g$ is a weak solution. The maximal inequality is obtained by applying \eqref{eq:maxestadptedcase} with $g_n$ and letting $n\to \infty$.
\end{proof}

\begin{remark}
In \cite[Proposition 5.3]{LeonNual}, restrictive conditions in terms of Malliavin differentiability of $S$ are given under which the forward stochastic integral $u_t = \int_0^t S(t,s) g_s  \ud W_s^-$ exists, has a continuous modification, and is a weak solution. Inspection of the proof shows that that if one sets $u_t^{(n)} := I^{-}(\one_{[0,t]} S(t,\cdot) g^{(n)})$, one needs that $\sup_{t\in [0,T]}\|u_t - u_t^{(n)}\|_{L^1(\Omega;X)} \to 0$. Although this is likely to hold in many situations, such considerations can be avoided by using the right-hand side of \eqref{eq:pathwisemild}.
\end{remark}

\begin{remark}\label{rem:splitting}
Theorem \ref{thm:splitting} extend {\em mutatis mutandis} to random evolution families. The only required change is to use the forward integral in the proof and to apply Theorem \ref{thm:contractionS-adapted} instead of Theorem \ref{thm:contractionS-new}. To obtain explicit decay rates under the assumption that $g$ has spatial smoothness, i.e., $g$ takes values in a Banach space $Y$ continuously embedded in $X$, one requires estimates for $\|S(s,\sigma_n(s)) - I\|_{\calL(Y,X)}$. In some applications (e.g. \cite[Section 5.2]{Pazy}) such estimates are available.
\end{remark}

\noindent {\em Acknowledgment.} \
We thank Antonio Agresti, Sonja Cox, Kristin Kirchner, Emiel Lorist, and Ivan Yaroslavtsev for helpful comments.

\newcommand{\etalchar}[1]{$^{#1}$}

\end{document}